\documentclass[11pt,amssymb]{amsart}
\usepackage{amssymb}
\usepackage[mathscr]{eucal}
\usepackage[all,cmtip]{xy}

\newcommand{\F}{{\mathbb F}}

\newcommand{\Z}{{\mathbb Z}}

\newcommand{\C}{{\mathbb C}}

\newcommand{\Q}{{\mathbb Q}}

\newcommand{\R}{{\mathbb R}}

\newcommand{\Br}{\mathrm{Br}}

\newcommand{\Ga}{\mathrm{Gal}}
\newtheorem{thm}{Theorem}[section]

\newtheorem{prop}[thm]{Proposition}
\newtheorem{cor}[thm]{Corollary}

\newcommand{\gen}{\mathbf{gen}}

\newcommand{\Ks}{K^{\rm sep}}

\newcommand{\uG}{\underline{G}}

\font\brus=wncyr10.240pk scaled 1200 .240pk
\DeclareFontFamily{U}{wncy}{}
    \DeclareFontShape{U}{wncy}{m}{n}{<->wncyr10}{}
    \DeclareSymbolFont{mcy}{U}{wncy}{m}{n}
    \DeclareMathSymbol{\Sha}{\mathord}{mcy}{"58}
\begin{document}
\title[Groups with good reduction]{Linear algebraic groups with good reduction}

\author[A.~Rapinchuk]{Andrei S. Rapinchuk}
\author[I.~Rapinchuk]{Igor A. Rapinchuk}

\address{Department of Mathematics, University of Virginia,
Charlottesville, VA 22904-4137, USA}

\email{asr3x@virginia.edu}

\address{Department of Mathematics, Michigan State University, East Lansing, MI
48824, USA}

\email{rapinchu@msu.edu}


\begin{abstract}
This article is a survey of conjectures and results on reductive algebraic groups having good reduction at a suitable set of discrete valuations of the base field. Until recently, this subject has received relatively little attention, but now it appears to be developing into one of the central topics in the emerging arithmetic theory of (linear) algebraic groups over higher-dimensional fields. The focus of this article is on
the Main Conjecture (Conjecture 5.7) asserting the finiteness of the number of isomorphism classes of forms of a given reductive group over a finitely generated field that have good reduction at a divisorial set of places of the field. Various connections between this conjecture and other problems in the theory of algebraic groups (such as the analysis of the global-to-local map in Galois cohomology, the genus problem, etc.) are discussed in detail. The article also includes a brief review of the required facts about discrete valuations, forms of algebraic groups, and Galois cohomology.
\end{abstract}

\maketitle

\section{Introduction}\label{S-Introduction}

Over the past six decades, the analysis of various properties of linear algebraic groups over local and global fields, the origins
of which can be traced to the works of Lagrange and Gauss, has developed into a well-established theory known as the {\it arithmetic
theory of algebraic groups} (cf. \cite{Pl-R}). While this subject remains an area of active research, there is growing interest in the arithmetic properties of linear algebraic groups over fields of an arithmetic nature that are not global (such as function fields of curves over various classes of fields, including $p$-adic fields and number fields). These recent developments rely on a symbiosis of methods from the theory of algebraic groups on the one hand, and arithmetic geometry on the other. At this stage, it is too early to give a comprehensive account of these new trends, so the goal of the present article is to discuss one important, and somewhat surprising, instance of the propagation of the ideas of arithmetic geometry into the theory of algebraic groups. Curiously, reduction techniques that have been used in the analysis of diophantine equations since antiquity, and the notion of good reduction, which is central to modern arithmetic geometry, were
utilized in the classical arithmetic theory of algebraic groups in a rather limited way (see the discussion in \S\ref{S-FinConjGR}).

A key novelty
in the current work is that the consideration of algebraic groups having good reduction at an appropriate set of discrete valuations of the base field has moved
to the forefront. In fact, one of the important conjectures in the area states that under suitable assumptions, the number of isomorphism classes of such
absolutely almost simple groups having a given type should be finite (see Conjecture 5.7 for the precise statement). Philosophically, this conjecture can be viewed as
an analogue of Shafarevich's conjecture, proved by Faltings \cite{Faltings}, on the finiteness of the number of isomorphism classes of abelian varieties defined over a given number field and having good reduction outside a fixed finite set of places of the field. More importantly, just like the work of Faltings is now the centerpiece of finiteness results in arithmetic geometry, this conjecture and related ones are likely to become
the cornerstone for various finiteness properties involving linear algebraic groups over higher-dimensional fields. Here we just mention that these conjectures deal with such classical aspects of the theory as the local-global principle (formulated in terms of properties of the global-to-local map in Galois cohomology --- cf. \cite{Serre-GC} and the recent survey \cite{ParimalaHasse}) as well as ways of extending the theorem on the finiteness of class numbers for groups over number fields to more general situations.
It should also be pointed out that, if proven, the finiteness conjecture for groups with good reduction would have numerous applications: we will discuss the genus problem for absolutely almost simple algebraic groups and weakly commensurable Zariski-dense subgroups of these that play a crucial role in the analysis of length-commensurable locally symmetric spaces, particularly those that are not arithmetically defined (such as, for example, nonarithmetic Riemann surfaces). We hope that this new direction of research in
the theory of algebraic groups, which brings together various topics from a number of areas and has rather unexpected applications, will be of interest to a broad audience.


The structure of the article is as follows. In \S \ref{S-GRAG}, we give a brief overview of the use of reduction techniques and the notion of good reduction in arithmetic geometry, focusing, in particular, on the (weak) Mordell theorem and Shafarevich's theorem on elliptic curves with good reduction that eventually culminated in the work of Faltings. Next, in \S \ref{S-AlgGpsGR}, we consider some motivating examples that naturally lead to the formal definition of good reduction for reductive groups. Since the formulation of our Main Conjecture relies on the notion of
forms of a given (reductive) algebraic group, we recall this in \S\ref{S-Forms}, along with the required facts about Galois cohomology.
In \S\ref{S-FinConjGR}, we first review previous work on the analysis of groups with good reduction, which focused exclusively on the case where the base field is the fraction field of a Dedekind domain and the relevant set of places consists of the discrete valuations of the field associated with maximal ideals of the ring. We then formulate the Main Conjecture (Conjecture 5.7) for groups having good reduction at a divisorial set of places of a finitely generated field. In \S\ref{S-OtherConj}, we discuss several other finiteness conjectures in this setting, which deal with the properness of the global-to-local map in Galois cohomology and finiteness conditions on the class set of an algebraic group. Available results on these conjectures are presented in \S\ref{S-Results}. We should point out that recently, all the finiteness properties discussed in this article were established for algebraic tori \cite{RR-Tori}. On the other hand, results for absolutely almost simple groups are more modest (see, in particular, \cite{CRR3}, \cite{CRR4}, \cite{CRR-Spinor}), and a great deal of work here still lies ahead. In \S\ref{S-AppGenus}, we apply these results to the genus problem, which focuses on understanding absolutely almost simple algebraic groups that have the same isomorphism classes of maximal tori over the field of definition (cf. \cite{CRR4}, \cite{CRR-Spinor}). In \S\ref{S-WCZariski}, we discuss some applications to Zariski-dense subgroups of absolutely almost simple groups related to the idea of eigenvalue
rigidity; the latter is based on the notion of weak commensurability and grew out of the investigation of isospectral and length commensurable locally symmetric spaces (cf. \cite{PrRap-WC}, \cite{R-ICM}). Finally, for completeness, in \S\ref{S-Conclusion}, we discuss two aspects of the
arithmetic theory of algebraic groups over higher-dimensional fields that are not directly related to the Main Conjecture, but which play an important role: strong approximation and rigidity.

\section{Good reduction in arithmetic geometry}\label{S-GRAG}

In order to provide some historical and philosophical context for the recent developments in the study of algebraic groups with good reduction, in this section, we give a brief overview of the use of reduction techniques and the notion of good reduction in arithmetic geometry.

To begin with, reduction techniques are among the most basic and oldest tools in number theory. Indeed, it has been known since antiquity that  they can be used to show that certain algebraic equations do not have integral or rational solutions (in other words, the corresponding algebraic varieties have no integral or rational points). For example, consider the equation
$$
x^2 - 7y^2 = -1
$$
and suppose that $(x_0, y_0)$ is an integral solution. Reducing this equation modulo 7, we obtain
$$
x_0^2 \equiv -1(\mathrm{mod}\: 7).
$$
However, there is no class modulo 7 that satisfies this condition as otherwise the multiplicative group $\left( \Z/7\Z \right)^{\times}$, which has order 6, would contain an element of order 4. This means that the original equation has no integral solutions.

At the beginning of the 20th century, it was realized that beyond simply
detecting the absence (and sometimes also the presence) of rational points, reduction techniques can be used to analyze the {\it structure} of the solution set, i.e. the set of rational points of the corresponding variety over a given field. One of the earliest, and perhaps most telling, examples of this arose in the study of elliptic curves. For simplicity, suppose that $K$ is a field of characteristic $\neq 2,3.$ We recall that an elliptic curve $E$ over $K$ is given
by an equation of the form
\begin{equation}\label{E:1}
y^2 = f(x),
\end{equation}
where $f(x) = x^3 + ax + b$ is a cubic polynomial over $K$ without multiple roots. More precisely, $E$ is the corresponding projective curve that is obtained by adding one point at infinity. It is well-known (see, for example, \cite[Ch. 1]{SilTate}) that the set of $K$-rational points $E(K)$ has the structure of an abelian group (with the group operation defined geometrically by the chord-tangent law). One of the cornerstone results in the arithmetic of elliptic curves is that when $K$ is a number field (or, more generally, a finitely generated field), then the group $E(K)$ is finitely generated. This statement, known as the {\it Mordell-Weil Theorem}, is actually true for any abelian variety (cf. \cite[Ch. VI]{LaDG}). The proof for elliptic curves over the field $K = \Q$ of rational numbers was given by Louis Mordell \cite{Mordell} in 1922. One of the key steps, which is now usually referred to as the {\it Weak Mordell Theorem}, is to show that the quotient $E(K)/ 2 E(K)$  is finite (in fact, $E(K) / nE(K)$ is finite for any $n \geq 1$); the argument here relies heavily on reduction.


For the purposes of this discussion, let us assume that the coefficients $a$ and $b$ of $f(x)$ in (\ref{E:1}) are integers, and let $p > 3$ be a prime. Reducing modulo $p$, we obtain the  equation
$$
y^2 = \bar{f}(x), \ \ \text{with} \ \ \bar{f}(x) = x^3 + \bar{a} x + \bar{b}
$$
over the finite field $\mathbb{F}_p$ with $p$ elements, where $\bar{a}$ and $\bar{b}$ denote the images of $a$ and $b$ in $\mathbb{F}_p$, respectively. If $\bar{f}$ does not have multiple roots, the reduced equation still defines an elliptic curve, in which case we say that the equation (\ref{E:1}) has {\it good reduction} at $p$. On the other hand, if $\bar{f}$ has multiple multiple roots, we say that (\ref{E:1}) has {\it bad reduction} at $p$. In this case, after reduction modulo $p$, the original elliptic curve degenerates into a singular rational curve. We note that the primes of bad reduction are precisely those that divide the discriminant $\Delta(f) =-4a^3-27b^2$ of  $f$, and therefore form a finite set. Furthermore, we say that an elliptic curve $E$ has a {\it good reduction at a prime $p > 3$} if (after a possible $\Q$-defined change of coordinates) it can be given by an equation (\ref{E:1}) that has good reduction at $p$.\footnote{More formally, let $v_p$ denote the (normalized) $p$-adic valuation on $\Q$ and let $\Z_{(p)}$ be the corresponding valuation ring. The elliptic curve $E$ is said to have good reduction at $p$ if there exists an abelian scheme $E_{(p)}$ over the valuation ring $\Z_{(p)}$ with generic fiber $E$ (the scheme $E_{(p)}$ is then unique, which leads to a well-defined notion of reduction modulo $p$).} For example, the equation $y^2 = x^3 - 625x$ has bad reduction at $p = 5$, but the elliptic curve it defines is isomorphic to the elliptic curve given by $y^2 = x^3 - x$, which has good reduction at $p = 5$. Otherwise, we say that $E$ has {\it bad reduction at $p$} (we refer the reader to \cite[Ch. VII]{Silverman} for a systematic account of these issues).

We will now sketch a proof of the Weak Mordell Theorem. Fix an elliptic curve $E$ over $\Q$ and let
$$
S = \{2,3 \} \cup \{p \mid E \ \text{has bad reduction at} \ p\}.
$$
(as noted above, $S$ is a finite set of primes).
Furthermore, for a point $R \in E(\bar{\Q})$, where $\bar{\Q}$ is a fixed algebraic closure of $\Q$, we let $\Q(R)$ denote the residue field of $R$, which is $\Q$ if $R$ is the point at infinity, and the field generated by the coordinates of $R$ for all other points. We now consider the isogeny
$$
\pi \colon E \to E, \ \ \ P \mapsto 2P.
$$
Since $\pi$ has degree 4, it follows that for any $P \in E(\Q)$ and any $R \in \pi^{-1}(P)$, the field extension $\Q(R)/\Q$ is of degree $\leq 4.$ Moreover, using the fact that $E$ has good reduction at any prime $p \notin S$, one shows that the extension $\Q(R)/\Q$ is {\it unramified} at $p$ (this part of the argument can be carried out either through the analysis of formal groups or Hensel's lemma).
We now recall that according to the classical Hermite-Minkowski theorem, $\Q$ has only finitely many extensions of a bounded degree that are unramified at all primes outside of a fixed finite set of primes (cf. \cite[Ch. III, Theorem 2.13]{Neukirch}). Applying this to the set $S$, we see that among the fields $\Q(R)$, where $R \in \pi^{-1}(P)$ and $P$ runs through $E(\Q)$, there are only {\it finitely} many distinct ones. Consequently, the field $\Q(\pi^{-1}(E(\Q)))$, which is the compositum of all such $\Q(R)$, is a {\it finite} extension of $\Q.$ One then derives the required finiteness of the quotient $E(\Q)/2E(\Q)$ by combining the Kummer sequence with the inflation-restriction sequence in Galois cohomology.
The reader who wishes to fill in the details of this sketch can consult \cite[Ch. VII and VIII]{Silverman} for a comprehensive account of the theory of elliptic curves over local and global fields.


Now, while this argument certainly demonstrates the utility of considering the places of good reduction of an elliptic curve, it still does not fully reveal their effect on the curve itself. In his 1962 ICM talk \cite{Shafarevich}, Shafarevich pointed out that if $S$ is a finite set of rational primes, then there are only finitely many isomorphism classes
of elliptic curves over $\Q$ having good reduction at all primes $p \notin S$ (in fact, he stated his theorem for an arbitrary number field $K$ and a finite set of places $S$). For the sake of completeness, we sketch an elegant proof of this theorem that appears in Serre's book \cite[IV-7]{Serre-Ladic} and which Serre attributes to Tate. The argument is based on Siegel's theorem from diophantine geometry and goes as follows. Let $A$ be the localization of $\Z$ with respect to the multiplicative set generated by $S$, which we can assume contains $2$ and $3$. First, using unique factorization in $A$, one shows that a given elliptic curve $E$ having good reduction at all $p \notin S$ is isomorphic to an elliptic curve given by (\ref{E:1}), where the coefficients $a$ and $b$ of the polynomial $f$ belong to $A$ and the discriminant $\Delta(f)$ belongs to the unit group $A^{\times}$.
Next, if we have two elliptic curves with discriminants $\Delta_1 , \Delta_2 \in A^{\times}$, and $\Delta_2 = \Delta_1 u^{12}$ for some $u \in
A^{\times}$, then one curve can be replaced by an isomorphic one, whose equation is still of the form (\ref{E:1}), so that the discriminants
actually become equal. Since the quotient $A^{\times}/{A^{\times}}^{12}$ is finite, it is enough to show that the polynomials $f$ with $a , b \in A$ and a fixed value $\Delta_0\in A^{\times}$ of the discriminant form a finite set. But the equation $\Delta(f) = \Delta_0$ can be written in the form $-27b^2 =4a^3 + \Delta_0$, which itself defines an elliptic curve, hence has only finitely many solutions in $A$ by Siegel's Theorem (cf. \cite[Ch. VII, \S\S1-2]{LaDG}).


The preceding argument is obviously very specific to elliptic curves, but Shafarevich felt that his theorem was an instance of a far more general phenomenon, which prompted him to formulate the following finiteness conjecture for abelian varieties (which are higher-dimensional analogues of elliptic curves):

\vskip2mm

\noindent {\it Let $K$ be a number field and $S$ be a finite set of primes of $K$. Then for every $g \geq 1$, there exist only finitely many $K$-isomorphism classes of abelian varieties of dimension $g$ having good reduction at all primes $\mathfrak{p} \notin S$.}

\vskip2mm

This conjecture was proved by Faltings \cite{Faltings} in 1982 as a culmination of research in diophantine geometry on finiteness properties over the course of several decades. Its numerous implications include the Mordell conjecture (a smooth projective curve of genus $\geq 2$ over a number field $K$ has only finitely many $K$-rational points) as well as Shafarevich's conjecture for curves (for any $g \geq 2$, there are only finitely many isomorphism classes of smooth projective curves over $K$ of genus $g$ having good reduction at all $\mathfrak{p} \notin S$). We refer the reader to the survey \cite{Darmon} for an account of these developments as well as a discussion of other instances of the analysis of good reduction for various classes of ``objects" over number fields. It will come as no surprise to the reader that, to this day, this subject remains one of the major themes in arithmetic geometry.



\section{Reductive algebraic groups with good reduction}\label{S-AlgGpsGR}
Having discussed the notion of good reduction in arithmetic geometry in the previous section, we will now transition to looking at good reduction in the context of linear algebraic groups. We begin this section with a series of examples that highlight several important points that arise in the consideration of reductions of algebraic groups. With these motivations in place, we will then formally define what it means for a reductive group to have good reduction with respect to a discrete valuation of the base field.

We refer the reader to Borel's book \cite{BorelAG} for a detailed account of the theory of linear algebraic groups.
For our purposes, we recall that a {\it linear algebraic group} is a subgroup $G \subset \mathrm{GL}_n (\Omega)$ that is defined by polynomial equations over $\Omega$ in terms of the matrix entries (where $\Omega$ is some algebraically closed field). Moreover, if $K$ is a subfield of $\Omega$ and the ideal of all polynomial functions over $\Omega$ that vanish on $G$ is generated by polynomials with coefficients in $K$, we say that $G$ is $K$-{\it defined} or is an {\it algebraic} $K$-{\it group}. We note that if $\mathrm{char}\:\Omega = 0$, then $G$ is $K$-defined if and only if it can be defined by polynomial equations with coefficients in $K$ (cf. \cite[AG, 12.2]{BorelAG}). For simplicity, we will consider examples over $K = \Q$, but these can be easily generalized.

\vskip3mm

\noindent {\bf Example 3.1.} Intuitively, the reduction of the general linear group $G = \mathrm{GL}_n$ modulo a prime $p$ should be the group $\mathrm{GL}_n$ over the field $\mathbb{F}_p = \Z/ p \Z.$ To justify this fact formally, one uses the following considerations. We observe that $G$
can be viewed
as a $\Z$-group scheme
$\mathrm{Spec}\: A$, where
$$
A = \Z\left[ x_{11}, \ldots , x_{nn}, \frac{1}{\det(x_{ij})} \right]
$$
in the sense that for any commutative ring $R$, one can identify the group $\mathrm{GL}_n(R)$ with
$\mathrm{Hom}_{\Z\text{-}\mathrm{alg}}(A , R)$, and this identification is natural in $R$. For any prime $p$, we can reduce $A$ modulo $p$:
$$
A_p := A \otimes_{\Z} \mathbb{F}_p = \mathbb{F}_p \left[ x_{11}, \ldots , x_{nn}, \frac{1}{\det(x_{ij})} \right].
$$
Then it is easy to see that $A_p$ represents $\mathrm{GL}_n$ over the category of $\F_p$-algebras, i.e. for any such algebra $R$, there is an identification of $\mathrm{GL}_n(R)$ with $\mathrm{Hom}_{\mathbb{F}_p\text{-}\mathrm{alg}}(A_p , R)$ that is natural in $R$. Thus, we can say that the $\Q$-group $\mathrm{GL}_n$ has a $\Z$-structure given by the algebra $A$, and that the reduction modulo $p$
of the latter represents the group $\mathrm{GL}_n$ over $\mathbb{F}_p$.


In particular, the 1-dimensional split torus $\mathbb{G}_m = \mathrm{GL}_1$ is represented by $\Z[x , x^{-1}]$; the reduction of this $\Z$-algebra modulo
$p$ is $\mathbb{F}_p[x , x^{-1}]$, which represents the 1-dimensional split torus over $\mathbb{F}_p$. More generally, the reduction modulo $p$ of the $d$-dimensional split torus $\mathbb{G}_m^d$, which is represented by the Laurent polynomial ring $$\Z[x_1, \ldots , x_d, x_1^{-1}, \ldots , x_d^{-1}],$$ is the $d$-dimensional split torus over $\mathbb{F}_p.$

\vskip3mm

\noindent {\bf Example 3.2.} The $\Q$-group $G= \mathrm{SL}_n$ is defined inside $\mathrm{GL}_n$ by the single equation $\det(x_{ij}) -1 = 0.$ Its reduction modulo $p$ give an equation of a similar shape over $\mathbb{F}_p$, suggesting that the reduction modulo $p$ should be $\mathrm{SL}_n$ over $\mathbb{F}_p$. Again, to justify this formally, we view
$G$ as an affine group scheme over $\Z$ represented by the Hopf $\Z$-algebra
$$
\Z[x_{11}, \ldots , x_{nn}]/(\det(x_{ij}) - 1).
$$
The reduction of this algebra modulo $p$ is $\mathbb{F}_p[x_{11}, \ldots , x_{nn}]/(\det(x_{ij}) - 1)$, which represents $\mathrm{SL}_n$ as a group scheme over $\mathbb{F}_p$.

\vskip3mm

\noindent {\bf Example 3.3.} Let $G$ be the special orthogonal group $\mathrm{SO}_n(q)$, where $q = x_1^2 + \cdots + x_n^2 \in \Z[x_1, \dots, x_n]$ and $n \geq 3$. Then the reduction of $G$ modulo any prime $p > 2$ is again the special orthogonal group $\mathrm{SO}_n(\bar{q})$ of the quadratic form $\bar{q} = x_1^2 + \cdots + x_n^2$ over $\mathbb{F}_p$.

\vskip3mm

The common feature in these three examples is, in essence, that reduction modulo $p$ yields an algebraic group of the {\it same type.} To be more precise, the groups in Example 3.1, i.e. $\mathrm{GL}_n$ and the split tori, are (connected and) {\it reductive}, and so are their reductions modulo all primes $p$ (in this context, connectedness is always understood in terms of the Zariski topology). The groups in Examples 3.2 and 3.3 are (connected and) {\it semi-simple}, and all of their reductions are again semi-simple.\footnote{We recall that one defines the {\it unipotent radical} of a connected algebraic group $G$ to be the largest connected unipotent normal subgroup, and one says that $G$ is {\it reductive} if its unipotent radical is trivial. For example, all {\it tori} (i.e. connected diagonalizable algebraic groups) are reductive. An algebraic group is (absolutely almost) {\it simple} if it does not contain any proper connected normal subgroups, and {\it semi-simple} if it admits a surjective morphism from a direct product of simple groups. We refer the reader to \cite{BorelAG} and \cite{CGP} for the details.}
On the other hand, the next two examples exhibit a different type of behavior.


\vskip3mm

\noindent {\bf Example 3.4.} Let $p > 2$ be a prime and consider the quadratic extension $L = \Q(\sqrt{p})$. Recall that the norm of an element $z = a + b \sqrt{p} \in L$ is given by $\mathrm{N}_{L/\Q}(z) = a^2 - pb^2$. There exists an algebraic $\Q$-group $G$ whose group of $\Q$-points $G(\Q)$ consists precisely of the elements $z \in L^{\times}$ with
$\mathrm{N}_{L/\Q}(z) = 1$. Explicitly, for any $\Q$-algebra $R$, the group $G(R)$ consists of matrices of the form $\left( \begin{array}{cc} a & pb \\ b & a \end{array}  \right)$, with $a, b \in R$, having  determinant 1.  In other words, $G$ is defined by the following equations on a $(2\times 2)$-matrix $X = (x_{ij})$:
\begin{equation}\label{E:3}
x_{11} = x_{22}, \ \ x_{12} = p x_{21}, \ \ x_{11}^2 - p x_{21}^2 = 1.
\end{equation}
We note that the matrix $\left( \begin{array}{cc} \sqrt{p} & -\sqrt{p} \\ 1 & 1 \end{array}  \right)$ conjugates $G$ (over $\overline{\Q}$) into the group of diagonal matrices $\left( \begin{array}{cc} u & 0 \\ 0 & v \end{array} \right)$ with determinant 1, which means that structurally, $G$ is a 1-dimensional ($\Q$-anisotropic) {\it torus}, usually denoted $\mathrm{R}^{(1)}_{L/\Q}(\mathbb{G}_m)$ and called the {\it norm torus} associated with the extension $L/\Q$.

Now, considering the closed subscheme of the $\mathbb{Z}$-scheme $\mathrm{GL}_2$ given by (\ref{E:3}) and reducing the defining equations modulo $p$, we obtain the equations
$$
x_{11} = x_{22}, \ \ x_{12} = 0, \ \ x_{11}^2 = 1.
$$
The solutions to these equations are matrices of the form $\pm \left( \begin{array}{cc} 1 & 0 \\ u & 1 \end{array}  \right)$. Thus, the reduced equations define an algebraic group over $\mathbb{F}_p$ which is disconnected and whose connected component is a 1-dimensional {\it unipotent} group! At the same time,
reducing the equations (\ref{E:3}) modulo any prime $q > 2$ different from $p$, we still get a 1-dimensional torus.

\vskip3mm

\noindent {\bf Example 3.5.} We now consider a noncommutative version of Example 3.4. Let $p > 2$ be a prime and $D$ be the quaternion algebra over $\Q$ corresponding to the pair $(-1 , p)$. Explicitly, $D$ is a 4-dimensional vector space over $\Q$ with basis $1, i, j, k$ and the following multiplication table:
$$
i^2 = -1, \ \ j^2 = p, \ \ k = ij = -ji.
$$
We recall that
the reduced norm of a quaternion $z = a + bi + cj + dk \in D$ is given by $\mathrm{Nrd}_{D/\Q}(z)= a^2 + b^2 -pc^2 - pd^2$. Again, there is a $\Q$-defined algebraic group $G$, which is usually denoted $\mathrm{SL}_{1 , D}$, whose group of $\Q$-points is
$$
G(\Q) = \{ z \in D^{\times} \ \vert \ \mathrm{Nrd}_{D/\Q}(z) = 1 \}.
$$
Using the regular representation of $D$, we can represent $G$ explicitly by matrices of the form
\begin{equation}\label{E:4}
\left( \begin{array}{rrrr} a & -b & pc & -pd \\ b & a & pd & -pc \\ c & d & a & -b \\ d & -c & b & a \end{array} \right) \ \ \text{such that} \ \ a^2 + b^2 -pc^2 - pd^2 = 1.
\end{equation}
One can easily find a matrix over $\Q(\sqrt{-1})$ or $\Q(\sqrt{p})$ that conjugates $G$ into matrices of the form $\left( \begin{array}{cc} A & O \\ O & A \end{array}  \right)$, with $A \in \mathrm{SL}_2$, which means that over these extensions (and hence over $\bar{\Q}$), this group is isomorphic to $\mathrm{SL}_2$. In other words, $\mathrm{SL}_{1 , D}$ is a $\Q$-form of $\mathrm{SL}_2$, hence, in particular, a simple algebraic group (see \S\ref{S-Forms} for a more detailed discussion of forms).

Taking the obvious linear equations defining the matrices (\ref{E:4}) and reducing them modulo $p$,
we obtain a system that defines the following group
$$
\left( \begin{array}{rrrr} a & -b & 0 & 0 \\ b & a & 0 & 0 \\ c & d & a & -b \\ d & -c & b & a \end{array}   \right) \ \ \text{with} \ \  a^2 + b^2 = 1.
$$
The matrices of the form $\left( \begin{array}{cc} I_2 & O \\ * & I_2 \end{array} \right)$ in this group form a connected unipotent normal subgroup, which shows that the reduction in this case is {\it not} a reductive group.


\vskip3mm

We see that the reductive $\Q$-groups in Examples 3.1-3.3 admit a structure of a group scheme over $\mathbb{Z}$ (or, more generally, over $\mathbb{Z}_{(p)}$ --- the localization of $\mathbb{Z}$ at the prime ideal $(p)$), which is described by a system of polynomial equations with coefficients in the respective ring such that the reductions of these equations modulo $p$ still define a reductive group.
On the other hand, in Examples 3.4 and 3.5, we have situations where a reductive group
is given by a system of polynomials with coefficients in $\Z$ (or $\Z_{(p)}$) such that the reduced system no longer defines a reductive group.

So, in analogy with the case of elliptic curves, we are naturally led to the following definition for $\Q$-groups: we say that a reductive $\Q$-group $G$ has {\it good reduction} at a prime $p$ if it can be defined by a system of polynomial equations
with coefficients in $\Z_{(p)}$\footnote{In more technical terms, this system defines a scheme over $\Z_{(p)}$ with generic fiber $G$.} such that the reduction of the system modulo $p$ still defines a reductive group; otherwise we say that $G$ has {\it bad reduction at $p$}. Thus, the groups in Examples 3.1-3.3 {\it do} have good reduction at the specified primes. On the other hand, with some additional work, one can show that for any prime $p > 2$, the group in Example 3.4 has bad reduction at $p$ and good reduction at every odd prime $q \neq p$, and
the group in Example 3.5 has bad reduction at $p$ whenever $p \equiv 3 \pmod{4}$. (We note that for $p \equiv 1 \pmod{4}$, the group $\mathrm{SL}_{1 , D}$ in Example 3.5 is in fact isomorphic to $\mathrm{SL}_2$, hence has good reduction at $p$. So, a group defined by a system whose reduction is not reductive may still have good reduction at $p$ since it may be possible to describe the group by another system whose reduction {\it does} yield a reductive group.)


\vskip3mm

Now, for (reductive) linear algebraic groups defined over general fields, one considers the notion of good reduction with respect to {\it discrete valuations}. We recall that a (normalized) discrete valuation on a field $K$ is a surjective map $v \colon K^{\times} \to \mathbb{Z}$ such that

\vskip2mm

(a) $v(ab) = v(a) + v(b)$;

\vskip1mm

(b) $v(a + b) \geq \min(v(a) , v(b))$ whenever $a + b \neq 0$.

\vskip2mm

\noindent Then
$$
\mathcal{O}(v) := \{ a \in K^{\times} \ \vert \ v(a) \geqslant 0 \} \cup \{ 0 \}
$$
is a subring of $K$ called the {\it valuation ring} of $v$. It is a local ring with the maximal ideal
$$
\mathfrak{p}(v) = \{ a \in K^{\times} \ \vert \ v(a) > 0 \} \cup \{ 0 \},
$$
which is called the {\it valuation ideal} of $v$.
An element $\pi \in K^{\times}$ with $v(\pi) = 1$ is called a {\it uniformizer}. It is easy to see that $\pi$ generates $\mathfrak{p}(v)$, and in fact {\it every} nonzero ideal of $\mathcal{O}(v)$ is of the form $(\pi^k)$ for some $k \geq 0.$ In other words, $\mathcal{O}(v)$ is a {\it discrete valuation ring} (DVR) --- we refer the reader to \cite[Ch. 9]{At} or \cite[Ch. VI, \S3, n$^{\circ}$ 6]{Bour-CA} for a discussion of various equivalent definitions of a DVR.
Next, the quotient $\mathcal{O}(v)/\mathfrak{p}(v)$ is called the {\it residue field} of $v$, and will be denoted by $K^{(v)}$. Furthermore, the function
$\vert \cdot \vert_v \colon K \to \R$ defined by
$$
\vert a \vert_v = \rho^{v(a)} \ \ \ \ \text{for} \ a \in K^{\times} \ \ \ \ \text{and} \ \ \ \ \vert 0 \vert_v = 0,
$$
where $\rho$ is a fixed real number with $0 < \rho < 1$, is a ({\it non-archimedean} or {\it ultrametric}) {\it absolute value} on $K$. We will write $K_v$ for the completion of $K$ with respect to the metric associated with $\vert \cdot \vert_v$. One shows that $v$ naturally extends to a discrete valuation on $K_v$; for ease of notation, we will denote this extension simply by $v$. The corresponding valuation ring and valuation ideal in $K_v$ will be denoted $\mathcal{O}_v$ and $\mathfrak{p}_v$, respectively; we note that the quotient $\mathcal{O}_v/\mathfrak{p}_v$ coincides with the residue field $K^{(v)}$ defined above. We also set $U_v = \mathcal{O}_v^{\times}$ to be the group of units in $\mathcal{O}_v$.

\vskip3mm

\noindent {\bf Example 3.6.} Here are several useful examples of discrete valuations.

\vskip2mm

\noindent (a) To every rational prime $p$, there corresponds the (normalized) $p$-adic valuation $v_p$ on $\mathbb{Q}$ defined as follows: if $a \in \mathbb{Q}^{\times}$ is of the form $\displaystyle a = p^{\alpha} \cdot \frac{m}{n}$ with $m$ and $n$ relatively prime to $p$, then $v(a) = \alpha$. The corresponding valuation ring is the localization $\mathbb{Z}_{(p)}$ considered earlier, and the residue field is $\mathbb{F}_p$. Furthermore, the completion is the field of $p$-adic numbers $\mathbb{Q}_p$, and the valuation ring of $\Q_p$ is the ring of $p$-adic integers $\mathbb{Z}_p$.

\vskip2mm

\noindent (b) Let $K = k(x)$ be the field of rational functions in one variable over a field $k$, and let $p(x) \in k[x]$ be a (monic) irreducible polynomial. Then the same construction as in part (a) enables us to associate to $p(x)$ a discrete valuation $v_{p(x)}$ on $K$. There is one additional discrete valuation on $K$ given by
$$
v_{\infty}\left(\frac{f}{g}\right) = \deg(g) - \deg(f), \ \ \text{where} \ \ f , g \in k[x].
$$
Note that all of these valuations are trivial on the field of constants $k$ (i.e. $v(a) = 0$ for all $a \in k$), and cumulatively they constitute all valuations of $K$ with this property. These valuations are often called ``geometric" since they naturally correspond to the closed points of the projective line $\mathbb{P}^1_k$.

\vskip2mm

\noindent (c) Let again $K = k(x)$, but now assume that we are given a discrete valuation $v_0$ of $k$. Then $v_0$ can be extended to a discrete valuation $v$ on $K$ by first extending it to the polynomial ring $k[x]$ using the formula
$$
\text{for} \ \ f(x) = a_n x^n + \cdots + a_0, \ \ v(f(x)) := \min_{i = 0, \ldots , n} v(a_i),
$$
and then extending $v$ to $K$ by multiplicativity
$$
v\left(\frac{f(x)}{g(x)}\right) = v(f(x)) - v(g(x)).
$$
This extension is often called ``Gaussian" (see \cite[Ch. VI, \S 10, n$^{\circ}$ 1]{Bour-CA} for further details). In particular, we can start with the $p$-adic valuation $v_p$ on $\Q$ and then, using this procedure, extend it to a valuation $v$ on $\Q(x).$


\vskip3mm

We are now ready to give a formal definition of good reduction for reductive groups. It requires language from the theory of reductive group schemes, for which we refer the reader to \cite{Con} or \cite{DemGr} (see also \cite{DemGab} or \cite{Waterhouse} for general introductions to group schemes). However, it turns out that in a number of particular cases, some of which will be discussed below, good reduction can be characterized in very concrete terms. We note that, in a broad sense, the theory of affine groups schemes shifts attention from concrete matrix groups to the Hopf algebras that represent these groups and on which one can consider various ``integral" structures.
For our purposes, this change in perspective provides an indispensable framework for discussing the reduction of algebraic groups (as we have already seen in Examples 3.1 and 3.2).

In our discussion, we will use the following standard notation: given an affine scheme $X = \mathrm{Spec}\: A$, where $A$ is a commutative algebra over a commutative ring $R$, and a ring extension $R \subset R'$, we will denote by
$X \times_R R'$ the affine scheme $\mathrm{Spec}\: (A \otimes_R R')$ over $R'$ (usually referred to as ``base change").

\vskip2mm

\noindent {\bf Definition 3.7.} {\it Let $K$ be a field equipped with a discrete valuation $v$ (with corresponding completion $K_v$ and valuation ring $\mathcal{O}_v \subset K_v$), and let $G$ be a connected reductive $K$-group. We say that $G$ has \emph{good reduction} at $v$ if there exists a reductive group scheme $\mathscr{G}$ over $\mathcal{O}_v$ with generic fiber $\mathscr{G} \times_{\mathcal{O}_v} K_v$ isomorphic to $G \times_K K_v$.}

\vskip3mm

We recall that a smooth $\mathcal{O}_v$-group scheme $\mathscr{G}$ is called {\it reductive} if for every $x \in {\rm Spec}(\mathcal{O}_v)$, the geometric fiber $\mathscr{G} \times_{\mathcal{O}_v} \overline{\kappa(x)}$ (where $\overline{\kappa(x)}$ is an algebraic closure of the residue field $\kappa(x)$) is a (connected) reductive algebraic group over $\overline{\kappa(x)}.$ We also note that the $\mathcal{O}_v$-scheme $\mathscr{G}$ with generic fiber $G \times_K K_v$ is unique (cf. \cite[Th\'eor\`eme 5.1]{Nisn1} for semi-simple groups as well as \cite[\S6]{Guo} for the general case); in particular, this means that reduction (or special fiber)
$$\uG^{(v)} := \mathscr{G} \times_{\mathcal{O}_v} K^{(v)}$$ is well-defined. Furthermore, when $G$ is an absolutely almost simple $K$-group, the reduction $\uG^{(v)}$ has the same Cartan-Killing type as $G$ (see \cite[Exp. XXII, Proposition 2.8]{DemGr}). It should be observed that the definition of good reduction is sometimes given in terms of reductive $\mathcal{O}(v)$-schemes instead of $\mathcal{O}_v$-schemes. This makes essentially no difference, but the above definition is more convenient for our purposes, particularly for discussing connections of good reduction with local-global principles (see \S\ref{S-OtherConj}).

\vskip3mm

We conclude this section with several examples of semi-simple groups with good reduction that will be sufficient for understanding the rest of the article.


\vskip2mm

\noindent {\bf Example 3.8.}

\vskip2mm

\noindent (a) An absolutely almost simple simply connected $K$-split group $G$ has good reduction at any $v$. This follows from the Chevalley construction, which provides a $\mathbb{Z}$-structure for $G$, described in detail in \cite{BorelChevalley}.
(This generalizes the above discussion of $\mathrm{SL}_n$ in Example 3.2.)

\vskip2mm

\noindent (b) The algebraic $K$-group $G = \mathrm{SL}_{1 , A}$ associated with the group of elements of reduced norm 1 in a central simple $K$-algebra $A$ (which generalizes the group $\mathrm{SL}_{1 , D}$ considered in Example 3.5)
has good reduction at $v$ if and only if $A \otimes_K K_v$ comes from an Azumaya algebra $\mathscr{A}$ over $\mathcal{O}_v$ (which means, in particular, that the reduction $\mathscr{A} \otimes_{\mathcal{O}_v} K^{(v)}$ is a central simple algebra over the residue field $K^{(v)}$). Another way of putting this is to say that $A$ is {\it unramified} at $v$ --- see the discussion at the beginning of \S\ref{S-ResultsMain} below.

\vskip2mm

\noindent (c) Assuming that $\mathrm{char}\: K^{(v)} \neq 2$, the spinor group $G = \mathrm{Spin}_n(q)$ of a nondegenerate quadratic form $q$ in $n > 2$ variables over $K$ has good reduction at $v$ if and only if, over $K_v$, the form $q$ is equivalent to a quadratic form
$$
\lambda (a_1 x_1^2 + \cdots + a_n x_n^2) \ \ \text{with} \ \ \lambda \in K_v^{\times} \ \ \text{and} \ \ a_i \in U_v \ \ \text{for all} \ \ i = 1, \ldots , n.
$$

\section{Forms with good reduction and Galois cohomology}\label{S-Forms}

The purpose of this section is to recall some key points concerning forms of algebraic groups and (nonabelian) Galois cohomology. These are needed, on the one hand, in order to formulate an appropriate analogue of Shafarevich's conjecture for reductive algebraic groups (see Question 4.3 and Conjecture 5.7), and, on the other hand, are indispensable for the discussion of local-global principles in \S\ref{S-OtherConj}.

\subsection{$F/K$-forms.}
As we saw in \S\ref{S-GRAG}, Shafarevich's conjecture asserts that the number of $K$-isomorphism classes of abelian varieties of a {\it given dimension} over a number field $K$ that have good reduction outside a given finite set of places of $K$ is finite. So, as a straightforward analogue of this conjecture for reductive linear algebraic groups, one can consider the following:

\vskip2mm

\begin{center}

\parbox[t]{15cm}{\it Let $K$ be a field equipped with a set $V$ of discrete valuations. What are the reductive algebraic groups of a given dimension that have good reduction at all $v \in V$? More specifically, what are the situations in which the number of $K$-isomorphism classes of such groups is finite?}

\end{center}

\vskip2mm

\noindent Of course, to make these questions meaningful, one needs to specialize $K$ and $V$, which we will do in \S\ref{S-FinConjGR}. However, we would first like to point out that in the case of reductive algebraic groups, considering reductive algebraic groups of a given dimension is far less natural than considering abelian varieties of the same dimension. Indeed, in a very coarse sense, all complex abelian varieties of dimension $d$ ``look the same": they are all analytically isomorphic to complex tori $\mathbb{C}^d / \Lambda,$ where $\Lambda \subset \mathbb{C}^d$ is a lattice of rank $2d$ (see, for example, \cite[Chapter 1]{MumfordAV}). At the same time, the ``fine" structure and classification of these varieties are highly involved as these varieties have nontrivial {\it moduli spaces}, which leads to infinite continuous families of nonisomorphic varieties. As a simple example, we recall that the isomorphism class of a complex elliptic curve is uniquely determined by its $j$-invariant, which can be any complex number (in fact, isomorphism classes of elliptic curves are classified by the $j$-invariant over {\it arbitrary} algebraically closed fields --- see \cite[Ch. III, Proposition 1.4]{Silverman}).

On the contrary, reductive algebraic groups of the same dimension may look completely different. For example, the absolutely almost simple group $\mathrm{SL}_2$ is 3-dimensional, as is the torus $(\mathbb{G}_m)^3$, which is a reductive group that is not semi-simple. Furthermore, the absolutely almost simple simply connected group of type $\textsf{B}_3$, which can be realized as the spinor group $\mathrm{Spin}_7$, has the same dimension as the product of $7$ copies of $\mathrm{SL}_2$, which is a semi-simple but not absolutely almost simple group. A fundamental result in the structure theory of reductive groups is that over an algebraically (or separably) closed field, for each integer $d \geq 1$, there are only {\it finitely} many isomorphism classes of (connected) reductive groups of dimension $d$ (thus, there are no nontrivial moduli spaces for such groups --- see, for example, \cite[Theorem 1.3.15]{Con} for the precise statement). So, in the analysis of finiteness phenomena for reductive groups, it makes sense to focus on those classes of groups that becomes isomorphic over a separable closure of the base field; the main issue then becomes the passage from an isomorphism over the separable closure to an isomorphism over the base field (so-called {\it Galois descent}). This brings us to the following.

\vskip2mm

\noindent {\bf Definition 4.1.} {\it Let $G$ be a linear algebraic group over a field $K$, and let $F/K$ be a field extension. An algebraic $K$-group $G'$ is called an $F/K$-\emph{form} of $G$ if there exists an $F$-isomorphism
$$
G \times_K F \simeq G' \times_K F
$$
(where $G \times_K F$ denote the algebraic $F$-group obtained by base change from $K$ to $F$).}

\vskip2mm


Although we will restrict ourselves primarily to forms of algebraic groups in the present article, we should point out that the consideration of forms of various other ``algebraic objects" comes up naturally in many different contexts --- we refer the reader to \cite[Ch. III, \S1]{Serre-GC} for a general discussion of forms as well as a number of concrete examples. For our purposes, we will be interested mostly in $K^{\rm sep}/K$-forms of a linear algebraic group $G$, where $K^{\rm sep}$ is a fixed separable closure of $K$; these will often be referred to simply as {\it $K$-forms} of $G$ \footnote{For comparison, we would like to point out the following finiteness theorem for the forms of abelian varieties (cf. \cite[\S3, Finiteness theorem for forms]{Zar-Par}): {\it Let $X$ be an
abelian variety over a field $K$, and let $F/K$ be a finite separable extension. Then the set of $K$-isomorphism classes of abelian $K$-varieties $X'$ such that there exists an $F$-isomorphism $X \times_K F \simeq X' \times_K F$ is finite.} On the contrary, for a semi-simple linear algebraic
$K$-group $G$, and a finite separable extension $F/K$, the set of $K$-isomorphism classes of $F/K$-forms $G'$ of $G$ is infinite in many cases, even when $K$ is a number field (see, however, the discussion of fields of type (F) in \S\ref{S-FConjCurves}). So, the problem of classifying forms of (semi-simple) linear algebraic groups with special properties, which is central to the current article, comprises some
challenges that do not arise in the context of abelian varieties.}.


Let us now look at several illustrative examples of forms of algebraic groups, which will suffice for understanding most of the article.



\vskip2mm

\noindent {\bf Example 4.2.} We let $K$ be a field and fix a separable closure $K^{\rm sep}$ of $K$.

\vskip2mm

\noindent (a) Let $T = (\mathbb{G}_m)^d$ be a $d$-dimensional $K$-split torus. Any other $d$-dimensional $K$-torus $T'$ splits over $\Ks$, i.e. we have a $\Ks$-isomorphism
$$
T' \times_K \Ks \simeq T \times_K \Ks.
$$
This means that all $d$-dimensional $K$-tori are $K$-forms of $T$.

\vskip2mm

\noindent (b) Similarly, let $G$ be an absolutely almost simple simply connected {\it split} algebraic group over $K$, and let $G'$ be any absolutely almost simple simply connected $K$-group of the same (Killing-Cartan) type as $G$. Again, $G'$ splits over $\Ks$, hence there is a $\Ks$-isomorphism
\begin{equation}\label{E:XXX1}
G \times_K \Ks \simeq G' \times_K \Ks.
\end{equation}
This means that $G'$ is a $K$-form of $G$. Thus, every absolutely almost simple simply connected $K$-group is a $K$-form of an absolutely almost simple simply connected split $K$-group. For later use, let us also mention the following variant of this statement. First, we recall that a $K$-group $G$ is said to be {\it $K$-quasi-split} (or simply {\it quasi-split} if there is no risk of confusion) if it contains a Borel subgroup defined over $K$. Then one shows that every absolutely almost simple simply connected $K$-group is an {\it inner} form of an absolutely almost simple simply connected $K$-quasi-split group (see Example 4.5 below for a brief discussion of these matters).

\vskip2mm

\noindent (c) Let $A$ be a central simple $K$-algebra of degree $n$, and let $G = \mathrm{SL}_{1 , A}$ be the algebraic $K$-group associated with the group of elements of reduced norm 1 in $A$. Since $A \otimes_K \Ks$ is isomorphic to the matrix algebra $\mathrm{M}_n(\Ks)$, the group $G \times_K \Ks$ is $\Ks$-isomorphic to $\mathrm{SL}_n$. In other words, $G$ is a $K$-form of $\mathrm{SL}_n$. In fact, it is an inner form, and all inner forms are obtained this way --- see \cite[Ch. II, \S2.3.4]{Pl-R} for the details.

\vskip2mm

\noindent (d) Assume that ${\rm char}~K \neq 2$. Let $q$ be a nondegenerate quadratic form in $n$ variables over $K$, and let
$G = \mathrm{Spin}_n(q)$ be the corresponding spinor group (which is an absolutely almost simple simply connected $K$-group for $n \geq 3$, $n \neq 4$). Then any other nondegenerate quadratic form $q'$ over $K$ in $n$ variable is equivalent to $q$ over $\Ks$, implying that for $G' = \mathrm{Spin}_n(q')$, there is a $\Ks$-isomorphism similar to (\ref{E:XXX1}). Thus, $G'$ is a $K$-form of $G$. If $n$ is odd, then the groups $G' = \mathrm{Spin}_n(q')$ account for all $K$-forms of $G$. For $n$ even however, there may be other $K$-forms of $G$ defined in terms of the (universal covers of) the unitary groups of (skew-)hermitian forms over noncommutative central division $K$-algebras with an involution of the first kind (such as, for example, quaternion algebras). For the details on this, as well as a discussion of the $K$-forms of other groups of classical types, we refer the reader to \cite[Ch. II, \S2.3.4]{Pl-R}.


\vskip3mm

We will close this subsection with the following adjusted version of our original question concerning reductive groups with good reduction.

\vskip2mm

\noindent {\bf Question 4.3.} {\it  Let $K$ be a field equipped with a set $V$ of discrete valuations. What are the $K$-forms (or even just inner $K$-forms) of a given reductive algebraic $K$-group $G$ that have good reduction at all $v \in V$? More specifically, in what situations is the number of $K$-isomorphism classes of such $K$-forms finite?}

\vskip2mm

\noindent (In fact, it follows from Example 4.1(b) that when $G$ is an absolutely almost simple simply connected algebraic $K$-group, we can assume that it is quasi-split over $K$).

\vskip2mm

We will return to this question in \S\ref{S-FinConjGR} below. First, however, we would like to review some of the main points of (nonabelian) Galois cohomology and its connection to forms of algebraic groups.



\vskip2mm

\subsection{A brief review of Galois cohomology}\label{S-ReviewGC} Since Galois cohomology will play a significant role in subsequent sections, we will now quickly review some of the basic notions, focusing particularly on nonabelian aspects (for further details, the reader can consult various sources, including 
\cite[Ch. 2]{Pl-R} and \cite[Ch. I, \S5]{Serre-GC}).


First, suppose a group $\Gamma$ acts by automorphisms on another group $A$ (which may be non-commutative).
We then define the $0$-th (noncommutative) cohomology by $H^0(\Gamma , A) = A^{\Gamma}$ (the subgroup of $\Gamma$-fixed elements). Furthermore, a (noncommutative) {\it 1-cocycle} is a map $f \colon \Gamma \to A$ satisfying
$$
f(\sigma \tau) = f(\sigma) \cdot \sigma(f(\tau)) \ \ \text{for all} \ \ \sigma , \tau \in \Gamma.
$$
It should be noted that while the set $Z^1(\Gamma, A)$ of 1-cocycles is of course an abelian group when $A$ is abelian, it may {\it fail} to be a group in the general case; instead, it will be treated simply as a {\it pointed set}, whose distinguished element is the trivial cocyle. Next, two
cocycles $f , g \in Z^1(\Gamma , A)$ are said to be {\it equivalent} if there exists an element $a \in A$ such that
$$
g(\sigma) = a^{-1} \cdot f(\sigma) \cdot \sigma(a) \ \ \text{for all} \ \ \sigma \in \Gamma.
$$
One easily checks that this gives an equivalence relation on $Z^1(\Gamma , A)$, and we define $H^1(\Gamma , A)$ to be the quotient of
$Z^1(\Gamma , A)$ by this relation. Again, in the general case, $H^1(\Gamma, A)$ will be viewed as a pointed set whose distinguished element is the equivalence class of the trivial cocycle.


As in the classical (abelian) situation, to every short exact sequence
\begin{equation}\label{E:YYY1}
1 \to A \longrightarrow B \longrightarrow C \to 1
\end{equation}
of $\Gamma$-groups and $\Gamma$-homomorphisms (i.e., group homomorphisms commuting with the $\Gamma$-action),
one can associate a long exact sequence
\begin{equation}\label{E:YYY2}
1 \to H^0(\Gamma , A) \longrightarrow H^0(\Gamma , B) \longrightarrow H^0(\Gamma , C) \longrightarrow H^1(\Gamma , A) \longrightarrow H^1(\Gamma , B) \longrightarrow H^1(\Gamma , C)
\end{equation}
of cohomology sets --- the exactness of the second half of the sequence is understood in terms of pointed sets, where the kernel is defined as the pre-image of the distinguished element. (Unfortunately, in general, this sequence cannot be extended beyond degree 1 since $H^i(\Gamma, A)$ is undefined for $i \geq 2$ if $A$ is noncommutative; however, as described in \cite[Ch. I, \S5, Proposition 43]{Serre-GC}, if $A$ is a {\it central} subgroup of $B$, then there is a natural map $H^1(G,C) \to H^2(G,A)$ that extends (\ref{E:YYY2})).

We also note that the standard functoriality properties of cohomology remain valid in the noncommutative situation. More precisely, let $A$ be a $\Gamma$-group and $B$ be a $\Delta$-group, and suppose we are given group homomorphisms
\begin{equation}\label{E:XXX2}
f \colon B \to A \ \ \text{and} \ \ \varphi \colon \Gamma \to \Delta \ \ \text{such that} \ \ f(\varphi(\sigma)(b)) = \sigma(f(b)) \ \ \text{for all} \ \ b \in B, \ \sigma \in \Gamma.
\end{equation}
Then there is a map of pointed sets $H^1(\Delta , B) \to H^1(\Gamma , A)$, which, on the level of cocycles, is defined as follows: the image of $g \in Z^1(\Delta , B)$ is $g^* \in Z^1(\Gamma , A)$ given by
$$
g^*(\sigma) = f(g(\varphi(\sigma))) \ \ \text{for all} \ \ \sigma \in \Gamma.
$$
In particular, given a $\Gamma$-group $A$, for any subgroup $\Delta \subset \Gamma$, we can apply this construction by taking $f$ to be the identity map $A \to A$ and $\varphi$ to be the natural embedding (in which case (\ref{E:XXX2}) obviously holds) to define the {\it restriction map} $H^1(\Gamma , A) \to H^1(\Delta , A)$. Next, given a normal subgroup $\Phi \unlhd\Gamma$, we can consider $A^{\Phi}$ as a group with the natural action of $\Delta = \Gamma/\Phi$. We can then apply the above construction by taking $f$ to be the natural embedding $A^{\Phi} \hookrightarrow A$ and $\varphi$ to be the canonical map $\Delta \to \Delta/\Phi$, to obtain the {\it inflation map} $H^1(\Gamma/\Phi , A^{\Phi}) \to H^1(\Gamma , A)$.

In order to transition to the Galois cohomology of algebraic groups, we need to specialize the preceding
definitions to one situation involving topological groups. More precisely, we will now assume that $\Gamma$ is a {\it profinite} group (i.e. a  compact and totally disconnected topological group) that {\it acts continuously} on a {\it discrete} group $A$ (the continuity assumption amounts to the requirement that for each $a \in A$, its stabilizer $\Gamma(a)$ is an open subgroup of $\Gamma$). We then consider {\it continuous} 1-cocycles $f \colon \Gamma \to A$. Since we will deal exclusively with continuous cocycles in this context, we will keep the notation
$Z^1(\Gamma , A)$ for the set of all continuous 1-cocycles. It is easy to see that an abstract cocycle that is equivalent to a continuous cocycle is itself continuous, and we again define $H^1(\Gamma , A)$ to be the pointed set of equivalence classes of continuous 1-cocycles. The basic properties of abstract noncommutative cohomology remain valid in this setting. In particular, to every exact sequence (\ref{E:YYY1}) of discrete groups with continuous $\Gamma$-action and $\Gamma$-homomorphisms, there corresponds the long exact sequence (\ref{E:YYY2}) of (continuous) cohomology. Furthermore, for every closed subgroup $\Delta \subset \Gamma$ and any discrete group $A$ with a continuous $\Gamma$-action, we have the restriction map $H^1(\Gamma , A) \to H^1(\Delta , A)$, and for every closed normal subgroup $\Phi \unlhd \Gamma$, we have the inflation map $H^1(\Gamma/\Phi, A^{\Phi}) \to H^1(\Gamma , A)$.

Let us now briefly indicate how the continuous cohomology of profinite groups is related to the cohomology of finite groups (see, for example, \cite[Ch. II]{Shatz} for the details). For a profinite group $\Gamma$, we let $\mathcal{N}$ denote the family of all open normal subgroups $N \subset \Gamma$. Then we have the identification
$$
\Gamma \simeq \lim_{\longleftarrow} \Gamma / N,
$$
where the inverse limit is
taken over all $N \in \mathcal{N}$ with respect to the canonical homomorphisms $\Gamma / N_1 \to \Gamma / N_2$ for $N_1 , N_2 \in \mathcal{N}$ with $N_1 \subset N_2$. Then for any discrete group $A$ with a continuous $\Gamma$-action, $H^1(\Gamma , A)$ defined as above in terms of continuous cocycles is naturally identified with the direct limit
$$
\lim_{\longrightarrow} H^1(\Gamma/N , A^N),
$$
taken over all $N \in \mathcal{N}$ with respect to the inflation maps $H^1(\Gamma/N_2 , A^{N_2}) \to H^1(\Gamma/N_1 , A^{N_1})$ for $N_1 \subset N_2$.

We will apply this general set-up only in situations where the profinite group $\Gamma$ is the Galois group $\mathrm{Gal}(L/K)$ of a (possibly infinite) Galois extension $L/K$; in fact, the most important case for us will be where $\Gamma = \Ga(\Ks/K)$ is the Galois group of a separable closure of $K$, i.e. the {\it absolute Galois group} of $K$. Then a discrete group $A$ with a continuous $\mathrm{Gal}(L/K)$-action will be called a {\it Galois $L/K$-module} (note that, as above, we allow $A$ to be noncommutative). The cohomology set $H^1(\mathrm{Gal}(L/K) , A)$ will be denoted simply by $H^1(L/K , A)$, which will be shortened to $H^1(K , A)$ if $L = \Ks$. For any subextension $K \subset M \subset L$, the Galois group $\mathrm{Gal}(L/M)$ is a closed subgroup of $\mathrm{Gal}(L/K)$, so for any
Galois $L/K$-module $A$, we have the restriction map $H^1(L/K , A) \to H^1(L/M , A)$. In particular, for any subextension $M$ of $\Ks$, we have the restriction map $H^1(K , A) \to H^1(M , A)$ (where $A$ is a Galois $K$-module, i.e. a $\Ks/K$-module).

Here is another important example of a restriction map. Let $K$ be a field equipped with a (discrete) valuation $v$, and let $K_v$ be the corresponding completion. It is well-known (cf. \cite[Ch. VII, Proposition 1.2]{ANT}) that the absolute Galois group $\mathrm{Gal}(\Ks_v/K_v)$ can be identified with the {\it decomposition group}
$D(\bar{v})$ of a fixed extension $\bar{v}$ of $v$ to $\Ks$, which is a closed subgroup of $\mathrm{Gal}(\Ks/K)$. This gives rise to the restriction map $H^1(K , A) \to H^1(K_v , A)$. Moreover, in a certain obvious sense, this map is independent of the choice of the extension $\bar{v}$ (see \cite[Ch. VII, Remark 2.4]{MilneCFT} for the details); in particular, its kernel is well-defined.

Now, given an algebraic $K$-group $G$, for any Galois extension $L/K$, the group of $L$-points $G(L)$ is naturally a Galois $L/K$-module. We then have the corresponding cohomology set $H^1(L/K, G(L))$, which is usually denoted $H^1(L/K, G)$; when $L = \Ks$, we will simply write $H^1(K, G).$
If $K$ is equipped with a (discrete) valuation $v$, then composing the restriction map discussed above with the map induced by the natural embedding $G(K) \hookrightarrow G(K_v)$, we obtain a map $H^1(K , G) \to H^1(K_v , G)$. By abuse of terminology, we will refer to this map, which will play an important role in \S\ref{S-OtherConj}, also as a restriction map.


\vskip2mm



\subsection{Forms and Galois cohomology}\label{S-FormsGC} To wrap up this section, we will now briefly review the description of $F/K$-forms in terms of Galois cohomology. As before, we will focus mostly on forms of algebraic groups, but in fact the description remains valid in many other situations, and we will mention some relevant examples. For further details, the reader can consult \cite[Ch. VII]{KMRT}, \cite[Ch. 2]{Pl-R}, or \cite[Ch. III, \S1]{Serre-GC}, among other sources.


Let $G$ be an algebraic $K$-group and $F/K$ be a Galois extension. Note that although the group $A_F$ of $F$-defined automorphisms of $G$ (or rather of $G \times_K F$) may not naturally be the group of $F$-points of an algebraic $K$-group, it is always a Galois $F/K$-module. Next, let $G'$ be an $F/K$-form of $G$. By definition, there exists an $F$-defined isomorphism $t \colon G \to G'$. Since the groups $G$ and $G'$ are defined over $K$, the Galois group $\mathrm{Gal}(F/K)$ acts on the set of $F$-isomorphisms between $G$ and $G'$. So, for every $\sigma \in \mathrm{Gal}(F/K)$, we can consider the isomorphism $\sigma(t) \colon G \to G'$. Then
$$
f(\sigma) := t^{-1} \circ \sigma(t)
$$
belongs to $A_F$, and the mapping $(\sigma \mapsto f(\sigma))$ defines a 1-cocycle in $Z^1(F/K , A_F)$, which we will denote simply by $f$. Moreover, one shows that the equivalence class of $f$ is independent of the choice of the isomorphism $t$. Then it turns out that sending $G'$ to the equivalence class of $f$ in $H^1(F/K, A_F)$ defines a bijection between the set of $K$-isomorphism classes of $F/K$-forms of $G$ and $H^1(F/K , A_F)$. In particular, there is a natural bijection between the set of $K$-isomorphism classes of $K$-forms of $G$ and $H^1(K , A_{\Ks})$.

Here are several explicit examples of this classification.


\vskip2mm

\noindent {\bf Example 4.4.} Let $T = (\mathbb{G}_m)^n$ be an $n$-dimensional $K$-split torus. Then the corresponding automorphism group $A_{\Ks}$ is $\mathrm{GL}_n(\mathbb{Z})$ equipped with the trivial action of $\mathrm{Gal}(\Ks/K)$. Since, according to Example 4.2(a), any $n$-dimensional $K$-torus is a $K$-form of $T$, we see that the $K$-isomorphism classes of $n$-dimensional $K$-tori are in bijection with the equivalence classes (in this case, conjugacy classes)
of continuous homomorphisms $f \colon \mathrm{Gal}(\Ks/K) \to \mathrm{GL}_n(\mathbb{Z})$. Let us now take an $n$-dimensional $K$-torus $T'$ and consider the corresponding homomorphism $f$. Then $N := \ker f$ is an open normal subgroup of $\mathrm{Gal}(\Ks/K)$, so the corresponding fixed subfield $L =
(\Ks)^N$ is a finite Galois extension of $K$ with Galois group $\mathscr{G} = \mathrm{Gal}(\Ks/K)/N$. In fact, $L$ is the {\it minimal splitting field} of $T'$, i.e. the minimal Galois extension of $K$ over which $T'$ becomes isomorphic to the split torus $T$ and it is uniquely determined by this property. We conclude that the $K$-isomorphism classes of $n$-dimensional tori with the minimal splitting field $L$ correspond bijectively to equivalence classes of $n$-dimensional faithful
representations $\mathscr{G} \to \mathrm{GL}_n(\mathbb{Z})$.

\vskip3mm

\noindent {\bf Example 4.5.} Let $G$ be an absolutely almost simple simply connected $K$-group. Then the automorphism group $A_{\Ks}$ fits into an exact sequence
\begin{equation}\label{E:ZZZ1}
1 \to I_{\Ks} \longrightarrow A_{\Ks} \longrightarrow S_{\Ks} \to 1
\end{equation}
of Galois $K$-modules, where $I_{\Ks}$ is the group of inner automorphisms, which can be identified with the group of $\Ks$-points $\overline{G}(\Ks)$ of the corresponding adjoint group $\overline{G}$, and $S_{\Ks}$ is the group of symmetries of the Dynkin diagram of $G$ (see \cite[Ch. 2, \S2.1.13]{Pl-R} for the details). Next, let $G'$ be another absolutely almost simple simply connected algebraic group of the same type as $G$. Then both $G$ and $G'$ become split, hence isomorphic, over $\Ks$. Thus, $G'$ is a $K$-form of $G$, implying that $H^1(K , A_{\Ks})$ in this case classifies the $K$-isomorphism classes of all absolutely almost simple simply connected $K$-groups of the same type as $G$. Associated to (\ref{E:ZZZ1}), we have the following exact sequence
of Galois cohomology
$$
H^1(K , I_{\Ks}) \stackrel{\alpha}{\longrightarrow} H^1(K , A_{\Ks}) \stackrel{\beta}{\longrightarrow} H^1(K , S_{\Ks}).
$$
Then the $K$-forms of $G$ that correspond to the cohomology classes in the image of $\alpha$ are called {\it inner}, while all other forms are called {\it outer} (note that $\alpha$ may not be injective!).


For concreteness, let us take $G = \mathrm{SL}_n$. Then the corresponding adjoint group is $\overline{G} = \mathrm{PGL}_n$, and it follows from the Skolem-Noether theorem that the elements of $H^1(K , \overline{G})$ are in one-to-one correspondence with the isomorphism classes of central simple $K$-algebras of degree $n$. Then if a cohomology class $\xi \in H^1(K , \overline{G})$ corresponds to an algebra $A$, the image $\alpha(\xi)$ corresponds to the norm 1 group $\mathrm{SL}_{1 , A}$ introduced in Example 4.1(c). We note that the norm 1 groups corresponding to an algebra $A$ and its opposite algebra $A^{\small \rm{opp}}$ are $K$-isomorphic, which means that the corresponding cohomology classes in $H^1(K , \overline{G})$, which are generally distinct, have the same image under $\alpha$!

Next, if $G$ is $K$-split, then it is known that the Galois action on $S_{\Ks}$ is trivial (so, the elements of $H^1(K , S_{\Ks})$ correspond to conjugacy classes of continuous homomorphisms $\mathrm{Gal}(\Ks/K) \to S_{\Ks}$), and the sequence (\ref{E:ZZZ1}) has a $\Ga(\Ks/K)$-equivariant splitting.
Let $\iota \colon H^1(K , S_{\Ks}) \to H^1(K , A_{\Ks})$ be the corresponding splitting for $\beta$. Then for $\xi \in H^1(K , S_{\Ks})$, the image $\iota(\xi)$ corresponds to a {\it quasi-split} group $G(\xi)$ (see, for example, \cite[Lemma 16.4.8]{SpringerAG}), and the inner forms of $G(\xi)$ correspond precisely to the elements of the fiber $\beta^{-1}(\xi)$. Thus, every form is an inner form of some quasi-split group (note that the quasi-split group is unique up to isomorphism).

Continuing with the assumption that $G$ is $K$-split, let us consider a $K$-form $G'$ of $G$. Let $c(G') \in H^1(K, A_{\Ks})$ be the corresponding cohomology class and let $d(G') \in H^1(K, S_{\Ks})$ be the image of $c(G')$ under $\beta$. Then $d(G')$ is represented by a continuous homomorphism $\Ga(\Ks/K) \to S_{\Ks}$ called the $*$-action associated with $G'$ (cf. \cite{TitsClassification}).
Thus, $G'$ is an inner form of the
split group $G$ if and only if the $*$-action is trivial.
Furthermore, if $\mathcal{H}$ denotes the kernel of the $*$-action, then the fixed field $L = (\Ks)^{\mathcal{H}}$ is the (uniquely defined) minimal Galois extension of $K$ over which $G'$ becomes an inner form of the split group $G$.
We also note that if $G'$ is $K$-quasi-split, then $L$ is the minimal
Galois extension of $K$ over which $G'$ splits.
We refer the reader to \cite[Lemma 4.1]{PrRap-WC} for a detailed discussion of these matters.

\vskip3mm

In our final example, we will briefly demonstrate how the above method for classifying the $K$-forms of algebraic groups can be used to classify (the isomorphism classes of) some other objects.

\vskip3mm

\noindent {\bf Example 4.6.} Fix a nondegenerate quadratic form in $n$ variables over a field $K$ of characteristic $\neq 2$, and let $G = \mathrm{O}_n(q)$ be the corresponding orthogonal group. It is well-known that any other nondegenerate quadratic form $q'$  in $n$ variables over $K$ becomes equivalent to $q$ over $\Ks$. Thus, if $Q$ and $Q'$ are the matrices of $q$ and $q'$, then there exists $X \in \mathrm{GL}_n(\Ks)$ such that $Q = X^t Q' X$. For each $\sigma \in \mathrm{Gal}(\Ks/K)$, the element
$$
f(\sigma) : = X^{-1} \cdot \sigma(X)
$$
belongs to $G(\Ks)$, and the mapping $(\sigma \mapsto f(\sigma))$ defines a 1-cocycle $f \colon \mathrm{Gal}(\Ks/K) \to G(\Ks)$. One then shows that associating to $q'$ the cohomology class of $f$ sets up a bijection between the set of equivalence classes of nondegenerate quadratic forms in $n$ variables over $K$ and $H^1(K , G)$. In fact, this easily follows from the fact that $H^1(K , \mathrm{GL}_n) = 1$ (the noncommutative version of Hilbert's Theorem 90). Furthermore, for $H = \mathrm{SO}_n(q)$, the elements of $H^1(K , H)$ correspond to the equivalence classes of nondegenerate $n$-dimensional quadratic forms that have the same discriminant as $q$ (see \cite[Ch. 2, \S2.2.2]{Pl-R} for the details).

\section{The finiteness conjecture for reductive groups with good reduction}\label{S-FinConjGR}

Having reviewed the necessary material on Galois cohomology, we now return to Question 4.3, our central question concerning forms with good reduction. In our discussion, we will deal with two natural choices for the field $K$ and the set of discrete valuations $V$ of $K$: we will first consider the case where $K$ is the field of fractions of a Dedekind ring $R$ and $V$ is the set of discrete valuations associated with the maximal ideals of $R$, after which we will turn to finitely generated fields $K$ equipped with a divisorial set of places $V$.

\subsection{The Dedekind case.}\label{S-Dedekind} Early interest in this case can be traced back to the work of G.~Harder \cite{Harder} and J.-L.~Colliot-Th\'el\`ene and  J.-J.~Sansuc \cite{CTS}. For example, the cohomology set introduced in \cite[Lemma 4.1.3]{Harder} (basically) yields, as a particular case, the set of $K$-isomorphism classes of $K$-forms of a given semi-simple group that have good reduction at all the relevant valuations. In \cite{CTS}, the authors define and analyze similar unramified cohomology sets in connection with the study of torsors under reductive group schemes over low-dimensional base schemes. The basic case where $K = \mathbb{Q}$ and $R = \mathbb{Z}$, and hence $V$ is the set of all $p$-adic valuations of $\Q$, was considered by B.H.~Gross \cite{Gross} and B.~Conrad \cite{Conrad}. One of the key observations in \cite{Gross} is the following.

\begin{thm}\label{T:XXX1}
{\rm (\cite[Proposition 1.1]{Gross})} Let $G$ be an absolutely almost simple simply connected algebraic group over $\mathbb{Q}$. Then $G$ has a good reduction at all primes $p$ if and only if $G$ is split over all $\mathbb{Q}_p$.
\end{thm}

Note that the fact that an absolutely almost simple simply connected {\it split} group $G$ over $\mathbb{Q}$
has good reduction at all primes is a particular case of Example 3.8(a). (For comparison, we recall that there are no abelian varieties over $\mathbb{Q}$ that have good reduction at all primes --- see  \cite{Abr} and \cite{Fon}.) As we will see below, the proof of Theorem \ref{T:XXX1} yields the fact that
the $\mathbb{Q}$-forms $G'$ of the $\mathbb{Q}$-split group $G$ that have good reduction at all primes are necessarily {\it inner}.  Furthermore, using deeper results on Galois cohomology, such as the Hasse principle for adjoint semi-simple groups over number fields (see \cite[\S 6.5]{Pl-R}, and particularly Theorem 6.22 therein), one concludes that such forms are uniquely determined by their isomorphism class over $\mathbb{R}$, opening
thereby a way to their classification up to $\mathbb{Q}$-isomorphism (and of course implying the finiteness of the number of such classes).
In fact, this analysis is taken much further in \cite{Gross} and \cite{Conrad} by explicitly constructing such nonsplit forms for certain types in terms of $\mathbb{Z}$-lattices and even considering the problem of the classification of such lattices. We note that the choice of a $\mathbb{Z}$-lattice determines the structure of a group scheme over $\mathbb{Z}$ on the corresponding $\mathbb{Q}$-algebraic group, so the classification of such lattices really amounts to the classification of forms with good reduction at all primes up to $\mathbb{Z}$-isomorphism rather than just $\mathbb{Q}$-isomorphism, which of course is a much harder problem (to get a sense of the difficulty, the reader may want to consider the problem of classifying unimodular integral quadratic forms up to equivalence versus rational equivalence --- cf. \cite[Ch. V]{Serre-Course}).
It turns out that in some cases, the required lattices can in fact be classified, but in other cases one shows using the mass formula that such lattices are so numerous that no reasonable classification appears possible.

We will now briefly sketch a proof of Theorem \ref{T:XXX1}, mainly to indicate the ingredients that go into the argument. One of the key facts that is used is the following very explicit local description of groups with good reduction in the situation at hand (we refer the reader to Example 4.5 for the relevant terminology). Namely, suppose $\mathscr{G}$ is an absolutely almost simple simply connected algebraic group over a field $\mathscr{K}$ that is a finite extension of $\Q_p$. Then $\mathscr{G}$ has good reduction if and only if it is quasi-split over $\mathscr{K}$ and the minimal Galois extension $\mathscr{L}$ of $\mathscr{K}$ over which it becomes an inner form of the split group is unramified over $\mathscr{K}$ --- cf. \cite[Corollary 5.2.14]{Con}. Now, let $G$ be an absolutely almost simple simply connected $\mathbb{Q}$-group that has good reduction at all primes, and let $L$ be the minimal Galois extension of $\mathbb{Q}$ over which $G$ becomes an inner form of the split group. Applying the above local description of groups with good reduction, we conclude that $L/\Q$ is unramified at all primes. On the other hand, a well-known consequence of Minkowski's estimate for the discriminant of a number field is that $\mathbb{Q}$ does not have nontrivial extensions with this property. Thus, $L = \Q.$ It follows that for every prime $p$, the group $G$ is quasi-split over $\mathbb{Q}_p$ and at the same time is an inner form over $\Q_p$, implying that it actually splits over $\mathbb{Q}_p$. Note that since $L = \Q$, our argument also shows that $G$ is an inner form of the split group over $\Q$.

We should point out that Theorem \ref{T:XXX1} easily extends to any number field $K$ and any set $V$ of discrete valuations of $K$ if we limit ourselves to absolutely almost simple simply connected algebraic groups of those types for which the Dynkin diagram does not have nontrivial symmetries (recall that these types are $\textsf{A}_1$, $\textsf{B}_{\ell}$, $\textsf{C}_{\ell}$, $\textsf{E}_7$, $\textsf{E}_8$, $\textsf{F}_4$, and $\textsf{G}_2$). However, if we consider groups of other types over a number field that has an everywhere unramified quadratic extension, the above argument and the result itself become invalid even if we take $V$ to be the set of all nonarchimedean valuations of $K$. Nevertheless, we have the following finiteness result (cf. \cite{JL} --- we refer the reader to this paper for analogues of Shafarevich's conjecture for several classes of varieties). In the statement below, we denote by $V_f^K$ the set of all nonarchimedean places of a number field $K$.

\begin{prop}\label{P:XXX1}
Let $G$ be an absolutely almost simple simply connected algebraic group over a number field $K$, and let $V \subset V_f^K$ be a set of discrete valuations of $K$ such that $V_f^K \setminus V$ is finite. Then the number of $K$-isomorphism classes of $K$-forms of $G$ that have good reduction at all $v \in V$ is finite.
\end{prop}
\begin{proof}

First, we recall that according to the Hermite-Minkowski theorem, for every integer $d \geq 1$, the number field $K$ has only finitely many extensions of degree $\leq d$ that are unramified at all $v \in V$. Now let $S$ be the group of symmetries of the Dynkin diagram of $G$. Then the Hermite-Minkowski theorem implies that there are only finitely many continuous homomorphisms
$$
\omega \colon \Ga(\Ks/K) \to S
$$
such that the fixed field $L(\omega) = (\Ks)^{\ker \omega}$ of $\ker \omega$ is unramified at all $v \in V$. Let $\omega_1, \ldots , \omega_r$ be representatives of the conjugacy classes (in $S$) of such homomorphisms, which we view as elements of $H^1(K,S).$ As we discussed in Example 4.5,
for each $i = 1, \ldots , r$, there is a quasi-split $K$-form $G_i$ of $G$ for which $L(\omega_i)$ is the minimal Galois extension of $K$ over which $G_i$ becomes split. Suppose now that $G'$ is a $K$-form of $G$ that has good reduction at all $v \in V$, and let $L$ be the minimal Galois extension of $K$ over which $G'$ becomes an inner form of the split group. Then, as discussed in Example 4.5,
$L$ is the fixed subfield of the kernel of the natural homomorphism $\omega \colon \mathrm{Gal}(\Ks/K) \to S$ given by the $*$-action. Besides, it follows from \cite[Corollary 5.2.14]{Con} that $L/K$ is unramified at all $v \in V$. By our construction, this means that  $\omega$ coincides with one of the $\omega_i$, in which case $G'$ is an inner form of $G_i$. Thus, it is enough to prove, for each $i = 1, \dots, r$, the finiteness of the number of isomorphism classes of those $K$-forms of $G$ that have good reduction at all $v \in V$ and are {\it inner} forms of $G_i$. By definition (see Example 4.5), any such $G'$ corresponds to an element $\xi \in H^1(K , \overline{G_i})$. Furthermore, as we already saw in our discussion of Theorem \ref{T:XXX1}, for each $v \in V$, the group $G'$ becomes quasi-split over $K_v$, hence $K_v$-isomorphic to $G_i$. In the language of Galois cohomology, this means that $\xi$ lies in the kernel of the restriction map $H^1(K , \overline{G_i}) \to H^1(K_v , \overline{G_i})$. Since this is true for all $v \in V$, we conclude that $\xi$ is contained in the kernel of the product map
$$
\theta_{G_i, V} \colon H^1(K , \overline{G_i}) \longrightarrow \prod_{v \in V} H^1(K_v, \overline{G_i}).
$$
However, it is well-known that over number fields, the map $\theta_{G_i , V}$ is {\it proper} (i.e. the pre-image of a finite set is finite), and hence, in particular, its kernel $\text{{\brus SH}}(\overline{G_i} , V)$ is finite --- see \cite[\S 5]{Borel} and \cite[\S 7]{BorelSerre} (in the next section we will discuss a conjectural extension of this property to a more general situation). So, the number of possible $\xi$ arising in this set-up is finite, and the finiteness of the number of $K$-isomorphism classes of such $K$-forms $G'$ follows.
\end{proof}

\vskip2mm

Another basic example of a Dedekind ring is the polynomial ring $R = k[x]$ over a field $k$. We let $V$ denote the set of discrete valuations of the field of rational functions $K = k(x)$ associated with monic irreducible polynomials $p(x) \in k[x]$ (see Example 3.6(b)). We then have the following.
\begin{thm}\label{T:XXX2}
{\rm (cf. \cite[Theorem 1.1]{RagRam} and \cite{GilleAffine})} Let $G_0$ be a (connected) semi-simple simply connected algebraic group over a field $k$. If $G'$ is a $K$-form of the group $G = G_0 \times_k K$ that has good reduction at all $v \in V$ and splits over $k^{\rm sep}(x)$, then $G' = G'_0 \times_k K$ for some $k$-form $G'_0$ of $G_0$.
\end{thm}

The original result in \cite{RagRam} is more general and is formulated in terms of torsors (principal homogeneous spaces): {\it Let $G_0$ be a connected reductive algebraic group over a field $k$, and let $\pi \colon B \to \mathbb{A}^1_k$ be a $G_0$-torsor over the affine line $\mathbb{A}^1_k = \mathrm{Spec}\: k[x]$. If $\pi$ is trivialized by the base change from $k$ to $k^{\rm sep}$, then $\pi$ is obtained by the base change $\mathbb{A}^1_k \to \mathrm{Spec}\: k$ from a $G_0$-torsor $\pi_0 \colon B_0 \to \mathrm{Spec}\: k$.} To derive Theorem \ref{T:XXX2} from this statement, one argues as follows. Since, by our assumption, $G'$ splits over $k^{\rm sep}(x)$, the homomorphism $\mathrm{Gal}(\Ks/K) \to S$ to the group of symmetries of the Dynkin diagram that corresponds to the $*$-action associated with $G'$ (see Example 4.5) factors through a homomorphism $\mathrm{Gal}(k^{\rm sep}/k) \to S$. Let $G_0^{\sharp}$ be the quasi-split $k$-form of $G_0$ corresponding to this homomorphism, and set $G^{\sharp} = G_0^{\sharp} \times_k K$. Then $G'$ is an inner $K$-form of $G^{\sharp}$, so it corresponds to a cohomology class $\xi \in H^1(K , \overline{G}^{\sharp})$ with values in the corresponding adjoint group. Since $G'$ has good reduction at all $v \in V$, it follows that this cocycle gives rise to a torsor of $\overline{G}_0^{\sharp}$ over $\mathbb{A}^1_k$; in other words, $\xi$ is the image of some
$\zeta \in H^1_{_{\mathrm{\acute{e}t}}}(k[x] , \overline{G}_0^{\sharp})$ under
the natural map
$$
H^1_{_{\mathrm{\acute{e}t}}}(k[x] , \overline{G}_0^{\sharp}) \to H^1(K , \overline{G}^{\sharp})
$$
(where $H^1_{_{\mathrm{\acute{e}t}}}(k[x] , \overline{G}_0^{\sharp})$ denotes the \'etale cohomology set over ${\rm Spec}(k[x])$) --- cf. \cite[Corollary A.8]{GillePian}. But since $G'$ splits over $k^{\rm sep}$, the cocycle $\xi$ lies in the kernel of the restriction map $H^1(K , \overline{G}^{\sharp}) \to H^1(k^{\rm sep} K , \overline{G}^{\sharp})$. Then $\zeta$ lies in the kernel of the map $H^1_{et}(k[x] , \overline{G}_0^{\sharp}) \to H^1_{et}(k^{\rm sep}[x] , \overline{G}_0^{\sharp})$, hence by the result from \cite{RagRam} is the image of some $\omega \in H^1(k , \overline{G}_0^{\sharp})$ under the map $H^1(k , \overline{G}_0^{\sharp}) \to H^1_{_{\mathrm{\acute{e}t}}}(k[x] , \overline{G}_0^{\sharp})$. Then the $k$-form $G'_0$ of $G_0$ corresponding to $\omega$ is as required.

\vskip2mm

We note that the assumption that $G'$ splits over $k^{\rm sep}(x)$ holds automatically when $k$ has characteristic zero. Indeed, in this case, the separable closure $k^{\rm sep}$ coincides with the algebraic closure $\overline{k}$, so the field $k^{\rm sep}(x) = \overline{k}(x)$ has cohomological dimension $\leq 1$ by Tsen's theorem (cf. \cite[Ch. II, \S3]{Serre-GC}). Hence, applying Steinberg's theorem, we conclude that $G'$ is quasi-split over $\overline{k}(x)$ (see \cite[Ch. III, \S2.3]{Serre-GC}).
On the other hand,  since $G'$ has good reduction at all $v \in V$, the minimal Galois extension $L/K$ over which $G'$ becomes an inner form of the split group is unramified at all $v \in V$. Using the Riemann-Hurwitz formula, one concludes that $\overline{k}L = \overline{k}(x)$, implying that $G'$ is actually split over $\overline{k}(x)$.


\vskip2mm

Yet another case where forms with good reduction have been considered in full involves the ring of Laurent polynomials $R = k[x , x^{-1}]$. Here, the set $V$ consists of the discrete valuations of $K = k(x)$ corresponding to monic irreducible polynomials $p(x) \in k[x]$, with $p(x) \neq x$.

\begin{thm}\label{T:XXX3}
{\rm (cf. \cite[Theorem 2.5]{ChGP})} Let $G_0$ be a (connected) semi-simple simply connected algebraic group over $k$. Assume that $\mathrm{char}\: k$ is prime to the order
of the Weyl group of $G_0$. Then there is a natural bijection
between the isomorphism classes of inner $K$-forms $G'$ of the group $G = G_0 \times_k K$ that have good reduction at all $v \in V$ and the elements of the Galois cohomology set $H^1(k(\!(x)\!) , \overline{G})$ of the corresponding adjoint group over the field of Laurent series.
\end{thm}

Again, this easily follows from the following more general result proved in \cite{ChGP}: {\it Let $G_0$ be a connected reductive group over a field $k$. Assume that
the characteristic of $k$ is good\footnotemark. Then there is a natural bijection between the isomorphism classes of $G_0$-torsors over the punctured affine
line $\mathbb{A}^{\times}_k = \mathbb{A}^1_k \setminus \{ 0 \}$ and the elements of $H^1(k(\!(x)\!) , G)$.} We refer the reader to \cite[\S5]{ChGP} for the details concerning this result, and only mention here that it plays a crucial role in the proof of the conjugacy of the analogues of Cartan subalgebras in certain infinite-dimensional Lie algebras \cite{CNPY16}.

\footnotetext{As defined in \cite[\S5]{CGP}.}

\subsection{The finiteness conjecture for function fields of curves.}\label{S-FConjCurves} An important class of examples of Dedekind rings consists of the rings of regular functions $R = k[C]$, where $C$ is a smooth geometrically integral affine curve over a field $k$. In view of the bijection between the maximal ideals of $R$ and the closed points of $C$, the corresponding $V$ is then the set of
discrete valuations of the function field $K = k(C)$ associated with the closed points of $C$.  Unfortunately, in most cases, no explicit description of the $K$-forms of a given reductive $K$-group $G$ that have good reduction at all $v \in V$, similar to Theorems \ref{T:XXX2} and \ref{T:XXX3}, seems to be available. A more tractable problem in this setting appears to be the qualitative question about the conditions that ensure the finiteness of the number of isomorphism classes of such forms. We observe that if $G_0$ is a reductive algebraic $k$-group and $G = G_0 \times_k K$ is the base change of $G_0$ to $K$, then for any $k$-form $G'_0$ of $G_0$ the group $G' = G'_0 \times_k K$ is a $K$-form of $G$ that has good reduction at all $v \in V$. So, even though non-isomorphic $k$-forms may become isomorphic after base change to $K$, in order to have the affirmative answer to the above question in a sufficiently general situation, one needs to assume that over $k$, the groups at hand have only finitely many non-isomorphic forms. This basically amounts to a hypothesis on the finiteness of Galois cohomology.


In \cite[Ch. III, \S\S4.1-4.3]{Serre-GC}, Serre described a class of fields over which one does have the required finiteness by
introducing the following condition on a profinite group $\mathscr{G}$:

\vskip2mm

\noindent (F) \parbox[t]{14.6cm}{\it For every integer $m \ge 1$, the group $\mathscr{G}$ has only finitely many open subgroups of index $m$}

\vskip2mm

\noindent (such profinite groups are sometimes called ``small"). He then defined a field $K$ to be of type (F) if it is {\it perfect} and its absolute Galois group $\mathrm{Gal}(\Ks/K)$ satisfies (F). The key result is that if $K$ is a field of type (F), then the set $H^1(K,G)$ is finite for any linear algebraic $K$-group $G$ (see \cite[Ch. III, \S4.3, Theorem 4]{Serre-GC}). Moreover, if $K$ is of characteristic 0 and of type (F), then any linear algebraic group has finitely many $K$-forms (see \cite[Ch. III, \S4.3, Remark 1]{Serre-GC} and \cite[\S6]{BorelSerre}).

Recently, in \cite{IRap}, the second-named author proposed a generalization of condition (F) that holds in some situations where (F) fails and is still sufficient to establish certain finiteness results. For this, let $K$ be a field, and $m \geq 1$ be an integer prime to $\mathrm{char}\: K$. We say that $K$ is of type $(\mathrm{F}'_m)$ if

\vskip2mm

\noindent $(\mathrm{F}'_m)$ \parbox[t]{14.5cm}{\it For every finite separable extension $L/K$, the quotient $L^{\times}/{L^{\times}}^m$ of
the multiplicative group $L^{\times}$, is finite.}

\vskip2mm

\noindent If $K$ is of type (F), then it satisfies $(\mathrm{F}'_m)$ for all $m$ prime to $\mathrm{char}\: K$ --- we refer the reader to \cite{IRap} for a proof of this fact as well as a discussion of various examples of fields of type $(\mathrm{F}'_m)$ and the associated finiteness results.
It is reasonable to expect that given an absolutely almost simple algebraic $K$-group $G$ whose Weyl group has order $w$, the fact that $\mathrm{char}\: K$ is prime to $w$ and $K$  satisfies condition $(\mathrm{F}'_p)$ for all prime divisors of $w$ should imply the finiteness of $H^1(K , G)$; however, this has not yet been established. Along these lines, we would like to propose the following finiteness conjecture for forms with good reduction.

\vskip2mm

\noindent {\bf Conjecture 5.5.} {\it Let $K = k(C)$ be the function field of a smooth geometrically integral affine curve $C$ over a field $k$,
and let $V$ be the set of discrete valuations associated with the closed points of $C$. Consider an absolutely almost simple simply connected algebraic $K$-group $G$, and let $w$ be the order of the Weyl group of $G$. Assume that $\mathrm{char}\: k$ is prime to $w$ and that the field $k$ satisfies $(\mathrm{F}'_p)$ for all prime divisors $p$ of $w$. Then the number of isomorphism classes of $K$-forms $G'$ of $G$ that have good reduction at all $v \in V$ is finite.}

\vskip2mm

As we will see in \S\ref{S-Results}, for certain types, already condition $(\mathrm{F}'_2)$ for $k$ implies the finiteness of the number of isomorphism classes of forms with good reduction, which suggest that the assumptions on $k$ in the above conjecture could probably be weakened.

\subsection{The finiteness conjecture for arbitrary finitely generated fields.}\label{S-ConjFGfields} We now turn to the main finiteness conjecture for forms with good reduction over arbitrary finitely generated fields (see Conjecture 5.7 below). Let us point out that, unlike the Dedekind situation discussed in the previous subsections, forms with good reduction have never been previously considered in the higher-dimensional setting.

We begin with a description of our general set-up. Let $K$ be an arbitrary {\it finitely generated field}, i.e. a field that can be generated over its prime subfield by finitely many elements. Then $K$ possesses natural sets of discrete valuations called {\it divisorial}. More precisely, let $\mathfrak{X}$ be a {\it model} of $K$, i.e. a normal separated irreducible scheme of finite type over $\Z$ (if $\mathrm{char}\: K = 0$) or over a finite field (if $\mathrm{char}\: K > 0$) such that $K$ is the field of rational functions on $\mathfrak{X}$. It is well-known that to every prime divisor $\mathfrak{Z}$ of $\mathfrak{X}$, there corresponds a discrete valuation $v_{\mathfrak{Z}}$ on $K$ (cf. \cite[12.3]{Cut}, \cite[Ch. II, \S 6]{Hart}). Then
$$
V(\mathfrak{X}) = \{ v_{\mathfrak{Z}} \ \vert \ \mathfrak{Z} \ \ \text{prime divisor of} \ \ \mathfrak{X} \}
$$
is called the divisorial set of places of $K$ corresponding to the model $\mathfrak{X}$. Any set of places $V$ of $K$ of this form (for some model $\mathfrak{X}$) will be simply called {\it divisorial}. In terms of commutative algebra, this construction amounts to finding a subring $R$ of $K$ whose fraction field is $K$ and such that $R$ is integrally closed (in $K$) and a finitely generated $\mathbb{Z}$-algebra (or $\mathbb{F}_p$-algebra). For any height one prime ideal $\mathfrak{p} \subset R$ (i.e. a minimal nonzero prime ideal), the localization $R_{\mathfrak{p}}$ is a discrete valuation ring, hence gives rise to a discrete valuation $v_{\mathfrak{p}}$ of $K$. Then $V$ consists of these valuations $v_{\mathfrak{p}}$ corresponding to all height one prime ideals of $R$. We note that given a divisorial set  of places $V$ associated with a model  $\mathfrak{X}$, one can consider an affine open subscheme $\mathfrak{U} = \mathrm{Spec}\: R$ of $\mathfrak{X}$. Then for the divisorial set $V'$ associated with $\mathfrak{U}$, we have $V' \subset V$ and $V \setminus V'$ is finite. This means that, in most situations, it is enough to consider divisorial set associated with {\it affine} models, where, as we have just seen, they can be described in terms of height one prime ideals.

Furthermore, we observe that any two divisorial sets $V_1$ and $V_2$ associated with two different models of $K$ are {\it commensurable}, i.e. $V_i \setminus (V_1 \cap V_2)$ is finite for $i = 1, 2$ (this makes a divisorial set of places almost canonical), and that for any finite subset $S$ of a divisorial set $V$, the set $V \setminus S$ contains a divisorial set. For the sake of completeness, let us briefly justify these properties. We first note that any divisorial set of places $V$ of a finitely generated field $K$ satisfies the following condition

\vskip2mm

\noindent (A) For any $a \in K^{\times}$, the set $V(a) := \{ v \in V
\, \vert \, v(a) \neq 0 \}$ is finite.

\vskip2mm

\noindent Now, in order to show that two divisorial sets $V_1$ and $V_2$ are commensurable, we can assume, without loss of generality, that $V_i$ is associated with a model $\mathfrak{X}_i = \mathrm{Spec}\:R_i$, where $R_i$ is a finitely generated $\mathbb{Z}$-algebra with
fraction field $K$, for $i = 1, 2$. The finite generation of $R_i$ implies that there exist nonzero elements $a_1 \in R_1$ and $a_2 \in R_2$ such
that $R_2 \subset R_1[1/a_1]$ and $R_1 \subset R_2[1/a_2]$. Then $V_1\setminus V_1(a_1) \subset V_2$ and $V_2 \setminus V_2(a_2) \setminus
V_1$. Consequently,
$$
V_1 \setminus (V_1 \cap V_2) \subset V_1(a_1) \ \ \ \text{and} \ \ \ V_2
\setminus (V_1 \cap V_2) \subset V_2(a_2),
$$
and the commensurability of $V_1$ and $V_2$ follows. Furthermore, given a divisorial set $V$ associated with $\mathfrak{X} = \mathrm{Spec}\: R$ and a finite subset $S \subset V$, using weak approximation for $K$ we can find a nonzero $a \in R$ such that $v(a)> 0$ for all $v \in S$. Then the divisorial set $V_a$ associated with $\mathfrak{X}_a := \mathrm{Spec}\: R[1/a]$ satisfies $V_a \subset V\setminus S$, as required.

\vskip3mm

\noindent {\bf Example 5.6.} The field of rational functions $K = \mathbb{Q}(x)$ is the fraction field of $R = \mathbb{Z}[x]$ (which, of course, is an integrally closed finitely generated $\mathbb{Z}$-algebra). It is well-known that any height one prime ideal $\mathfrak{p}$ of $R$ is  principal with a generator $p$ of one of the following two types: (a) $p \in \mathbb{Z}[x]$ is an irreducible polynomial with content 1; or (b) $p \in \mathbb{Z}$ is a rational prime. In the first case,
the associated discrete valuation of $K$ coincides with the one corresponding to the monic irreducible polynomial in $\mathbb{Q}[x]$ obtained by dividing $p$ by its leading coefficient (cf. Example 3.6). We will call such discrete valuations {\it geometric} and denote the set of all these valuations by $V_0$. In the second case, the associated discrete valuation coincides with the Gaussian extension of the $p$-adic valuation of $\mathbb{Q}$. We will refer to such valuations as {\it arithmetic} and denote the set of all these valuations by $V_1$. Thus, the set of divisorial valuations of $K$ associated with the model $\mathfrak{X} = \mathrm{Spec}\: R$ is $V = V_0 \cup V_1$.

\vskip3mm

We are now in a position to formulate our central conjecture for forms with good reduction over arbitrary finitely generated fields.

\vskip2mm

\noindent {\bf Conjecture 5.7. (Main Conjecture for forms with good reduction)} {\it Let $G$ be a connected reductive algebraic group over a finitely generated field $K$, and $V$ be a divisorial set of places of $K$. Then the set of $K$-isomorphism classes of (inner) $K$-forms $G'$ of $G$ that have good reduction at all $v \in V$ is finite (at least when the characteristic of $K$ is ``good'').}

\vskip2mm

\noindent (When $G$ is an absolutely almost simple algebraic group, we say that $\mathrm{char}\: K = p$ is ``good" for $G$ if either $p = 0$ or $p > 0$ and does not divide the order of the Weyl group of $G$. For non-semisimple reductive groups, only characteristic $0$ will be considered good.)

\vskip2mm

To conclude this section, we would now like to make a few
brief comments on the use of the word ``main" in the above designation. As we will see in \S\ref{S-OtherConj}, this conjecture has links to several other finiteness conjectures for reductive algebraic groups over finitely generated fields. In fact, the Main Conjecture, on the one hand, automatically implies the truth of some of these conjectures in certain cases (cf. the discussion following Conjecture 6.1), and on the other hand, some of the other conjectures are likely to provide tools for attacking the Main Conjecture. Second, the Main Conjecture has important applications to the genus problem (cf. \S\ref{S-AppGenus}) and to the analysis of weakly commensurable Zariski-dense subgroups of absolutely almost simple algebraic groups and the related concept of eigenvalue rigidity (cf. \S\ref{S-WCZariski}). We observe that these developments go back to geometric problems involving length-commensurable locally symmetric spaces \cite{PrRap-WC} and can be viewed as a rather inspiring  example of the application of techniques from
number theory and arithmetic geometry  to differential geometry. It should also be noted that we were led to the Main Conjecture by our earlier work on these applications, which indicated the necessity of considering the above statement in a higher-dimensional setting. To the best of our knowledge, this point of view has never come up before in the context of algebraic groups.

As a final remark, we note that one can also consider the original conjecture of Shafarevich for abelian varieties (as well as some other classes of varieties) over higher-dimensional fields. Let us first observe that, in the case of elliptic curves, the argument of Serre-Tate that we sketched at the end of \S\ref{S-GRAG} yields, with minimal changes,  the following statement:

\vskip2mm

\noindent {\it If $K$ is a finitely generated field of characteristic zero and $V$ is a divisorial set of places of $K$, then the set of
$K$-isomorphism classes of elliptic curves over $K$ that have good
reduction at all $v \in V$ is finite.}

\vskip2mm

\noindent Important results in this direction, particularly over function fields of characteristic $p > 0$, are due to Parshin \cite{Par-Izv}, Zarhin \cite{Zarhin}, Szpiro (\cite{Szpiro1} and \cite{Szpiro2}), and Moret-Bailly \cite{Moret-Bailly}. We also refer to the reader to \cite{Voloch} for a survey of results in diophantine geometry in positive characteristic.


\section{Some other finiteness conjectures}\label{S-OtherConj}

While Conjecture 5.7 on forms of algebraic groups with good reduction appears to be most important (and in line with results on abelian varieties), it should really be viewed as part of a ``package" of several conjectures on finiteness properties of linear algebraic groups over higher-dimensional fields.
One of these finiteness properties is related to the local-global principle in this setting. We will begin with the general formulation of the local-global
principle in terms of Galois cohomology, and will then indicate how it translates into statements about norms in finite separable extensions, finite-dimensional
simple algebras, and quadratic forms.


Let $K$ be a field equipped with a set $V$ of valuations (not necessarily discrete), and let $G$ be a linear algebraic $K$-group.
We say that the (cohomological) local-global principle holds for $G$ with respect to $V$ if the global-to-local map in Galois cohomology
$$
\theta_{G , V} \colon H^1(K , G) \longrightarrow \prod_{v \in V} H^1(K_v , G),
$$
given by the product of restriction maps, is injective. As we noted in \S\ref{S-ReviewGC}, these Galois cohomology sets do \emph{not}, in general, have a natural group structure, and should instead be viewed as pointed sets whose distinguished element is the cohomology class of the trivial cocycle. For such a map $\theta_{G,V}$, one defines the corresponding {\it Tate-Shafarevich set} $\text{\brus SH}(G , V)$ to be the kernel $\ker \theta_{G,V}$, i.e. the preimage of the distinguished element. We should emphasize that, as a reflection of the absence of a natural group structure, different fibers of $\theta_{G , V}$ may have different sizes. In particular, while $\text{\brus SH}(G , V)$ is certainly trivial (i.e. reduces to a single element) when $\theta_{G,V}$ is injective, in general, the injectivity of $\theta_{G,V}$ is \emph{not} a formal consequence of the triviality of $\text{\brus SH}(G , V)$ and typically requires additional considerations involving \emph{twisting} (cf. \cite[Ch. 1, \S1.3.2]{Pl-R}, \cite[Ch. I, \S5]{Serre-GC}).

We now briefly recall how the cohomological local-global principle is interpreted in several concrete situations. First, let $L/K$ be a finite separable field extension, and let $T = \mathrm{R}^{(1)}_{L/K}(\mathbb{G}_m)$ be the corresponding norm torus (see Example 3.4 for a special case of such a torus and, e.g., \cite[Ch. 2, \S2.1.7]{Pl-R} for a discussion of the general case). As a consequence of Hilbert's Theorem 90 and Shapiro's Lemma, we have a natural isomorphism $$H^1(K, T) \simeq K^{\times}/ N_{L/K}(L^{\times}).$$ Therefore, the cohomological local-global principle for $T$ is equivalent to the statement (known as the {\it local-global norm principle}) that an element $a \in K^{\times}$ is a norm in the extension $L/K$ (i.e., $a$ belongs to the norm subgroup $N_{L/K}(L^{\times}) \subset K^{\times}$) if and only if it is a norm locally at all $v \in V$ (i.e., $a \in N_{L \otimes_K K_v/K_v}((L \otimes_K K_v)^{\times})$ for all $v \in V$).

In slightly different terms, if $L/K$ has degree $n$, then the norm $N_{L/K}(x)$ of an element $x \in L$ is given by a homogeneous polynomial $\nu(x_1, \ldots , x_n)$ over $K$ in terms of the coordinates $x_1, \ldots , x_n$ of $x$ with respect to a fixed basis of $L/K$. Then the norm principle asserts that the equation $$\nu(x_1, \ldots , x_n) = a$$ has a solution $(x_1, \ldots , x_n) \in K^n$ if and only it has a solution $(x_1^v, \ldots , x_n^v) \in K_v^n$ for every $v \in V$. Among the first results on the local-global norm principle was the famous Hasse Norm Theorem asserting that this principle indeed holds for all cyclic Galois extensions of number fields when $V$ is the set of all valuations of $K$, including the archimedean ones.
So, to recognize the pioneering role of this result in the subject, the local-global principle  is often referred to as the {\it Hasse principle}.

Next, it follows from Example 4.5 that the cohomological local-global principle for the group $G = \mathrm{PGL}_n$ is equivalent to the statement that two central simple algebras $A_1$ and $A_2$ of degree $n$ over $K$ are isomorphic if and only the algebras $A_1 \otimes_K K_v$ and $A_2 \otimes_K K_v$ are isomorphic over $K_v$ for all $v \in V$. Thus, the truth of the local-global principle for $G = \mathrm{PGL}_n$ for all $n \geq 2$ amounts to the fact that the natural map of Brauer groups
\begin{equation}\label{E:FF1}
\mathrm{Br}(K) \longrightarrow \prod_{v \in V} \mathrm{Br}(K_v),
\end{equation}
defined by sending the Brauer class $[A]$ of a finite-dimensional central simple $K$-algebra $A$ to $([A \otimes_K K_v])_{v \in V}$, is injective.
Furthermore, the local-global principle for $G = \mathrm{O}_n(q)$ (the orthogonal group of a non-degenerate $n$-dimensional quadratic form $q$ over $K$) means that two nondegenerate $n$-dimensional quadratic forms $q_1$ and $q_2$ over $K$ are $K$-equivalent if and only if they are $K_v$-equivalent for all $v \in V$.
Similar interpretations can be given in the context of simple algebras with involution, hermitian forms, and so on.

Initially, the study of local-global principles focused almost exclusively on the case where $K$ is a number field and $V$ is the set of all valuations of $K$, and dealt primarily with some concrete situations rather than with the general cohomological set-up that we just described.
In particular, the norm principle was thoroughly investigated for arbitrary finite field extensions of number fields using techniques from class field theory (cf. \cite[Ch. VII, \S11.4]{ANT}). Another consequence of class field theory is the theorem of Albert-Brauer-Hasse-Noether stating that the map (\ref{E:FF1}) is injective in this case (see \cite[Ch. 18, \S18.4]{Pierce} for number fields and \cite[Ch. 6, \S6.5]{Gille} for function fields). As we mentioned above, this implies the cohomological Hasse principle for $G = \mathrm{PGL}_n$ for all $n \geq 2$, and, more generally, for $G = \mathrm{PGL}_{1 , A}$ for any finite-dimensional central simple algebra $A$ over $K$. Furthermore, the Hasse-Minkowski Theorem in the theory of quadratic forms implies the local-global principle for equivalence of quadratic forms, hence the cohomological local-global principle for the orthogonal groups (see \cite[Ch. IV, \S 3]{Serre-Course} for a discussion of the Hasse-Minkowski theorem in the special case $K = \Q$ and \cite[Ch. VI, \S66]{OMeara} for the general case). Eventually, these results and their variations led to the cohomological Hasse principle for \emph{all} semi-simple \emph{simply connected} groups with components of classical types.
This result was subsequently extended, using structural information provided by the theory of algebraic groups,  to include exceptional types, ultimately culminating in the proof of the cohomological Hasse principle for \emph{all} semi-simple simply connected groups over number fields (cf. \cite[Ch. 6, \S\S6.7-6.8]{Pl-R}). In fact, this result implies that the principle also holds for all absolutely almost simple groups, and consequently for all adjoint groups --- see \cite[Ch. 6, \S6.5]{Pl-R}.

On the other hand, it was discovered rather early that the Hasse norm principle may fail for non-cyclic finite extensions of number fields, which, in turn, entails
the failure of the cohomological principle for the corresponding norm torus. Later, examples of the failure of the cohomological Hasse principle were also found for semi-simple groups --- cf. \cite[Ch. III, \S4.7, Theorem 8]{Serre-GC} (note that such groups are, of course, neither simply connected nor adjoint). Nevertheless, it was proved in \cite{Borel} using reduction theory for adelic groups (see also \cite{BorelSerre}) that when $K$ is a number field, the map $\theta_{G,V}$ is \emph{proper} (i.e., the pre-image of a finite set is finite) for \emph{any} linear algebraic group $G$ whenever $V$ contains almost all valuations of $K$. The analogous statement for reductive groups over global fields of positive characteristic follows from results of Harder \cite{Harder1}.
Informally, these results mean that while the local-global principle for groups over global fields may fail, the deviation from it is always of finite size.

In the past 10-15 years, the Hasse principle has been analyzed over certain other classes of fields, including function fields of $p$-adic curves, cf. \cite{CTPS}, \cite{HHK}, \cite{HHK1}. This work has provided numerous examples where the local-global principle holds for fields other than global. In this article, however, we would like to focus on the arithmetic situation, where the classical results over global fields and our own work over higher-dimensional global fields (see \S\ref{S-ResultsHasse}) strongly suggest that the following statement should be true.


\vskip2mm

\noindent {\bf Conjecture 6.1.} {\it Let $G$ be a (connected) reductive algebraic group defined over a finitely generated field $K$, and let $V$ be a divisorial set of places of $K$. Then the global-to-local map $\theta_{G , V}$ is \emph{proper}. In particular, the Tate-Shafarevich set $\text{\brus SH}(G , V)$ is finite.}

\vskip2mm

Thus, Conjecture 6.1 expresses a broad expectation that the deviation from the local-global principle should be finite in all situations involving reductive linear algebraic groups over a finitely generated field $K$ with respect to any divisorial set of places $V$ of $K$. We will discuss available results on Conjecture 6.1 in \S\ref{S-ResultsHasse} At this point, we would like to indicate how the truth of Conjecture 5.7 automatically implies that of Conjecture 6.1 for \emph{adjoint} groups (cf. \cite[\S 6]{CRR2}, \cite[\S7]{R-ICM}).

To fix notations, let $G$ be an absolutely almost simple simply connected algebraic group over a finitely generated $K$ of ``good" characteristic, and let $V$ be a divisorial set of discrete valuations of $K$. Note that since $V$ satisfies condition (A) (see \S\ref{S-ConjFGfields}), it follows that we can pick a finite subset $V_0 \subset V$ so that $G$ has good reduction at all $v \in V \setminus V_0$. As we observed in \S\ref{S-ConjFGfields}, $V \setminus V_0$ contains a divisorial set $V'$. Thus, replacing $V$ by $V'$, we may assume that $G$ has good reduction at all $v \in V$. Now set $\overline{G}$ to be the corresponding adjoint group.
Suppose $\xi \in \text{{\brus
SH}}(\overline{G},V)$ and let $G' = G(\xi)$ be the corresponding (inner) $K$-form of $G$. By our assumption, $G' \times_K K_v \simeq G \times_K K_v$ for all $v \in V$, and consequently $G'$ has good reduction at all $v \in V$. Therefore, assuming Conjecture 5.7, we conclude that the groups $G(\xi)$ for
$\xi \in \text{{\brus SH}}(\overline{G},V)$
form finitely many $K$-isomorphism classes. In cohomological terms, this means that that the image of
$\text{{\brus SH}}(\overline{G},V)$ under the canonical map
$H^1(K , \overline{G}) \stackrel{\alpha}{\longrightarrow} H^1(K , A_{\Ks})$ is finite (as in Example 4.5, we denote by $A_{\Ks}$ the group of automorphisms of $G \times_K \Ks$). But since $\overline{G} \simeq I_{\Ks}$ has finite index in $A_{\Ks}$, the
map $\alpha$ has finite fibers, which yields the finiteness of
$\text{{\brus SH}}(\overline{G},V)$ (see \cite[Ch. I, \S\S5.3-5.5]{Serre-GC}) for a discussion of the fibers of such maps in non-abelian cohomology).



\vskip2mm

Another fundamental finiteness property in number theory is the finiteness of the ideal class group of a number field. In the higher-dimensional situation, one replaces the class group by the Picard group. While this group can be infinite, it is known that if $X$ is a scheme that is normal and of finite type over ${\rm Spec}(\Z)$, then the Picard group ${\rm Pic}~X$ is finitely generated (cf. \cite{Kahn}). In order to transport this notion into the context of linear algebraic groups, one uses {\it adeles}.


So, let $K$ be a field equipped with a set $V$ of (pairwise inequivalent) discrete valuations, and let $G$ be a linear algebraic $K$-group with a fixed
matrix realization $G \subset \mathrm{GL}_n$. For each $v \in V$, we set
$$
G(\mathcal{O}_v) = G(K_v) \cap \mathrm{GL}_n(\mathcal{O}_v),
$$
where $\mathcal{O}_v$ is the valuation ring in the completion $K_v$. We then define the corresponding {\it adelic group} as
$$
G(\mathbb{A}(K , V)) = \{ (g_v) \in \prod_{v \in V} G(K_v) \ \vert \ g_v \in G(\mathcal{O}_v) \ \ \text{for almost all} \ \ v \in V \}.
$$
In other words, $G(\mathbb{A}(K , V))$ is the {\it restricted (topological) product} of the groups $G(K_v)$ for $v \in V$ with respect to the (open) subgroups
$G(\mathcal{O}_v)$ (cf. \cite[Ch. 5, \S5.1]{Pl-R} for the details). The product
$$
G(\mathbb{A}^{\infty}(K , V)) = \prod_{v \in V} G(\mathcal{O}_v)
$$
is called the {\it subgroup of integral adeles}. Henceforth, we will assume that $V$ satisfies condition (A) introduced in \S\ref{S-ConjFGfields} (we recall that this condition holds automatically for a divisorial set of places of a finitely generated field).
%
%
%
Then one has a diagonal embedding $G(K) \hookrightarrow G(\mathbb{A}(K , V))$, whose image is called the {\it subgroup of principal adeles} and which we will still denote by $G(K)$. The set of double cosets
$$
\mathrm{cl}(G, K, V) := G(\mathbb{A}^{\infty}(K , V)) \backslash G(\mathbb{A}(K , V)) / G(K)
$$
is called the {\it class set} of $G$ (we should point out that the class set is sometimes defined using {\it rational adeles} rather than the {\it full}
adeles we introduced above). The following examples link this definition with classical notions.

\vskip3mm

\noindent {\bf Example 6.2.} Let $G = \mathbb{G}_m$ be the 1-dimensional $K$-split torus. Then $G(\mathbb{A}(K , V))$ is the {\it group of ideles} $\mathbb{I}(K , V)$ and $G(\mathbb{A}^{\infty}(K , V))$ is the {\it subgroup of integral ideles}
$$
\mathbb{I}^{\infty}(K , V) = \prod_{v \in V} \mathcal{O}_v^{\times}.
$$
So, there is a natural bijection between the class set $\mathrm{cl}(G, K, V)$ and the quotient $\mathbb{I}(K , V)/\mathbb{I}^{\infty}(K , V) K^{\times}$. On the other hand, it is easy to see that the latter quotient is isomorphic to the Picard group $\mathrm{Pic}(K , V)$, which is defined as follows. Let $\mathrm{Div}(V , K)$ be the free abelian on the set $V$, which we call the {\it group of divisors}. By virtue of condition (A), we can define a group homomorphism
$$
K^{\times} \to \mathrm{Div}(K , V), \ \ \ a \mapsto \sum_v v(a) \cdot v,
$$
the image of which is the called the {\it subgroup of principal divisors} $\mathrm{P}(K , V)$. We set
$$
\mathrm{Pic}(K , V) = \mathrm{Div}(K , V) / \mathrm{P}(K , V).
$$
The isomorphism $\mathbb{I}(K , V)/\mathbb{I}^{\infty}(K , V)K^{\times} \simeq {\rm Pic}(K,V)$ is then induced by the map
$$
\nu \colon \mathbb{I}(K,V) \to {\rm Div}(K,V), \ \ \ (x_v) \mapsto \sum_{v \in V}v(x_v) \cdot v.
$$
If $K$ is the fraction field of a Dedekind domain $A$ and $V$ is the set of discrete valuations of $K$ corresponding to the maximal ideals
of $A$, then $\mathrm{Pic}(K , V)$ is precisely the ideal class group of $A$. More generally, if $K$ is the function field of an integral regular
scheme $X$ and $V$ is the set of discrete valuations associated with prime divisors of $X$, then $\mathrm{Pic}(K , V)$ coincides with the usual Picard group $\mathrm{Pic}(X)$. So, it follows from classical results that $\mathrm{Pic}(K , V)$ is finite in the first situation when $A$ is the ring of $S$-integers in a number field, and is finitely generated in the second situation when $X$ is integral, normal, and of finite type over ${\rm Spec}(\Z)$.

\vskip3mm

\noindent {\bf Example 6.3.} Let $K$ be a number field with the ring of integers $\mathcal{O}$, and let $V$ be the set of all nonarchimedean valuations of $K$.
Furthermore, let $q$ be a nondegenerate quadratic form in $n$ variables with coefficients in $\mathcal{O}$, and $G = \mathrm{O}_n(q)$ be the corresponding orthogonal group. It is well-known that in this case, there is a natural bijection between the class set
$\mathrm{cl}(G, K, V)$ and the set of classes in the genus of $q$ (see \cite[Ch. 8, Proposition 8.4]{Pl-R} for the details and relevant definitions). 
We recall that when $n \geq 3$ and $q$ is {\it indefinite} (i.e., there exists an archimedean place $v \in V^K_{\infty}$ that is either
complex ($K_v = \mathbb{C}$) or is real and $q$ is indefinite in the usual sense over $K_v = \mathbb{R}$), then $\mathrm{cl}(G, K, V)$ has a natural structure of an abelian group, which is {\it finite} of order a power of $2$ (see \cite[Ch. 8, \S8.2, Theorem 8.6]{Pl-R}). On the contrary, if $q$ is {\it definite}, then $\mathrm{cl}(G, K, V)$ is a finite set which in general does not have a natural group structure and whose size can be made divisible by any given integer if one changes $q$ to a rationally equivalent form  (see \cite[Ch. 8. \S8.3, Theorem 8.9]{Pl-R}).

\vskip2mm

More generally, it was shown by Borel \cite{Borel} using reduction theory that if $K$ is a number field and $V$ is the set of all nonarchimedean valuations of $K$, then the class set $\mathrm{cl}(G, K, V)$ is {\it finite} for {\it any} linear algebraic $K$-group $G$. This finiteness result was extended to global fields of positive characteristic by Conrad \cite{ConradFiniteness}, who employed the theory of pseudo-reductive groups developed by Conrad-Gabber-Prasad \cite{CGP}. On the other hand, for $G = \mathbb{G}_m$ over a finitely generated field $K$ with $V$ a divisorial set of places, $\mathrm{cl}(G, K, V) \simeq \mathrm{Pic}(K , V)$ may be an infinite group, which is nevertheless finitely generated. For an arbitrary linear algebraic group, however, $\mathrm{cl}(G, K, V)$ may not have a natural group structure, so no general finiteness condition on the class set can conceivably be stated on the basis of finite generation. We have proposed the following path towards a possible generalization, which consolidates the two cases discussed above and appears to be quite useful. First, we observe that if $\mathrm{cl}(G, K, V)$ is either finite or a finitely generated group (in the presence of a natural group structure), then one easily shows that there exists a finite subset $S \subset V$ such that the class set $\mathrm{cl}(G, K, V\setminus S)$ reduces to a single element (see the argument following \cite[Definition 3.4]{CRR-Israel} for the details). This suggests the following condition on the triple $(G, K, V)$:

\vskip2mm

\noindent {\bf Condition (T).} {\it There exists a finite subset $S \subset V$ such that $\vert \mathrm{cl}(G, K, V \setminus S) \vert = 1.$}

\vskip2mm

While one does not expect Condition (T) to hold for an arbitrary reductive algebraic group $G$ over a general finitely generated field $K$ and a divisorial set $V$, it is likely to be true for all $G$ in certain important situations, including when

\vskip2mm

\noindent $\bullet$ \parbox[t]{16cm}{$K$ is a {\it 2-dimensional global field} (i.e. the function field of a smooth geometrically integral curve over a number field or the function field of a smooth geometrically integral surface over a finite field -- see \cite{CRR-Spinor} and \cite{Kato}) and $V$ is a divisorial set of places; and}

\vskip2mm

\noindent $\bullet$ \parbox[t]{16cm}{$K = k(C)$, the function field of a smooth geometrically integral curve $C$ over a finitely generated field $k$ and $V$ is the set of places of $K$ associated with the closed points of $C$.}

\vskip2mm

\noindent In fact, one can expect the following weaker property to be true even more generally: for a reductive group $G \subset \mathrm{GL}_n$ over a finitely generated field $K$ and a divisorial set of places $V$, there exists a finite subset $S \subset V$ such that
$$
G(\mathbb{A}(K , V \setminus S)) \bigcap \left( \mathrm{GL}_n(\mathbb{A}^{\infty}(K , V \setminus S)) \cdot \mathrm{GL}_n(K) \right) = G(\mathbb{A}^{\infty}(K , V \setminus S)) \cdot G(K).
$$

Finally, we would like to observe that although at this point no direct connections between Conjectures 5.7 and 6.1 and Condition (T) have been established in the general case, ideas involving Condition (T) were used in a very essential way in the proof of Conjecture 6.1 for tori --- see \S\ref{S-ToriResults} for a brief discussion and \cite{RR-Tori} for complete details. Moreover, as we pointed out in \cite{CRR-Israel}, Condition (T) can also be used in the analysis of some finiteness questions for unramified cohomology in degree 3 and hence the genus problem for the groups of type $\textsf{G}_2$ (we will touch upon some aspects of these issues in \S\ref{S-ResultsMain} and \S\ref{S-GenusConj} below).

\section{Results}\label{S-Results}

In this section, we will give an overview of the currently available results on Conjectures 5.7 and 6.1 as well as Condition (T). We begin with the case of algebraic tori, where all conjectures were recently resolved in \cite{RR-Tori}. We then move on to absolutely almost simple algebraic groups, where there has been notable progress (see, in particular, \cite{CRR4} and \cite{CRR-Spinor}), but much work still remains to be done.


\vskip2mm

\subsection{Algebraic tori.}\label{S-ToriResults} First, we have the following finiteness result for tori with good reduction, which completely settles Conjecture 5.7 in this case.

\begin{thm}\label{T:Res1}{\rm (\cite[Theorem 1.1]{RR-Tori})}
Let $K$ be a finitely generated field of characteristic zero and $V$ be a divisorial set of places of $K$. Then for any integer $d \geq 1$, the set of $K$-isomorphism classes of $d$-dimensional $K$-tori that have good reduction at all $v \in V$ is finite.
\end{thm}

As we discussed in Examples 4.2(a) and 4.4, all $d$-dimensional $K$-tori are $K$-forms of the $d$-dimensional $K$-split torus $(\mathbb{G}_m)^d$, and their $K$-isomorphism classes are classified by the minimal splitting field $L$ and the equivalence class of a faithful representation $\mathrm{Gal}(L/K) \to \mathrm{GL}_d(\mathbb{Z})$. Furthermore, it is a general fact (see \cite[Ch. 4, \S4.4, Theorem 4.3]{Pl-R}) that a given finite group has only finitely many equivalence classes of integral representations in each dimension. So, it is enough to prove that there are only finitely many possibilities for the extension $L/K$. The key observations here are, first, that the degree $[L : K]$ is bounded by a constant depending only on $d$ (as are the orders of finite subgroups of $\mathrm{GL}_d(\mathbb{Z})$) and, second, that this extension is unramified at all $v \in V$.  Let $X$ be a model of $K$ used to define $V$. Then the fact that there are only finitely many possible extensions $L/K$ is derived from the result that the fundamental group of $X$ is ``small," i.e. satisfies Serre's condition (F) (see \cite{HH}), which can
be viewed as a higher-dimensional analogue of the Hermite-Minkowski theorem that we have already mentioned several times in our previous discussion.

It should be mentioned that Theorem \ref{T:Res1} is no longer true in positive characteristic. Indeed, an infinite family of pairwise nonisomorphic $K$-tori with good reduction can be constructed using Artin-Schreier extensions over the global field $K = \mathbb{F}_p(t)$ (where $\mathbb{F}_p$ is the field with $p$ elements), with $V$ being the set of discrete valuations corresponding to all monic irreducible polynomials $f(t) \in \mathbb{F}_p[t]$ (as in Example 3.6(b)).

Now, let $G$ be an absolutely almost simple simply connected $K$-group, and let $L/K$ be the minimal Galois extension over which  $G$ becomes an inner form
of the split group (see the discussion in Example 4.5). If $G$ has good reduction at a discrete valuation $v$ of $K$, then $L/K$ is unramified at $v$. Replacing the use of the
Hermite-Minkowski theorem in the proof of Proposition \ref{P:XXX1} with its
higher-dimensional analogue, we see that in the case where $K$  is a
finitely generate field of characteristic zero and $V$ is a divisorial
set of places, there are only finitely many continuous homomorphisms
$$
\omega \colon \mathrm{Gal}(K^{\mathrm{sep}}/K) \to S
$$
to the symmetry group $S$ of the Dynkin diagram of $G$ such that
$L(\omega) = (K^{\mathrm{sep}})^{\ker \omega}$ is unramified at all $v
\in V$. Let $\omega_1, \ldots , \omega_r$ be representatives of the
conjugacy classes of such homomorphisms, and let $G_1, \ldots , G_r$
be the corresponding $K$-quasi-split forms of $G$.
Then any $K$-form of $G$ that has good reduction at all $v \in V$ is an {\it inner} form of one of the $G_i$. This implies that it is enough to prove Conjecture 5.7 for inner forms of all quasi-split $K$-groups. In fact, this conclusion remains valid over finitely generated fields of positive characteristic $p > 3$, but some care needs to be exercised in characteristic 2 and 3, even as far as the formulation of Conjecture 5.7 is concerned --- cf. \cite[Remark 2.6]{RR-Tori} for further details.

Finally, if $G$ is a non-semi-simple reductive group over a finitely generated field $K$, then applying Theorem \ref{T:Res1} to the maximal central torus, we see that it is enough to prove Conjecture 5.7 for the derived group of $G$ (cf. \cite[Ch. IV, \S 14.2, Proposition]{BorelAG} for the relevant structural information), which reduces the conjecture to the semi-simple case. In turn, the semi-simple case can essentially be reduced to the case of absolutely almost simple simply connected groups, which we will consider in the next subsection. Now, however, we turn our attention to the finiteness of the Tate-Shafarevich group of an algebraic torus.

\begin{thm}\label{T:Res2}{\rm (\cite[Theorem 1.2]{RR-Tori})}
Let $K$ be a finitely generated field and $V$ be a divisorial set of places of $K$. Then for any algebraic $K$-torus $T$, the Tate-Shafarevich group
$$
\text{{\brus SH}}^1(T , V) = \ker \left( H^1(K , T) \to \prod_{v \in V} H^1(K_v , T) \right)
$$
is finite.
\end{thm}

As we already indicated in \S\ref{S-OtherConj}, this result can be interpreted in a variety of concrete situations. For example,
let $L$ be a finite separable extension of a finitely generated field $K$ with a divisorial set of places $V$. Let
$$
\mathcal{N} := \{ a \in K^{\times} \ \vert \ a \in N_{L \otimes_K K_v/K_v}((L \otimes_K K_v)^{\times}) \ \ \text{for all} \ \ v \in V \}
$$
be the group of ``local norms." Then Theorem \ref{T:Res2} applied to the norm torus $T = \mathrm{R}^{(1)}_{L/K}(\mathbb{G}_m)$
implies that the quotient $\mathcal{N}/ N_{L/K}(L^{\times})$ of $\mathcal{N}$ modulo the subgroup of ``global norms" is  always finite.

It should be noted that the classical proof of the finiteness of $\text{{\brus SH}}^1(T , V)$ when $K$ is a number field and $V$ is set of {\it all} places of $K$ (including the archimedean ones) relies on Tate-Nakayama duality (see \cite[Ch. 4, \S11.3, Theorem 6]{Voskr} and the subsequent discussion), which is not available in the general situation. So, in \cite{RR-Tori}, we gave two different proofs of Theorem \ref{T:Res2}. The first one requires an additional assumption on the characteristic of $K$, but it develops an approach that is applicable in some situations involving fields that are not finitely generated. The second one systematically uses adelic techniques in the context of arbitrary finitely generated fields and their divisorial sets of valuations, which, to the best of our knowledge, have not been previously employed in this generality. More specifically, the argument relies on the validity of Condition (T) in the present case --- see Theorem \ref{T:Res3} below.
Our second proof demonstrates, in particular, that in the classical situation where $K$ is a number field, the finiteness of the Tate-Shafarevich group can be established without Tate-Nakayama duality, and is actually a direct consequence of two basic facts: the finite generation of the group of $S$-units and the finiteness of the class number. Moreover, this argument applies to cohomology groups in all degrees and yields the following result: {\it Let $K$ be a finitely generated field, $V$ be a divisorial set of places, and $F/K$ be a finite separable extension. Then for any $i \geq 1$, the group
$$
\text{{\brus SH}}^i(F/K, T, V) := \ker\left(H^i(F/K , T) \longrightarrow \prod_{v \in V} H^i(F_w/K_v , T) \right)
$$
is finite (here for each $v \in V$, we pick one extension $w$ to $F$).}

\vskip2mm

The following result verifies Condition (T) for groups more general than tori.

\begin{thm}\label{T:Res3}{\rm (\cite[Theorem 3.4]{RR-Tori})}
Let $K$ be a finitely generated field and $V$ be a divisorial set of places of $K$. Then any linear algebraic $K$-group $G$ whose connected component
is a torus satisfies Condition {\rm (T)}.
\end{thm}

\vskip3mm

We would now like to mention analogues of Theorems \ref{T:Res1} and \ref{T:Res2} for function fields of curves over base fields of type (F) (see \S\ref{S-FConjCurves} for the definition). To fix notations, suppose $k$ is a field of type (F) and having characteristic 0. Let $K = k(C)$ be the function field of a smooth geometrically integral curve $C$ over $k$, and let $V$ be the set of discrete valuations of $K$ associated with the closed points of $C$. We then have the following statement concerning tori with good reduction.

\begin{thm}\label{T:Res4}
With notations as above, for each $d \geq 1$, there are finitely many $K$-isomorphism classes of $d$-dimensional $K$-tori that have good reduction at all $v \in V$.
\end{thm}

The proof of this (unpublished) result proceeds essentially along the same lines as that of Theorem \ref{T:Res1} sketched above, with the key input being the fact that the \'etale fundamental group of $C$ is small. For this, we fix an algebraic closure $\bar{k}$ of $k$, set $\overline{C} = C \times_{{\rm Spec}(k)} {\rm Spec}(\bar{k})$, and let $\bar{x}$ be the corresponding geometric point of $\overline{C}.$ We then have the following standard exact sequence of profinite groups
$$
1 \to \pi_1(\overline{C}, \bar{x}) \to \pi_1(C, \bar{x}) \to \Ga(\bar{k}/k) \to 1
$$
induced by the natural maps $\overline{C} \to C$ and $C \to {\rm Spec}(k).$ By our assumption, $\Ga(\bar{k}/k)$ is small. Furthermore, it is well-known that $\pi_1(\overline{C}, \bar{x})$ is topologically finitely generated (see, for example, \cite[Theorem 4.6.7]{SzamuelyFG}), and hence is small by \cite[Ch. III, \S4, Proposition 9]{Serre-GC}. Applying \cite[Lemma 2.7]{HH}, we conclude that $\pi_1(C, \bar{x})$ is small, as required.

Next, the first proof of Theorem \ref{T:Res2} alluded to above, together with the finiteness results for the unramified cohomology of tori obtained by the second author in \cite[\S5]{IRap}, yield the following.


\begin{thm}\label{T:Res5}
With notations as above, for any $K$-torus $T$, the Tate-Shafarevich group $\text{{\brus SH}}^1(T , V)$ is finite.
\end{thm}

\vskip3mm

\subsection{Absolutely almost simple groups: Conjectures 5.5 and 5.7.}\label{S-ResultsMain} We begin with the results that have been obtained so far on the Main Conjecture 5.7 for absolutely almost simple groups.


First, for inner forms of type $\mathsf{A}_n$, Conjecture 5.7 has been proved completely.


\begin{thm}\label{T:ResX1}
Let $K$ be a finitely generated field and $V$ be a divisorial set of places of $K$. Then for any $n \geq 2$ that is prime to $\mathrm{char}\: K$ and any simply connected inner form $G$ of type $\textsf{A}_{n-1}$, the number of $K$-isomorphism classes of \emph{inner} $K$-forms of $G$ that have good reduction at all $v \in V$ is finite.
\end{thm}

We recall that the group $G$ in the statement of the theorem is of the form $\mathrm{SL}_{1,A}$ --- the algebraic group associated with the group of elements of reduced norm 1 in a central simple $K$-algebra $A$
of degree $n$ over $K$ (cf. \cite[Ch. 2, \S2.3, Proposition 2.17]{Pl-R}). As we remarked in Example 3.8(b), such a group
has good reduction at a discrete valuation $v$ of $K$ if and only if the algebra $A$ is unramified at $v$, i.e.
there exists an { Azumaya algebra} $\mathcal{A}$ over the valuation ring $\mathcal{O}_v \subset K_v$ such that there is an isomorphism of $K_v$-algebras
$$
A \otimes_K K_v \simeq \mathcal{A} \otimes_{\mathcal{O}_v} K_v.
$$
A detailed discussion of Azumaya algebras can be found, for example, in \cite{KO}, \cite[Ch. IV]{Milne}, and \cite[Ch. 2]{Salt}. For our purposes, we merely recall that an $\mathcal{O}_v$-algebra $\mathcal{A}$ is called an {\it Azumaya algebra} if it is a free $\mathcal{O}_v$-module of finite rank and if the canonical homomorphism of $\mathcal{O}_v$-algebras
$$
\mathcal{A} \otimes_{\mathcal{O}_v} \mathcal{A}^{\mathrm{op}} \to \mathrm{End}_{\mathcal{O}_v}(\mathcal{A})
$$
(where $\mathcal{A}^{\mathrm{op}}$ denotes the opposite algebra) that sends a simple tensor $a \otimes a'$ to the endomorphism $(x \mapsto axa')$,
is an isomorphism. In this case, the quotient $\mathcal{A}/\mathfrak{p}_v\mathcal{A}$ (where $\mathfrak{p}_v \subset \mathcal{O}_v$ is the valuation ideal) is a central simple algebra over the
residue field $K^{(v)} = \mathcal{O}_v/\mathfrak{p}_v$.

The key input in the proof of Theorem \ref{T:ResX1} is a finiteness result for the {\it unramified Brauer group}, which we now describe.
First, we recall that the Brauer group $\Br(K)$ of a field $K$ consists of the Brauer equivalence classes of finite-dimensional central simple $K$-algebras (for the details of this construction, the reader can consult \cite{FarbDennis}, \cite{Pierce}, or \cite{Salt}).
Given a central simple algebra $A$ over $K$, we denote by $[A]$ the corresponding class in $\Br(K).$ It is well-known that if $A$ is a central simple $K$-algebra of degree $n$, then $[A]$ is annihilated by multiplication by $n$ in $\mathrm{Br}(K)$, i.e. belongs to the $n$-torsion subgroup ${}_n\mathrm{Br}(K)$. Furthermore, given a discrete valuation $v$ of $K$, we say that $a \in \mathrm{Br}(K)$ is unramified at $v$ if it can be represented by a central simple $K$-algebra $A$ that is unramified at $v$, as defined above. Now, if $V$ is a set of discrete valuations of $K$, we let $\mathrm{Br}(K)_V$ denote the subgroup of $\mathrm{Br}(K)$ consisting of elements that are unramified at all $v \in V$ (this group is usually referred to as the {\it unramified Brauer group} of $K$ with respect to $V$).

With these preliminaries, we now come to the following finiteness statement from which Theorem \ref{T:ResX1} is derived (cf. \cite{CRR3}):

\vskip2mm

\noindent {\it Let $K$ be a finitely generated field and $V$ be a divisorial set of places of $K$. Then for any $n \geq 1$ that is prime to
$\mathrm{char}\: K$, the group ${}_n\mathrm{Br}(K)_V = {}_n\mathrm{Br}(K) \cap
\mathrm{Br}(K)_V$ is finite.}

\vskip2mm

\noindent Now, with notations as in Theorem \ref{T:ResX1}, any inner $K$-form $G'$ of $G$ that has good reduction at all $v \in V$ is of the form
$\mathrm{SL}_{1 , A'}$, where $A'$ is a central simple $K$-algebra of degree $n$ with $[A'] \in {}_n\mathrm{Br}(K)_V$. Since the latter group is finite, the required finiteness in the theorem follows from the fact that if $A_1$ and $A_2$ are two central simple $K$-algebras of degree $n$ and $[A_1] = [A_2]$ in $\mathrm{Br}(K)$, then the algebras $A_1$ and $A_2$ are isomorphic, implying that the corresponding algebraic groups $G_1 = \mathrm{SL}_{1 , A_1}$ and $G_2 = \mathrm{SL}_{1 , A_2}$ are also isomorphic.

The proof of the finiteness of ${}_n\mathrm{Br}(K)_V$ relies on cohomological techniques. We will not go into the details here and will only briefly indicate the basic set-up.
To begin with, it is well-known that $\mathrm{Br}(K)$ can be identified with $H^2(K , (\Ks)^{\times})$, and under this identification, for any $n \geq 1$ that is prime to $\mathrm{char}\: K$, the $n$-torsion subgroup ${}_n\mathrm{Br}(K)$ corresponds to $H^2(K , \mu_n)$, where $\mu_n$ is the group of $n$th roots of unity in $\Ks$ (cf. \cite[Ch. 14]{Pierce}). Furthermore, for any discrete valuation $v$ of $K$ such that $\mathrm{char}\: K^{(v)}$ is prime to $n$,  there exists a  {\it residue map}
$$
\partial_v^2 \colon H^2(K , \mu_n) \longrightarrow H^1(K^{(v)} , \mathbb{Z}/n\mathbb{Z}),
$$
where $\mathbb{Z}/n\mathbb{Z}$ denotes $\mu_n$ with the trivial Galois action --- see \cite[\S6.8]{Gille} for the details. It turns out that an element $a \in H^2(K , \mu_n)$ is unramified at $v$ as defined above if and only if $\partial_v^2(a) = 0$. So, assuming that $n$ is prime to $\mathrm{char}\: K^{(v)}$ (which we can always do in terms of proving Theorem \ref{T:ResX1} as $n$ is prime to $\mathrm{char}\: K$), one can define the {\it unramified cohomology group}
$$
H^2(K , \mu_n)_V := \bigcap_{v \in V} \ker \partial_v^2
$$
and then identify it with ${}_n\mathrm{Br}(K)_V$. The proof then proceeds via a careful analysis of the group $H^2(K, \mu_n)_V.$ Let us point out that in \cite{CRR3}, we not only established the finiteness of the latter group, but also provided explicit estimations of its order in certain situations.

\vskip2mm

Next, we will discuss  results on Conjecture 5.7 for groups of several other types over a special class of finitely generated fields --- the so-called {\it two-dimensional global fields}. Following Kato \cite{Kato}, by a two-dimensional global field, we mean either the function field $K = k(C)$ of a smooth geometrically integral curve $C$ over a number field $k$, or the function field $K = k(S)$ of a smooth geometrically integral surface $S$ over a finite field $k$.

\begin{thm}\label{T:ResX2}{\rm (\cite[Theorem 1.1]{CRR-Spinor})}
Let $K$ be a two-dimensional global field of characteristic $\neq 2$ and let $V$ be a divisorial set of places of $K$. Fix an integer $n \geq 5$. Then the set of $K$-isomorphism classes of spinor groups $G = \mathrm{Spin}_n(q)$ of nondegenerate quadratic forms in $n$ variables over $K$ that have good reduction at all $v \in V$ is finite.
\end{thm}

Whereas the proof of Theorem \ref{T:ResX1} was based on the study of the Brauer group, the proof of Theorem \ref{T:ResX2} requires an analysis of the Witt ring (we refer the reader to \cite{EKM} and \cite{Lam} for the construction and basic properties of the Witt ring). We will assume henceforth that $\mathrm{char}\: K \neq 2$, and denote by $W(K)$ the Witt ring of $K$ and by $I(K)$ its fundamental ideal. For a
nondegenerate quadratic form $q$ over $K$, we set $[q]$ to be the corresponding class in $W(K)$. Now, a consequence of Voevodsky's proof of Milnor's conjecture is that for $d \geq 1$, there are natural isomorphisms of abelian groups
$$
\gamma_{K , d} \colon I(K)^d/I(K)^{d+1} \longrightarrow H^d(K , \mu_2)
$$
(cf. \cite{OVV}).
On the other hand, as above, for any discrete valuation $v$ such that $\mathrm{char}\: K^{(v)} \neq 2$ and any $d \geq 1$, there exists a
{\it residue map}
$$
\partial_v^d \colon H^d(K , \mu_2) \longrightarrow H^{d-1}(K^{(v)} , \mu_2),
$$
which actually factors through the restriction map $H^d(K , \mu_2) \to H^d(K_v , \mu_2)$ (where, as before, $K_v$ denotes the completion of $K$ at $v$). Then one says that $a \in H^d(K , v)$ is {\it unramified at $v$} if
$\partial_v^d(a) = 0$. Moreover, if $K$ is equipped with a set $V$ of discrete valuations such that $\mathrm{char}\: K^{(v)} \neq 2$ for all $v \in V$, one defines the corresponding {\it unramified cohomology group} by
$$
H^d(K , \mu_2)_V = \bigcap_{v \in V} \ker \partial_v^d.
$$
To factor in good reduction, one proves the following technical statement. Let $v$ be a discrete valuation of $K$ such that $\mathrm{char}\: K^{(v)} \neq 2$. We let $W_0(K_v)$ denote the subring of $W(K_v)$ generated by the classes of 1-dimensional forms $ux^2$ with $u \in U_v = \mathcal{O}_v^{\times}$. Then we show the following:

\vskip2mm

\noindent {\it If $q$ is a nondegenerate form over $K_v$ such that $[\lambda q] \in I(K_v)^d \cap W_0(K_v)$ (where $d \geq 1$) for some $\lambda \in K_v^{\times}$, then $[q] \in I(K_v)^d$ and
$$
\gamma_{K_v , d}([q] + I(K_v)^{d+1}) \in H^d(K_v , \mu_2)
$$
is unramified at $v$.}

\vskip2mm

\noindent (See \cite[Lemma 3.3]{CRR-Spinor}). On the other hand, as we remarked in Example 3.8(c), the spinor group $G = \mathrm{Spin}_n(q)$ has good reduction at $v$ if and only if $q$ is $K_v$-equivalent to a quadratic form $\lambda q_0$ with $\lambda \in K_v^{\times}$ and $q_0 = u_1 x_1^2 + \ldots + u_n x_i^2$, where $u_i \in U_v$ for all $i = 1, \ldots , n$.

Suppose now that $K$ is a finitely generated field of characteristic $\neq 2$ equipped with a divisorial set of places $V$. Without loss of generality we may assume that $\mathrm{char}\: K^{(v)} \neq 2$ for all $v \in V$. Using the above discussion, in conjunction with some facts from the theory of quadratic forms (most notably, the Hauptsatz --- see \cite[Ch. X, \S 5]{Lam}), one shows that to prove Theorem \ref{T:ResX2},
it is enough to establish the finiteness of the unramified cohomology groups $H^d(K, \mu_2)_V$ for all $d \geq 1$ (see \cite[Theorem 2.1]{CRR-Spinor} for a more precise statement). So, to complete the argument, we prove the required finiteness over two-dimensional global fields of characteristic $\neq 2$. Here, we will comment on this fact assuming that ${\rm char}~K =0$ (for the positive characteristic case, the reader can consult \cite[\S7]{CRR-Spinor}).  Then the finiteness of $H^1(K , \mu_2)_V$ is a standard result and, as mentioned above,
the finiteness of $H^2(K , \mu_2)_V$ was established in the course of the proof of Theorem \ref{T:ResX1}. On the other hand, the finiteness of $H^d(K , \mu_2)_V$ for $d \geq 4$ can be derived from results of Poitou-Tate (cf. \cite[Proposition 4.2]{CRR-Spinor}). The most challenging case is the finiteness of $H^3(K, \mu_2)_V$, for which we gave two proofs in \cite{CRR-Spinor}. One proof makes use of several powerful results, first and foremost, those of Kato \cite{Kato} on cohomological Hasse principles. The second proof requires considerably less input; in particular, it does not rely on Kato's local-global prinicple, but instead is based on a modification of Jannsen's proof of the latter \cite{Jann}. This argument appears to be more amenable to generalizations in the spirit of Jannsen's proof \cite{Jann2} of Kato's conjecture on the local-global principle for higher-dimensional varieties, which extended his original argument in \cite{Jann}.

The above sketches of the proofs of Theorems \ref{T:ResX1} and \ref{T:ResX2} indicate an intimate connection between Conjecture 5.7 and finiteness properties of unramified cohomology. The analysis of unramified cohomology in general and of unramified $H^3$ in particular, in both arithmetic and geometric situations, is a major direction of independent interest, which, however, lies beyond the scope of the current paper. We refer the reader to \cite{CT-SB} for an informative survey of this subject.

\vskip2mm

Next, we should mention that Theorem \ref{T:ResX2} yields similar results for some other types of groups. Here is the statement for simply connected {\it outer} forms of type $\textsf{A}_{n-1}$ that split over a quadratic extension $L/K$. We recall that such forms are obtained as follows. Let $L/K$ be a quadratic extension with a nontrivial automorphism $\tau$, and let $h$ be a nondegenerate
$\tau$-hermitian form of dimension $n \geq 2$.
Then the groups in question are the special unitary groups $G = \mathrm{SU}_n(L/K , h)$ (cf. \cite[Ch. 2, \S2.3.4]{Pl-R}).

\begin{thm}\label{T:ResX3}{\rm (\cite[Theorem 8.1]{CRR-Spinor})}
Let $K$ be a two-dimensional global field of characteristic $\neq 2$ and $V$ be a divisorial set of places of $K$. Fix a quadratic extension $L/K$, and let $n \geq 2$. Then the number of $K$-isomorphism classes of special unitary groups $G = \mathrm{SU}_n(L/K , h)$ of nondegenerate hermitian $L/K$-forms in $n$ variables that have good reduction at all $v \in V$ is finite.
\end{thm}

Since the number of quadratic extensions $L/K$ that are unramified at all $v \in V$ is finite, Theorem \ref{T:ResX3} in effect yields the finiteness of the number of $K$-isomorphism classes of special unitary groups with good reduction at all $v \in V$ of $n$-dimensional nondegenerate hermitian forms associated with {\it all} quadratic extensions $L/K$.

\vskip2mm

A result similar to Theorem \ref{T:ResX3} is also valid for the special unitary groups $G = \mathrm{SU}_n(D , h)$ of nondegenerate hermitian forms of dimension $n \geq 2$ over a quaternion division algebra $D$ with center $K$ with the canonical involution, which are precisely the absolutely almost simple simply connected groups of type $\textsf{C}_n$ that split over a quadratic extension of $K$.

\vskip2mm

Over two-dimensional global fields, we also have the following finiteness result for groups of type $\textsf{G}_2$.

\begin{thm}\label{T:ResX4}{\rm (\cite[Theorem 9.1]{CRR-Spinor})}
Let $K$ be a two-dimensional global field of characteristic $\neq 2$ and $V$ be a divisorial set of places of $K$. The number of $K$-isomorphism classes of $K$-groups of type $\textsf{G}_2$ that have good reduction at all $v \in V$ is finite.
\end{thm}

These preceding results suggest that the proof of the following general fact should be within reach in the near future:
if $K$ is a two-dimensional global field, then for each type, there are only finitely many $K$-isomorphism classes of $K$-forms that split over a quadratic extension of $K$ and have good reduction at all discrete valuations in some divisorial set of places of $K$.


The very recent results on finiteness of unramified cohomology obtained in \cite[\S5]{RR-Tori}
make it possible to extend the above results beyond the class of two-dimensional global fields. We will just mention that Theorems \ref{T:ResX2} and \ref{T:ResX3} remain valid for a purely transcendental extension $K = k(x_1 , x_2)$ of transcendence degree two of a number field $k$, while Theorem \ref{T:ResX4} remains valid for $K$ a purely transcendental extension of a number field $k$ of any (finite) transcendence degree as well as the function field of a Severi-Brauer variety corresponding to a central simple algebra of either odd degree or degree 2 (in all cases $V$ can be any divisorial set of places of the field at hand).

\vskip2mm

To close this subsection, we will briefly survey the results on Conjecture 5.5 in the case where $K = k(C)$ is the function field of a smooth geometrically integral curve $C$ over a base field $k$ that satisfies conditions
$(\mathrm{F})$ or $(\mathrm{F}'_m)$ (see \S\ref{S-FConjCurves} for the definitions), and $V$ is the set of discrete valuations of $K$ associated with the closed points of $C$.

\begin{thm}\label{T:ResX5}
With notations as above, assume that $k$ is of characteristic $\neq 2$ and satisfies condition $(\mathrm{F}'_2)$. Then the number of $K$-isomorphism classes of

\vskip2mm

\noindent $\bullet$ \parbox[t]{16cm}{spinor groups $G = \mathrm{Spin}_n(q)$ of nondegenerate quadratic forms $q$ over $K$ of dimension $n \geq 5$,}

\vskip1mm

\noindent $\bullet$ \parbox[t]{16cm}{special unitary groups $G = \mathrm{SU}_n(L/K , h)$ of nondegenerate hermitian forms of dimension $n \geq 2$ over a quadratic extension $L/K$ with respect to its nontrivial automorphism,}

\vskip1mm

\noindent $\bullet$ \parbox[t]{16cm}{special unitary groups $G = \mathrm{SU}_n(D , h)$ of nondegenerate hermitian forms of dimension $n \geq 1$ over a central quaternion division algebra $D$ over $K$ with respect to the canonical involution,}

\vskip1mm

\noindent $\bullet$ \parbox[t]{16cm}{groups of type $\textsf{G}_2$}

\vskip2mm

\noindent that have good reduction at all $v \in V$ is finite.
\end{thm}

The proof proceeds along the same lines as the proofs of the Theorems \ref{T:ResX2}-\ref{T:ResX4} and relies on the finiteness results for unramified cohomology established in \cite{IRap}. We should point out that the list of groups in Theorem \ref{T:ResX5} was recently augmented by S.~Srinivasan \cite{Srimathy}. She proved the corresponding finiteness statement in the same situation as considered in the theorem for the universal covers $G = \widetilde{\mathrm{SU}}_n(D , h)$ of the special unitary groups $\mathrm{SU}_n(D , h)$, where $D$ is a central quaternion $K$-algebra and $h$ is a nondegenerate {\it skew-hermitian} form of dimension $n \geq 4$ with respect to the canonical involution. Recall that these groups are of type $\textsf{D}_n$, and in fact all simply connected $K$-groups of this type that split over a quadratic extension of $K$ are of the form $\widetilde{\mathrm{SU}}_n(D , h)$ (cf. \cite[Ch. 2, \S2.3.4]{Pl-R}). So, combining Theorem \ref{T:ResX5} with the main result of \cite{Srimathy}, we see that for each $r \geq 1$, there exist only finitely many $K$-isomorphism classes of absolutely almost simple simply connected algebraic $K$-groups $G$ of rank $r$ that belong to one of the types $\textsf{A}_n, \textsf{B}_n, \textsf{C}_n, \textsf{D}_n$ or $\textsf{G}_2$, split over a quadratic extension of $K$, and have good reduction at all $v \in V$. The types $\textsf{E}_6, \textsf{E}_7, \textsf{E}_8$, and $\textsf{F}_4$ have not been considered yet.

\vskip2mm

\subsection{Absolutely almost simple groups: Conjecture 6.1.}\label{S-ResultsHasse}
As we already remarked in \S\ref{S-OtherConj}, the truth of Conjecture 5.7 for the inner forms of an absolutely almost simple simply connected group $G$ automatically implies the properness of the global-to-local map
$$
\theta_{\overline{G} , V} \colon H^1(K , \overline{G}) \longrightarrow \prod_{v \in V} H^1(K_v , \overline{G})
$$
for the corresponding adjoint group $\overline{G}$. Thus, it follows from Theorem \ref{T:ResX1} that if $K$ is a finitely generated field equipped with a divisorial set of places $V$, then for
a central simple $K$-algebra $A$ of degree $n$ which is prime to $\mathrm{char}\: K$, the map
$$
\theta_{\mathrm{PSL}_{1 , A} , V} \colon H^1(K , \mathrm{PSL}_{1 , A}) \longrightarrow \prod_{v \in V} H^1(K_v , \mathrm{PSL}_{1 , A})
$$
is proper; in particular, the map $\theta_{\mathrm{PSL}_n , V} \colon H^1(K , \mathrm{PSL}_n) \longrightarrow \prod_{v \in V} H^1(K_v , \mathrm{PSL}_n)$ is proper\footnotemark. It should be pointed out, however, that the properness of the global-to-local map for the adjoint group does not automatically imply its properness for the corresponding simply connected group (or another isogenous group). In fact, the properness of the map global-to-local $\theta$  remains an open problem for the group $G = \mathrm{SL}_{1 , A}$ in the general case. We have the following partial result over two-dimensional global fields.

\footnotetext{Using twisting, one shows that the properness of $\theta_{\mathrm{PSL}_n , V}$ in fact implies the properness of $\theta_{\mathrm{PSL}_{1 , A} , V}$ for any central simple $K$-algebra $A$ of degree $n$.}

\begin{thm}\label{T:ResX6}{\rm (\cite[Theorem 5.7]{CRR-Spinor})}
Let $K$ be a two-dimensional global field, $V$ a divisorial set of places of $K$, and $n$ a square-free integer that is prime to $\mathrm{char}\: K$. Then for a central simple $K$-algebra $A$ of degree $n$ and $G = \mathrm{SL}_{1 , A}$, the map
$$
\theta_{G , V} \colon H^1(K , G) \longrightarrow \prod_{v \in V} H^1(K_v , G)
$$
is proper.
\end{thm}

Next, for {\it odd} integers $n \geq 5$, the adjoint group for $G = \mathrm{Spin}_n(q)$ is $\overline{G} = \mathrm{SO}_n(q)$. So, in this case, Theorem \ref{T:ResX2} automatically implies the properness of the map
$$
\theta_{\mathrm{SO}_n(q) , V} \colon H^1(K , \mathrm{SO}_n(q)) \longrightarrow \prod_{v \in V} H^1(K_v , \mathrm{SO}_n(q))
$$
when $K$ is a two-dimensional global field of characteristic $\neq 2$ and $V$ is a divisorial set of places of $K$.
In fact, one proves that this map is proper for {\it all} $n \geq 5$ --- see \cite[Theorem 1.3]{CRR-Spinor}. Moreover, still assuming that $K$ is a two-dimensional global field and $V$ is a divisorial set of places of $K$, we establish in \cite{CRR-Spinor} that this map is proper for
special unitary groups $\mathrm{SU}_n(L/K , h)$ and $\mathrm{SU}_n(D , h)$ over quadratic extensions and quaternion algebras, and for groups of type $\textsf{G}_2$.
In addition, in \cite{RR-Tori}, we show that the map $\theta_{G , V}$ is proper for the groups $\mathrm{SO}_n(q)$, $\mathrm{SU}_n(L/K , h)$, and $\mathrm{SU}_n(D , h)$ in the previous notations when $K$ is a purely transcendental extension $k(x_1 , x_2)$ of transcendence degree two of a number field $k$ and $V$ is any divisorial set of places of $K$. The next result, also obtained in \cite{RR-Tori}, deals with  purely transcendental extensions of number fields of {\it any} (finite) transcendence degree.

\begin{thm}\label{T:ResX7}
Let $k$ be a number field and suppose $K = k(x_1, \ldots , x_r)$ is a purely transcendental extension of $k$ or $K = k(X)$ is the function field of a
Severi-Brauer variety $X$ over $k$ associated with a central division algebra $D$ over $k$ of degree $\ell$, and let $V$ be a divisorial set of places of $K$.
Then in each of the following situations

\vskip2mm

\noindent $\bullet$ \parbox[t]{16cm}{$G = \mathrm{SL}_{1 , A}$ where $A$ is a central simple $K$-algebra of a square-free degree $m$ such that $k$ contains a primitive $m$-th root of unity, and either $m$ is relatively prime to $\ell$ or $\ell$ is a prime number if $K$ is the function field of a Severi-Brauer variety;}

\vskip1mm

\noindent $\bullet$ \parbox[t]{16cm}{$G$ is a simple algebraic group of type $\textsf{G}_2$ and either $\ell$ is odd or $\ell = 2$ if $K$ is the function field of a Severi-Brauer variety,}

\vskip2mm

\noindent the global-to-local map $\theta_{G , V} \colon H^1(K , G) \to \prod_{v \in V} H^1(K_v , G)$ is proper.
\end{thm}

\vskip2mm

We conclude with the following theorem, which collects available results on the properness of the global-to-local map over function fields of curves. The proofs are based on a combination of the arguments used to establish properness in \cite{CRR-Spinor}, together with the finiteness results for unramified cohomology obtained in \cite{IRap}.

\begin{thm}\label{T:ResX8}
Let $K = k(C)$ be the function field of a smooth geometrically integral curve $C$ over a field $k$, and let $V$ be the set of discrete valuations
of $K$ associated with the closed points of $C$. In each of the following situations:

\vskip2mm

\noindent $\bullet$ \parbox[t]{16cm}{$G = \mathrm{SL}_{1 , A}$, where $A$ is a central simple $K$-algebra of a square-free degree $n$ prime to
$\mathrm{char}\: k$ and such that $k$ satisfies condition $(\mathrm{F}'_n)$;}

\vskip2mm

\noindent and assuming that $k$ is of characteristic $\neq 2$ and satisfies $(\mathrm{F}'_2)$

\vskip2mm

\noindent $\bullet$ \parbox[t]{16cm}{$G = \mathrm{SO}_n(q)$, where $q$ is a nondegenerate quadratic form over $K$ of dimension $n \geq 5$;}

\vskip1mm

\noindent $\bullet$ \parbox[t]{16cm}{$G = \mathrm{SU}_n(L/K , h)$, where $h$ is nondegenerate hermitian form of dimension $n \geq 2$ over a quadratic
extension $L/K$ with respect to its nontrivial automorphism;}

\vskip1mm

\noindent $\bullet$ \parbox[t]{16cm}{$G = \mathrm{SU}_n(D , h)$, where $h$ is a nondegenerate hermitian form of dimension $n \geq 1$ over a central
quaternion division $K$-algebra with respect to the canonical involution'}

\vskip1mm

\noindent $\bullet$ \parbox[t]{16cm}{$G$ is of type $\textsf{G}_2$}

\vskip2mm

\noindent the global-to-local map $\theta_{G , V} \colon H^1(K , G) \to \prod_{v \in V} H^1(K_v , G)$ is proper.
\end{thm}

\vskip3mm

\subsection{Condition (T)}\label{S-ResultsT} As mentioned previously, we have been able to establish Condition (T) for algebraic tori in all situations --- see Theorem \ref{T:Res3} for the precise statement. However, for general reductive groups, the analysis of Condition (T) is only unfolding. We begin with the following statement over function fields of curves.

\begin{thm}\label{T:ResT1}{\rm (\cite[Theorem 4.1]{CRR-Israel})}
Let $K = k(C)$ be the function field of a smooth geometrically integral affine $C$ curve over a finitely generated field $k$, and
let $V$ be the set of discrete valuations of $K$ associated with closed points of $C$. Then Condition {\rm (T)} with respect to $V$ holds for any connected reductive split $K$-group $G$.
\end{thm}

Using considerations involving strong approximation, the proof of this result essentially reduces to verifying Condition (T) for a maximal split torus of $G$, where it follows from the finite generation of the Picard group $\mathrm{Pic}(K , V)$ (note that $V$ can be included in a divisorial set of places of $K$). We should point out that questions about strong approximation over fields other than global are interesting in their own right, and we will briefly comment on the initial steps in their study in \S\ref{S-Strong}.


In the higher-dimensional situation, the investigation of Condition (T) is expected to be a very challenging problem. Indeed, as we will see shortly, already in the case $G = \mathrm{GL}_n$, it is related to a famous conjecture of H.~Bass \cite{Bass}, which has not seen much progress since it was posed in 1972. Now, it makes sense to consider this problem first in the more general context of commutative algebra without any finite generation assumptions. So, let $R$ be a noetherian integral domain that is integrally closed in its field of fractions $K$. We denote by $\mathrm{P}$ the set of height one primes of $R$ and let $V$ be the associated set of discrete
valuations of $K$. As in \S\ref{S-OtherConj}, we set $G(\mathbb{A}(K , V))$ to be the corresponding adele group, and let $G(\mathbb{A}^{\infty}(K , V))$ and $G(K)$ denote the subgroups of integral and principal adeles, respectively. We first observe that for $G = \mathrm{GL}_n$ the class set
$$
\mathrm{cl}(G, K, V) = G(\mathbb{A}^{\infty}(K , V)) \backslash G(\mathbb{A}(K , V)) / G(K)
$$
has the following interpretation in terms of {\it reflexive} $R$-modules. Given an $R$-module $M$, we let $M^* = \mathrm{Hom}_R(M , R)$ denote the dual module. Then there is a natural homomorphism of $R$-modules $M \to M^{**} = (M^*)^*$, and $M$ is called reflexive if this homomorphism is an isomorphism. An $R$-submodule $M$ of $W = K^n$ is called a {\it lattice} if it is finitely generated and contains a $K$-basis of $W$. We let $\textsf{Refl}_n(R)$ (resp., $\textsf{Proj}_n(R)$) denote the set of isomorphism classes of lattices in $W = K^n$ that are reflexive (resp., projective) $R$-modules. One proves that {\it there is a natural bijection between the class set $\mathrm{cl}(\mathrm{GL}_n, K, V)$ and} $\textsf{Refl}_n(R)$ --- see \cite[Proposition B.2]{CRR-Israel}. (To be more precise, one actually constructs a bijection between the class set defined using {\it rational} adeles and $\textsf{Refl}_n(R)$, but since $\mathrm{GL}_n$ has weak approximation with respect to any finite set of places, the full adeles and the rational adeles result in the same class set.) We note that in general, $\textsf{Proj}_n(R)$ is a proper subset of $\textsf{Refl}_n(R)$, but it is known that if $R$ is a regular integral domain of Krull dimension $\leq 2$, then $\textsf{Refl}_n(R) = \textsf{Proj}_n(R)$ for all $n \geq 1$ (cf. \cite[Proposition 2]{Sam} and \cite[Corollary 6, p. 78]{Serre-LocAlg}).


\vskip2mm

Next, let $K_0(R)$ be the Grothendieck group of $R$ --- cf., for example, \cite[\S1]{MilnorKTheory} for the definition and basic properties. The following conjecture was proposed by Bass (\cite[\S9.1]{Bass-Conj}).


\vskip2mm

\noindent {\bf Conjecture 7.15} (Bass) {\it Let $R$ be a finitely generated $\mathbb{Z}$-algebra which is a regular ring. Then the group $K_0(R)$ is finitely generated.}

\vskip2mm

\noindent (We note that Bass actually conjectured the finite generation of all groups $K_n(R)$ $(n \geq 0)$ for such $R$.)

\vskip2mm

\addtocounter{thm}{1}

The following statement reveals a rather surprising connection between Conjecture 7.15 and Condition (T).
\begin{prop}\label{P:ResX1}{\rm (\cite[Corollary 6.16]{CRR-Israel})}
Let $R$ be an integral domain which is a finitely generated $\mathbb{Z}$-algebra and a regular ring of Krull dimension $\leq 2$, and let $V$ be the set of discrete
valuations of the fraction field $K$ associated with the height one prime ideals of $R$. If Conjecture 7.15 is true then $G = \mathrm{GL}_n$ for $n \geq 3$ satisfies Condition (T) with respect to $V$.
\end{prop}

The proof relies on the equality $\textsf{Refl}_n(R) = \textsf{Proj}_n(R)$ in the situation at hand, in conjunction with the following Cancellation Theorem due to Bass \cite{Bass}: {\it Let $R$ be a noetherian commutative ring of Krull dimension $d < \infty$, and let $P$ and $Q$ be finitely generated projective $R$-modules of constant rank $r > d$. If $P \oplus F \simeq Q \oplus F$, with $F$ free and finitely generated, then $P \simeq Q$.}

\vskip2mm

As we observed in \cite[\S3]{CRR-Israel}, Condition (T) can be used to show the finiteness of certain subgroups of unramified cohomology in degree 3 with $\mu_2$-coefficients (which is needed, in particular, for the analysis of groups of type $\mathsf{G}_2$ with good reduction). However, to implement this approach, we need Condition (T) to hold not for $\mathrm{GL}_n$, but rather for its $K$-forms $\mathrm{GL}_{1 , A}$, where $A$ is a central simple $K$-algebra of degree $n$. For this, we developed in \cite[Appendix C]{CRR-Israel} a descent procedure that, under some additional assumptions, enables one to derive Condition (T) for $\mathrm{GL}_{1 , A}$ from the fact that Condition (T) holds for $\mathrm{GL}_n$ over a suitable finite Galois extension $L/K$ that splits $A$.

While the investigation of Condition (T) is still in its initial stages, the range of its potential applications as well as its connections to various other problems, make this a very natural avenue for future work.


\section{Applications to the genus problem}\label{S-AppGenus}

After surveying the available results on Conjectures 5.5, 5.7, and 6.1, we now turn to applications. In this section, we will consider applications of the Main Conjecture 5.7 to the genus problem for absolutely almost simple algebraic groups, while in the next one, we will relate it to the analysis of weakly
commensurable Zariski-dense subgroups of such groups, which in turn is linked to questions in differential geometry about isospectral and length-commensurable locally symmetric spaces. Historically,
many of these developments can be traced back to \cite{PrRap-WC}: although this work focused primarily on geometric problems, it became apparent that some of the ideas introduced therein should be
considered in a much more general context. Eventually, research in this direction has led to the
Main Conjecture. So, the reader should regard \S\S\ref{S-AppGenus} and \ref{S-WCZariski} of this article
as an overview of the
problems that motivated the Main Conjecture, and for which it provides a uniform approach (as it
does with regard to some other issues, such as the local-global principle).

\vskip2mm

\subsection{The genus problem for algebraic groups}\label{S-GenusAG} Given two reductive algebraic groups $G_1$ and $G_2$ over a field $K$, we say that $G_1$ and $G_2$ {\it have the same isomorphism classes of maximal $K$-tori} if every maximal $K$-torus $T_1$ of $G_1$ is $K$-isomorphic to some maximal $K$-torus $T_2$ of $G_2$, and vice versa.
\vskip2mm

\noindent {\bf Definition 8.1.} {\it Let $G$ be an absolutely simple simply connected algebraic group over a field $K$. The {\it genus} $\mathbf{gen}_K(G)$ of $G$ is the set of $K$-isomorphism classes of (inner) $K$-forms $G'$ of $G$ that have the same isomorphism classes of maximal $K$-tori as $G$.}

\vskip2mm

\noindent (We note that if $K$ is a finitely generated field, then any $K$-form $G'$ of $G$ that has the same isomorphism classes of maximal $K$-tori as $G$ is necessarily inner --- cf. \cite[Lemma 5.2]{PrRap-Kunming}. So, restricting ourselves to just inner forms in this context does not result in a loss of generality.)

\vskip2mm

In the most general terms, the goal of the {\it genus problem} is to characterize the genus of a given group, which is of course crucial for understanding how two (absolutely almost simple simply connected) algebraic groups are related given the fact that they have the same isomorphism classes of maximal $K$-tori. From a variety of more precise questions that one can ask in connection with the genus problem, we will focus on the following two.

\vskip2mm

\noindent {\bf Question 8.2.} {\it When does $\mathbf{gen}_K(G)$ reduce to a single element?}

\vskip1mm

\noindent {\bf Question 8.3.} {\it When is $\mathbf{gen}_K(G)$ finite?}

\vskip2mm

The basic case where $K$ is a number field was considered in \cite[Theorem 7.5]{PrRap-WC}, where the following result was established (although the term ``genus," which appeared later, was not used).

\addtocounter{thm}{3}

\begin{thm}\label{T:A1}
Let $G$ be an absolutely almost simple simply connected algebraic group over a number field $K$. Then

\vskip1.5mm

\noindent {\rm (1)} $\mathbf{gen}_K(G)$ is finite;

\vskip1.5mm

\noindent {\rm (2)} if $G$ is not of type $\textsf{A}_n$, $\textsf{D}_{2n+1}$ $(n > 1)$, or $\textsf{E}_6$, we have $\vert \mathbf{gen}_K(G) \vert = 1$.
\end{thm}

A noteworthy feature here is the completely different behavior of the groups of type $\textsf{D}_n$ for $n$ even and odd. This difference was worked out in \cite{PrRap-CMH} in the context of algebras with involution and in \cite{Gar} in the context of algebraic groups. Another observation is that the types excluded in part (2)
are precisely the types for which the automorphism of multiplication by $(-1)$ of the corresponding root system is not in the Weyl group of the root system. In fact, these types are honest exceptions: indeed, it follows from \cite[Theorem 9.12]{PrRap-WC}
that the genus for each of those types can be arbitrarily large.

\vskip2mm

Having addressed number fields, the next question is what can one expect regarding $\mathbf{gen}_K(G)$ over more general fields?
In order to provide some context for the conjecture that we will formulate in \S\ref{S-GenusConj} and also to motivate the genus problem in general, we will next briefly review the genus problem for division algebras (see \cite{CRR2} for a more detailed account).

\subsection{The genus problem for division algebras}\label{S-GenusDivAlg} Let $D_1$ and $D_2$ be two central division $K$-algebras of degree $n$. We say that $D_1$ and $D_2$ {\it have the same maximal subfields} if a degree $n$ field extension $P/K$ admits a $K$-embedding $P \hookrightarrow D_1$ if and only if it admits a $K$-embedding $P \hookrightarrow D_2$. Then one can ask the following natural question:

\vskip2mm

\noindent $(*)$ \parbox[t]{16cm}{\it Let $D_1$ and $D_2$ be central division algebras of the same degree. How are they related given the fact that they have the same maximal subfields?}

\vskip2mm

\noindent This question can be seen as an extension of the following famous theorem of Amitsur \cite{Amitsur}:

\vskip2mm

\begin{thm}\label{T:Amitsur}{\rm (Amitsur)} Let $D_1$ and $D_2$ be finite-dimensional central division algebras over a field $K$ that have the same \emph{splitting fields}, i.e for a field extension $F$, the algebra $D_1 \otimes_K F$ is $F$-isomorphic to
a matrix algebra $\mathrm{M}_{n_1}(F)$ if and only if the algebra $D_2 \otimes_K F$ is isomorphic to a matrix algebra $\mathrm{M}_{n_2}(F)$. Then $n_1 = n_2$ and the classes $[D_1]$ and $[D_2]$ in the Brauer group $\mathrm{Br}(K)$ generated the same subgroup, $\langle [D_1] \rangle = \langle [D_2] \rangle$.

\end{thm}

The important point is that the proof of this result relies in a very essential way on
{\it infinite} (non-algebraic) extensions of $K$ --- namely, so-called {\it generic splitting fields} (concrete examples of which are function fields of Severi-Brauer varieties). So, one may wonder if it is possible to prove Amitsur's Theorem, or perhaps another statement along the same lines, using only {\it finite} extensions of $K$. In other words, is it enough to assume only that $D_1$ and $D_2$ have the same {\it finite-dimensional} splitting fields or just the same maximal subfields? It turns out that the conclusion of Amitsur's Theorem is {\it false} in this setting. In fact, using the Albert-Brauer-Hasse-Noether theorem (see \cite[Ch. 18, \S18.4]{Pierce}), one can easily construct arbitrary large collections of pairwise non-isomorphic cubic division algebras having the same maximal subfields over number fields (the same construction actually works for division algebras of any degree $d > 2$ --- cf. \cite[\S1]{CRR2}).
On the other hand, two quaternion division algebras over a number field that have the same quadratic subfields are necessarily isomorphic (as we will see in \S\ref{S-Quat} below, this fact turns out to have important consequences for Riemann surfaces). Thus, even over number fields, the question $(*)$ appears to be interesting. Moreover, until about 10 years ago, no information at all was available on $(*)$ over any fields other than global.
The following question along these lines was first asked in \cite[Remark 5.4]{PrRap-WC}:

\vskip2mm

\centerline{\parbox[t]{16cm}{\it Are quaternion division algebras over $\mathbb{Q}(x)$ determined uniquely up to isomorphism by their maximal subfields?}}

\vskip2mm

\noindent Shortly after it was formulated, this question was answered in the affirmative by D.~Saltman. In subsequent work, he and S.~Garibaldi \cite{GS} showed that the answer is still affirmative over the field of rational functions $k(x)$, where $k$ is any number field, and also in some other situations. This was the starting point of the investigation of question $(*)$ over fields more general than global, and we will now present the results that have been obtained since then. We note that a similar question, formulated in terms of finite-dimensional splitting fields, was considered in \cite{KrashMc}.

\vskip2mm

For our discussion, it will be convenient to quantify the problem by introducing the notion of the {\it genus} of a division algebra (this terminology was suggested by L.H.~Rowen).


\vskip2mm

\noindent {\bf Definition 8.6.} {\it Let $D$ be a finite-dimensional central division algebra over a field $K$. Then the genus $\mathbf{gen}(D)$ of $D$ is defined to be the set of classes $[D'] \in \mathrm{Br}(K)$ represented by central division $K$-algebras $D'$ having the same maximal subfields as $D$.}

\vskip2mm

Let us remark that Definition 8.1 in the previous subsection is a straightforward generalization of this definition (which historically was given earlier) to algebraic groups, in which maximal subfields are replaced with maximal tori. Moreover, just as in the case of algebraic groups, the general question $(*)$ essentially reduces to the analysis of the genus $\mathbf{gen}(D)$. In the present situation, we would like to focus on the analogues of Questions 8.2 and 8.3 in this situation, i.e.

\vskip2mm

\noindent $\bullet$ \parbox[t]{16cm}{{\it When does $\mathbf{gen}(D)$ reduce to a single element?} (Note that this is the case if and only if $D$ is determined uniquely up to isomorphism by its maximal subfields.)}

\vskip1mm

\noindent $\bullet$ \parbox[t]{16cm}{{\it When is $\mathbf{gen}(D)$ finite?}}

\vskip2mm

Over a number field $K$, the description of the Brauer group $\mathrm{Br}(K)$ provided by the Albert-Brauer-Hasse-Noether Theorem enables one to resolve both questions. Namely, it turns out that the genus of every quaternion division algebra is trivial (i.e., reduces to a single element), while the genus of any division algebra of higher degree is nontrivial but always finite (see \cite[Proposition 3.1]{CRR2} for the details).

Next,
the following theorem for the field of rational functions was established in \cite{RR}.

\addtocounter{thm}{1}
\begin{thm}\label{T:A2}
{\rm (Stability Theorem)}
Assume that $\mathrm{char}\: k \neq 2$. If  $\vert \mathbf{gen}(\Delta) \vert = 1$  for any central division quaternion algebra $\Delta$ over $k$,
then  $\vert \mathbf{gen}(D) \vert = 1$  for any quaternion algebra $D$ over $k(x)$.
\end{thm}

\vskip1mm

\noindent We note that the same statement remains valid for all division algebras having exponent two in the Brauer group (cf. \cite{CRR01}). On the other hand, $\vert \mathbf{gen}(D) \vert > 1$ whenever $D$ does not have exponent two since in that case, the opposite algebra $D^{\small \mathrm{op}}$ is not isomorphic to $D$, but clearly has the same maximal subfields as $D$. Now, a consequence of Theorem \ref{T:A2} is that the genus of a quaternion algebra over the purely transcendental extension $k(x_1, \ldots , x_r)$ of a number field $k$ of any (finite) transcendence degree reduces to a single element. At the same time, the following question remains open.


\vskip2mm

\noindent {\bf Question 8.8.} {\it Does there exist a central quaternion division algebra $D$ over a finitely generated field $K$ of characteristic $\neq 2$ having nontrivial genus?}

\vskip2mm

Turning now to the question of the finiteness of the genus, we first should point out that over general fields, the genus $\mathbf{gen}(D)$ can be infinite. Indeed, adapting a construction that has been suggested by a number of people, including M.~Schacher, A.~Wadsworth, M.~Rost, S.~Garibaldi and D.~Saltman, J.~Meyer \cite{Meyer} produced examples of quaternion algebras over ``large" fields with infinite genus\footnotemark. \footnotetext{We observe that if the genus $\mathbf{gen}(D)$ is infinite for a central division $K$-algebra $D$, then the genus $\mathbf{gen}_K(G)$ is also infinite for the corresponding algebraic group $G = \mathrm{SL}_{1 , D}$.}
(By construction, these fields have infinite transcendence degree over the prime subfield.) In particular, there exist quaternion algebras over such fields with nontrivial genus (this was actually already observed by Garibaldi and Saltman \cite{GS}), indicating that the finite generation assumption in Question 8.8 cannot be omitted. Subsequently, S.~Tikhonov \cite{Tikh} extended this approach to construct examples of division algebras of any prime degree having infinite genus. On the other hand, in  \cite{CRR3} and \cite{CRR-Israel},
we proved the following.


\addtocounter{thm}{1}

\begin{thm}\label{T:A3}
Let $K$ be a finitely generated field. Then for any finite-dimensional central division $K$-algebra $D$, the genus $\mathbf{gen}(D)$ is finite.
\end{thm}

There are two versions of the proof of Theorem \ref{T:A3}. Both of them rely on the analysis of ramification, but they differ in the amount of information about the unramified Brauer group that is needed as an input. The proof given in \cite{CRR3} requires the additional assumption that the degree $n$ of the division algebra $D$ is relatively prime to $\mathrm{char}\: K$, and proceeds along the following lines. Fix a divisorial set of places $V$ of $K$.
First, since $\mathrm{char}\: K$ is prime to $n$, we can assume without loss of generality that for each $v \in V$, the characteristic of the residue field $K^{(v)}$ is prime to $n$.  Consequently, the residue map
$$
\partial^2_v \colon H^2(K , \mu_n) \longrightarrow H^1(K^{(v)} , \mathbb{Z}/n\mathbb{Z})
$$
that we encountered in our discussion of the proof of Theorem \ref{T:ResX1} is defined. One then shows that if $D'$ is a central division algebra with $[D'] \in \mathbf{gen}(D)$ and
\begin{equation}\label{E:A1}
\chi_v = \partial^2_v([D]) \ \ \text{and} \ \ \chi'_v = \partial^2_v([D'])
\end{equation}
are the corresponding characters of the absolute Galois group of $K^{(v)}$, then $\ker \chi_v = \ker \chi'_v$ (see \cite[Lemma 2.5]{CRR01}). In particular, we see that for $D$ and $D'$ as above, the algebras are either simultaneously ramified or simultaneously unramified at a place $v$.
This fact leads to the following estimate:
$$
\vert \gen(D) \vert \leq \vert {}_n\mathrm{Br}(K)_V \vert \cdot \varphi(n)^r,
$$
where $r$ is the number of $v \in V$ where $D$ ramifies (which is necessarily finite for a divisorial set). Thus, we obtain an upper bound on the size of the genus that is uniform over all central division $K$-algebras of a given degree with a fixed number of ramification places; in particular, this estimate is uniform over all quaternion division $K$-algebras.

Our second proof of Theorem \ref{T:A3}, which we gave in \cite{CRR-Israel}, also uses the analysis of ramification, but avoids imposing restrictions on the characteristic of the field $K$. The reason for this is that the argument does not require the finiteness of the ($n$-torsion of the) unramified Brauer group, but only the finiteness of certain subgroups of the latter. On the other hand, since these subgroups depend on the division algebra at hand, we do not obtain a nice estimate on the size of the genus as
provided by our first proof.


To finish up this discussion, let us mention that the proof of Theorem \ref{T:A2} applies similar considerations to the set $V$ of geometric places of the field $K = k(x)$ (i.e., the set of those $v$ that correspond to the closed points of $\mathbb{P}^1_k$). Namely, let $D$ be a quaternion division algebra over $K$ and let $D'$ be another quaternion division $K$-algebra such $[D'] \in \mathbf{gen}(D)$. As above, for each $v \in V$, letting $\chi_v$ and $\chi'_v$ denote the characters defined by (\ref{E:A1}), we have $\ker \chi_v = \ker \chi'_v$. But since $n = 2$, this means that actually $\chi_v = \chi'_v$. It follows that the class $[D] \cdot [D']^{-1}$ in ${}_2\mathrm{Br}(K)$ is unramified at all $v \in V$. But according to a result of Faddeev (cf. \cite[Theorem 6.9.1]{Gille}), we have
$$
{}_2\mathrm{Br}(K)_V = {}_2\Br(k),
$$
implying that $[D'] = [D] \cdot [\Delta \otimes_k K]$ for some quaternion algebra $\Delta$ over $k$. Finally, using the assumption that the genus of every quaternion algebra over $k$ is trivial and applying a specialization argument, one concludes that the class $[\Delta] \in \mathrm{Br}(k)$ is trivial, hence $D = D'$.

\subsection{Quaternion algebras and Riemann surfaces}\label{S-Quat} In this subsection, we will briefly describe how quaternion algebras sharing ``many" (although a priori not all) quadratic subfields arise in the investigation of Riemann surfaces. These considerations, which provide a geometric context for our general discussion, will be extended in \S\ref{S-WCZariski} to locally symmetric spaces of arbitrary simple real algebraic groups.

\vskip2mm

Let $\mathbb{H} = \{ x +iy \in \mathbb{C} \, \vert \, y > 0 \}$ be the complex upper half-plane equipped with the standard hyperbolic metric $ds^2 = y^{-2}(dx^2 + dy^2)$. The action of $\mathrm{SL}_2(\mathbb{R})$ by fractional linear transformations is transitive and isometric, allowing us to identify $\mathbb{H}$ with the homogeneous (symmetric) space $\mathrm{SL}_2(\mathbb{R}) / \mathrm{SO}_2(\mathbb{R})$. Let $\pi \colon \mathrm{SL}_2(\mathbb{R}) \to \mathrm{PSL}_2(\mathbb{R})$ be the canonical projection. It is well-known (cf., for example, \cite[Theorem 27.12]{Forster}) that any compact Riemann surface of genus $> 1$ can be presented as a quotient $\Gamma \backslash \mathbb{H}$ by some discrete subgroup $\Gamma \subset \mathrm{SL}_2(\mathbb{R})$ containing $\{ \pm I \}$ and having torsion-free image $\pi(\Gamma)$. It was demonstrated in \cite{McLReid} that some properties of $M$ can be understood in terms of the associated quaternion algebra $A_{\Gamma}$, which is constructed as follows.

\vskip2mm

Let $\Gamma^{(2)}$ denote the subgroup of $\Gamma$ generated by the squares of all elements,
and set $A_{\Gamma}$ to be the $\mathbb{Q}$-subalgebra of $\mathrm{M}_2(\mathbb{R})$ generated by $\Gamma^{(2)}$. One shows that $A_{\Gamma}$ is a quaternion algebra (not necessarily a {\it division} algebra) with center
$$
K_{\Gamma} = \mathbb{Q}( \mathrm{tr}\: \gamma \, \vert \, \gamma \in \Gamma^{(2)} )
$$
(the so-called {\it trace field}) --- cf. \cite[Ch. 3]{McLReid}. Furthermore, it turns out that if $\Gamma_1$ and $\Gamma_2$ are {\it commensurable} (i.e., their intersection has finite index in both of them), then $A_{\Gamma_1} = A_{\Gamma_2}$; in other words, $A_{\Gamma}$ is an invariant of the commensurability class of $\Gamma$. Moreover, if $\Gamma$ is an {\it arithmetic} Fuchsian group, then $K_{\Gamma}$ is a number field and $A_{\Gamma}$ is {\it the} quaternion algebra involved in the description of $\Gamma$ (cf. \cite[\S8.2]{McLReid}). It follows that if $\Gamma_1$ and $\Gamma_2$ are arithmetic subgroups and the algebras $A_{\Gamma_1}$ and $A_{\Gamma_2}$ are isomorphic, then $\Gamma_1$ is commensurable with a conjugate of $\Gamma_2$. Thus, in the arithmetic case, $A_{\Gamma}$ completely determines the commensurability class of $\Gamma$ (up to conjugation). This is no longer true for non-arithmetic subgroups, but nevertheless $A_{\Gamma}$ remains an important invariant of the commensurability class.

\vskip2mm

Next, in differential geometry, one attaches various {\it spectra} to a Riemannian manifold $M$: particularly when $M$ is compact, one considers the Laplace spectrum $E(M)$, which consists of the eigenvalues of the Beltrami-Laplace operator with multiplicities; in the general case, one can also look at the (weak) length spectrum $L(M)$, which is defined as the set of lengths of all closed geodesics in $M$. Then two Riemannian manifolds $M_1$ and $M_2$ are said to be

\vskip2mm

\noindent (1) {\it isospectral} if $E(M_1) = E(M_2)$ (assuming that $M_1$ and $M_2$ are compact);

\vskip1mm

\noindent (2) {\it iso-length-spectral} if $L(M_1) = L(M_2)$;

\vskip1mm

\noindent (3) {\it length-commensurable} if $\mathbb{Q} \cdot L(M_1) = \mathbb{Q} \cdot L(M_2)$.

\vskip2mm

\noindent On the other hand, $M_1$ and $M_2$ are called {\it commensurable} if they have a common finite-sheeted cover, with the covering maps being local isometries. In general, one would like to understand how $M_1$ and $M_2$ are related if they satisfy one of the above conditions (1)-(3) or similar ones. Probably the most famous version of this general question is due to M.~Kac \cite{Kac}, who asked ``Can one hear the shape of a drum?" In other words, are two compact isospectral Riemannian manifolds necessarily isometric? For our purposes, however, we will focus on length-commensurable Riemann surfaces and their commensurability. We refer the reader to \cite{PrRap-Kunming} for a more detailed discussion of these conditions, and only mention here that for compact locally symmetric spaces of simple real algebraic groups (in particular, for compact Riemann surfaces), condition (1) implies condition (2), which in turn trivially implies condition (3). What is interesting is that as far as questions of commensurability are concerned, condition (3), despite being the weakest, has essentially the same consequences as the strongest condition (1).

We will now examine how the length-commensurability of two compact Riemann surfaces of genus $> 1$ impacts the associated quaternion algebras. First, let $M = \Gamma \backslash \mathbb{H}$ be a compact Riemann surface as above. It is well-known that every closed geodesic in $M$ corresponds to a semi-simple element $\gamma \in \Gamma$, $\gamma \neq \pm I$, and we will denote it by $c_{\gamma}$. For our discussion, we will only need the following formula for the length $\ell(c_{\gamma})$ of $c_{\gamma}.$ Since $\Gamma \subset \mathrm{SL}_2(\mathbb{R})$ is discrete and $\pi(\Gamma) \subset \mathrm{PSL}_2(\mathbb{R})$ is torsion-free, any semi-simple element $\gamma \in \Gamma$ is automatically hyperbolic, hence is conjugate to a matrix of the form $\left(\begin{array}{cc} t_{\gamma} & 0 \\ 0 & t_{\gamma}^{-1} \end{array}  \right)$ with $t_{\gamma} \in \mathbb{R}$. Then
\begin{equation}\label{E-GeoLength}
\ell(c_{\gamma}) = \frac{2}{n_{\gamma}} \cdot \vert \log \vert t_{\gamma} \vert \vert,
\end{equation}
where $n_{\gamma}$ is an integer (in fact, it is the winding number --- we refer the reader to \cite[\S8]{PrRap-WC} for the details).
It follows that
\begin{equation}\label{E:A5}
\mathbb{Q} \cdot L(M) = \mathbb{Q} \cdot \{ \log \vert t_{\gamma} \vert \ \vert \ \gamma \in \Gamma \ \ \text{semi-simple and}  \ \neq \pm I \}.
\end{equation}
Now, let $M_1 = \Gamma_1 \backslash \mathbb{H}$ and $M_2 = \Gamma_2 \backslash \mathbb{H}$ be two compact Riemann surfaces as above, and assume that they are length-commensurable. One then shows that
$$
K_{\Gamma_1} = K_{\Gamma_2} =: K
$$
(this is a consequence of the more general Theorem 8.15 in \cite{PrRap-WC} --- see also Theorem \ref{T:AZD3} below).
Furthermore, it follows from (\ref{E:A5}) that for any semi-simple $\gamma_1 \in \Gamma_1^{(2)}$ different from $\pm I$, there exists a semi-simple $\gamma_2 \in \Gamma_2^{(2)}$ such that
\begin{equation}\label{E:A6}
t_{\gamma_1}^m = t_{\gamma_2}^n
\end{equation}
for some nonzero integers $m , n$, and consequently $\gamma_1^m$ and $\gamma_2^n \in \mathrm{M}_2(\mathbb{R})$ are conjugate. Then for $i = 1, 2$, the algebra $K[\gamma_i]$ is a maximal \'etale subalgebra of $A_{\Gamma_i}$, and we have an isomorphism of $K$-algebras
$$
K[\gamma_1] = K[\gamma_1^m] \simeq K[\gamma_2^n] = K[\gamma_2].
$$
Thus, the geometric condition of length-commensurability translates into the algebraic condition that {\it $A_{\Gamma_1}$ and $A_{\Gamma_2}$ have a common center and the same isomorphism classes of maximal \'etale subalgebras that have a nontrivial intersection with
$\Gamma_1^{(2)}$ and $\Gamma_2^{(2)}$, respectively.}

We recall from our discussion in \S\ref{S-GenusDivAlg} that the genus of a quaternion algebra over a number field reduces to a single element. Note, however, that the preceding condition implied by length-commensurability is technically weaker than the condition that $A_{\Gamma_1}$ and $A_{\Gamma_2}$ belong to the same genus. Nevertheless, it was observed by A.~Reid \cite{Reid} (prior to the systematic investigation of the genus problem) that if $M_1$ and $M_2$ are isospectral (hence also iso-length spectral) Riemann surfaces with arithmetic fundamental groups $\Gamma_1$ and $\Gamma_2$, respectively, then the latter is still sufficient to conclude that $A_{\Gamma_1} \simeq A_{\Gamma_2}$. As we remarked above, this implies that $\Gamma_1$ is commensurable to a conjugate of $\Gamma_2$, and hence the
Riemann surfaces $M_1$ and $M_2$ are commensurable. In \S\ref{S-WCZariski}, we will give a brief overview of similar results for arithmetically defined locally symmetric spaces of arbitrary real simple algebraic groups (cf. \cite{PrRap-WC}, \cite{PrRap-Kunming}). We will also see that even for not necessarily arithmetic Riemann surfaces, the above condition leads to the same finiteness results as we have for the genus.

\subsection{Conjectures and results on the genus problem.}\label{S-GenusConj} Juxtaposing the treatment of the genus of absolutely almost simple simply connected algebraic groups over number fields in Theorem \ref{T:A1} with the results in \S\ref{S-GenusDivAlg} on the genus of division algebras over general fields, one is led to the following.

\vskip2mm

\noindent {\bf Conjecture 8.10.}

\vskip1mm

\noindent (1) \parbox[t]{16cm}{{\it Let $K = k(x)$ be the field of rational functions in one variable over a number field $k$. If $G$ is an absolutely almost simple simply connected algebraic $K$-group with center $Z(G)$ of size $\leq 2$, then the genus $\mathbf{gen}_K(G)$ reduces to a single element.}}

\vskip1mm

\noindent (2) \parbox[t]{16cm}{{\it Let $G$ be an absolutely almost simple simply connected algebraic group over a finitely generated field of ``good" characteristic. Then the genus
$\mathbf{gen}_K(G)$ is finite.}}

\vskip2mm

\noindent (Here, ``good" characteristic is used in the same sense as in Conjecture 5.7.)


\vskip1mm

The general hope is that this conjecture will be proved through a far-reaching extension to algebraic groups of the techniques developed for the study of the genus of division algebras; in this extension, groups with good reduction are expected to play a role similar to that of unramified division algebras. More precisely, as we pointed out in our sketch of the first proof of Theorem \ref{T:A3}, conceptually, one of the critical observations needed to establish the finiteness of the genus is that if $D$ is a finite-dimensional central division $K$-algebra that is unramified at a discrete valuation $v$ of $K$, then every central division $K$-algebra $D'$ such that $[D'] \in \mathbf{gen}(D)$ is also unramified at $v$. While this fact is certainly nontrivial, at the same time, it is not particularly surprising, and can be proved by exploiting the equivalence of several different characterizations of unramified algebras. Informally speaking, it means that the maximal subfields of a division algebra detect whether or not the algebra is unramified. On the other hand, there were no indications in the literature as to why the maximal tori of a reductive group should be able to detect whether or not the group has good reduction in a sufficiently general situation. So, the following result from \cite{CRR4} and \cite{CRR-GR} is quite surprising.

\addtocounter{thm}{1}

\begin{thm}\label{T:A4}
Let $G$ be an absolutely almost simple simply connected algebraic group over a field $K$, and let $v$ be a discrete valuation of $K$. Assume that the residue field $K^{(v)}$ is finitely generated and that $G$ has good reduction at $v$. Then every $G' \in \mathbf{gen}_K(G)$ also has good reduction at $v$.
\end{thm}

Since the proof relies on techniques that we will not address in this article (specifically, considerations involving so-called {\it generic tori} --- cf. \cite[\S9]{PrRap-Thin} for an overview of these), we refer the reader to \cite{CRR-GR} for the details; here, we will only discuss the consequences of this statement for the genus problem. We would also like to draw the reader's attention to Theorem \ref{T:AZD4} below that
elaborates on Theorem \ref{T:A4}.



\vskip2mm

Suppose now that $G$ is an absolutely almost simple simply connected algebraic group defined over a finitely generated field $K$, and let $V$ be a divisorial set of places of $K$. Using the fact that $V$ satisfies condition (A) (see \S\ref{S-ConjFGfields}), it is easy to show that there exists a finite subset $S \subset V$ such that $G$ has good reduction at all $v \in V \setminus S$. Then it follows from Theorem \ref{T:A4} that every $K$-form $G'$ of $G$ whose $K$-isomorphism class lies in the genus $\mathbf{gen}_K(G)$ also has good reduction at all $v \in V \setminus S$. Since $V \setminus S$ contains a divisorial set of places of $K$, we conclude that the truth of Conjecture 5.7 would automatically imply the finiteness of $\mathbf{gen}_K(G)$ (at least if $\mathrm{char}\: K$ is ``good" for the type of $G$). We will now list some results on the genus in the spirit of Conjecture 8.10 that have already been established,  beginning with inner forms of type $\textsf{A}_n$, where the conjecture has been proved in full.

\begin{thm}\label{T:A5}  \ \
\vskip1mm

\noindent {\rm (1)} \parbox[t]{16cm}{Let $D$ be a central division algebra of exponent two over the field of rational functions $K = k(x_1, \ldots , x_r)$, where $k$ is
either a number field or a finite field of characteristic $\neq$ 2. Then for $G = \mathrm{SL}_{m , D}$ $(m \geq 1)$, the genus $\mathbf{gen}_K(G)$
reduces to a single element.}

\vskip1mm

\noindent {\rm (2)} \parbox[t]{16cm}{Let $G = \mathrm{SL}_{m , D}$, where $D$ is a central division algebra over a finitely generated field $K$ of degree prime to $\mathrm{char}\: K$. Then $\mathbf{gen}_K(G)$ is finite.}
\end{thm}

The discussion preceding the statement of Theorem \ref{T:A5} shows that part (2) follows from Theorem \ref{T:ResX1}.
On the other hand, we should point out that part (2) is {\it not} a direct consequence of Theorem \ref{T:A3} on the finiteness of $\gen(D)$, even for $m = 1$. The problem is that while every maximal  $K$-torus of the group $G = \mathrm{SL}_{1 , D}$,  with $D$ a central division $K$-algebra, is the norm torus $\mathrm{R}^{(1)}_{F/K}(\mathbb{G}_m)$ for some maximal separable subfield $F$ of $D$, the fact that two such tori $\mathrm{R}^{(1)}_{F_1/K}(\mathbb{G}_m)$ and $\mathrm{R}^{(1)}_{F_2/K}(\mathbb{G}_m)$ are $K$-isomorphic does not in general imply that the field extensions $F_1$ and $F_2$ are isomorphic over $K$. Thus, the fact that the algebraic $K$-groups $\mathrm{SL}_{1 , D_1}$ and $\mathrm{SL}_{1 , D_2}$, for some central division $K$-algebras $D_1$ and $D_2$, are in the same genus may {\it not} imply that the algebras $D_1$ and $D_2$ are themselves in the same genus. So, some additional considerations (involving generic tori) are needed to derive the finiteness of $\mathbf{gen}_K(G)$  from that of $\mathbf{gen}(D)$ (cf. \cite[Theorem 5.3]{CRR01}). Analogous considerations implemented in the context of the proof of the Stability Theorem (Theorem \ref{T:A2}) yield part (1)
of Theorem \ref{T:A5}. Since there are no restrictions on the characteristic in Theorem \ref{T:A3}, it would be interesting to determine if these are necessary in Theorem \ref{T:A5}(2).

\vskip2mm

Next, we will consider the genus of spinor groups.
\begin{thm}\label{T:A6} 
Suppose $K$ is either the field of rational functions $k(x , y)$ in two variables or the function field $k(C)$ of a smooth geometrically integral curve $C$ over $k$, where, in both cases, $k$ is a number field. Let $G = \mathrm{Spin}_n(q)$ be the spinor group of a nondegenerate quadratic form $q$ over $K$ of \emph{odd} dimension $n \geq 5$. Then $\mathbf{gen}_K(G)$ is finite.
\end{thm}

We note that the argument sketched prior to the statement of Theorem \ref{T:A4}, combined with Theorem \ref{T:ResX2} and \cite[Theorem 5.5]{RR-Tori},
shows that for $G = \mathrm{Spin}_n(q)$, where $q$ is a nondegenerate quadratic form over $K$ of \emph{any} dimension $n \geq 5$, the number of $K$-isomorphism classes of spinor groups $G' = \mathrm{Spin}_n(q')$ that have the same maximal $K$-tori as $G$ is finite. When $n$ is odd, all $K$-forms of $G$ are again spinor groups, and we obtain the above theorem. On the other hand, when $n$ is even, $G$ has $K$-forms coming from skew-hermitian forms over noncommutative central division algebras over $K$, and so far, we have not been able to eliminate these as potential members of the genus $\mathbf{gen}_K(G)$ (see, however, \cite[Theorem 1.2]{CRR-Spinor} for a partial result in this direction).
We also know that $\gen_K(G)$ is finite if $K$ is the same as in Theorem \ref{T:A6} and $G$ is either $\mathrm{SU}_n(L/K , h)$, where $h$ is a nondegenerate hermitian form of dimension $n \geq 2$ over a quadratic extension $L/K$, or $\mathrm{SU}_n(D , h)$, where $h$ is a nondegenerate hermitian form of dimension $n \geq 1$ over a central quaternion division algebra $D$ over $K$ with the canonical involution (the case where $K$ is a 2-dimensional global field is handled in \cite[Theorem 8.3 and Remark 8.6]{CRR-Spinor}; for the field of rational functions in two variables, one proceeds analogously, making use of \cite[Theorem 5.1]{RR-Tori}).


\vskip2mm

We conclude this section with the following result for groups of type $\textsf{G}_2$.
\begin{thm}\label{T:A7}
Let $G$ be a simple algebraic $K$-group of type $\textsf{G}_2$.

\vskip2mm

\noindent {\rm (1)} If $K$ is the field of rational functions $k(x)$, where $k$ is a number field, then $\vert \mathbf{gen}_K(G) \vert = 1$.

\vskip2mm

\noindent {\rm (2)} \parbox[t]{16cm}{If $k$ is a number field and $K$ is one of the following:

\vskip2mm

\hskip5mm $\bullet$ $K = k(x_1, \ldots, x_r)$ is the field of rational functions in any (finite) number of variables;

\vskip1mm

\hskip5mm $\bullet$ $K = k(C)$ is the function field of a smooth geometrically integral curve $C$ over $k$;

\vskip1mm

\hskip5mm $\bullet$ \parbox[t]{15cm}{$K = k(X)$ is the function field of a Severi-Brauer variety $X$ over $k$ associated with a central division algebra $D$ over $k$ of degree $\ell$, where $\ell$ is either odd or $\ell = 2$,}

\vskip2mm

\noindent then $\mathbf{gen}_K(G)$ is finite.}
\end{thm}

\vskip2mm

\noindent The proofs of these statements ultimately rely on an analysis of unramified cohomology, and are treated in detail in \cite[Theorems 9.1 and 9.3]{CRR-Spinor} and \cite[Proposition 5.3]{RR-Tori}.


\subsection{The genus and base change.}\label{S-GenusBaseChange} To conclude our overview of the genus problem, in this subsection, we will briefly discuss how the genus $\mathbf{gen}_K(G)$ varies under base change. As a first example, we observe that if $G = \mathrm{SL}_{1,D}$,  where $D$ is a cubic division algebra over a number field $K$, then, using the description of the Brauer group of a number field provided by Albert-Brauer-Hasse-Noether theorem,
one can construct a sequence of {\it finite} extensions $F_i/K$ so that the sizes of $\mathbf{gen}_{F_i}(G \times_K F_i)$ grow unboundedly (of course, this cannot happen if $D$ is a quaternion algebra). On the other hand, the following theorem from \cite{CRR-GR} shows that the genus cannot grow under {\it purely transcendental} extensions.

\begin{thm}\label{T:A8}
Let $G$ be an absolutely almost simple simply connected algebraic group over a finitely generated field $k$ of characteristic 0, and let $K = k(x)$ be the field of rational functions in one variable. Then any $G' \in \mathbf{gen}_K(G \times_k K)$ is of the form $G' = G'_0 \times_k K$ for some $G'_0 \in \mathbf{gen}_k(G)$.
\end{thm}

To sketch the idea of the proof, we let $V$ denote the set of geometric places of $K$ (i.e. those valuations that correspond to the closed points of $\mathbb{P}_k^1$). Then the group $G \times_k K$ has good reduction at all $v \in V$. So, it follows from Theorem \ref{T:A4} that $G'$ also has good reduction at all $v \in V$. The desired conclusion is then derived from Theorem \ref{T:XXX2}.

\vskip2mm

In conjunction with Theorem \ref{T:A1}, this yields the following.
\begin{cor}\label{C:A1}
Let $G$ be an absolutely almost simple simply connected algebraic group over a number field $k$, and let $K = k(x_1, \ldots , x_r)$ be the field of rational functions in $r \geq 1$ variables. Then the genus $\mathbf{gen}_K(G \times_k K)$ is finite, and in fact reduces to a single element if the type of $G$ is different from $\textsf{A}_n$, $\textsf{D}_{2n+1}$ $(n > 1)$, and $\textsf{E}_6$.
\end{cor}

We note that for the exceptions $\textsf{A}_n$, $\textsf{D}_{2n+1}$ $(n > 1)$, and $\textsf{E}_6$, we have $\vert \Z(G) \vert > 2$,
so this corollary proves Conjecture 8.10 in full for groups of the form $G \times_k K$ in the above notations.

\vskip2mm

It is important to point out that for $G'_0 \in \mathbf{gen}_k(G)$, the group $G' = G'_0 \times_k K$ may \emph{not} lie in $\mathbf{gen}_K(G \times_k K)$. In fact, we have the following two results from \cite{CRR-GR}, which point to a new phenomenon that we refer to as ``killing the genus"
by a purely transcendental extension of the base field.
\begin{thm}\label{T:A10}
Let $A$ be a central simple algebra of degree $n$ over a finitely generated field $k$, and let $G = \mathrm{SL}_{1 , A}$. Assume that $\mathrm{char}\: k$ is prime to $n$, and let $K = k(x_1, \ldots , x_{n-1})$ be the field of rational functions in $(n-1)$ variables. Then $\mathbf{gen}_K(G \times_k K)$ consists of
(the isomorphism classes of) groups of the form $H \times_k K$, where $H = \mathrm{SL}_{1 , B}$ and $B$ is a central simple $k$-algebra of degree $n$ such that its class $[B]$ in the Brauer group $\mathrm{Br}(k)$ generates the same subgroup as the class $[A]$.
\end{thm}

\begin{thm}\label{T:A11}
Let $G$ be a group of type $\textsf{G}_2$ over a finitely generated field $k$ of characteristic $\neq 2$, and let $K = k(x_1, \ldots , x_6)$ be the field of rational functions in 6 variables. Then $\mathbf{gen}_K(G \times_k K)$ reduces to a single element.
\end{thm}

We note that the genus of a group of type $\textsf{G}_2$ can be nontrivial (cf. \cite{GilleG2}). The proof of Theorem \ref{T:A10} uses Amitsur's Theorem (Theorem \ref{T:Amitsur})  on generic splitting fields, while the proof of Theorem \ref{T:A11} relies on properties of quadratic Pfister forms. These two results, and also similar ones for the genus of division algebras (\cite{CRR-GR}), point to the following conjecture.

\vskip2mm

\noindent {\bf Conjecture 8.19.} {\it Let $G$ be an absolutely almost simple simply connected algebraic group over a finitely generated field $k$, and assume that
$\mathrm{char}\: k$ is ``good" for the type of $G$. Then there is a purely transcendental extension $K = k(x_1, \ldots , x_r)$ of transcendence degree $r$ depending only on the type of $G$ such that every $G' \in \mathbf{gen}_K(G \times_k K)$ is of the form $G'_0 \times_k K$, where $G'_0$ has the property that
$$
G'_0 \times_k F \in \mathbf{gen}_F(G \times_k F)
$$
for \emph{any} field extension $F/k$.}

\vskip2mm

This can be reformulated using a different, more functorial, notion of the genus (proposed by A.S.~Merkurjev), which is also based on the consideration of maximal tori, but at the same time incorporates infinite extensions like those involved in Amitsur's theorem. Namely, one defines the {\it motivic genus} $\mathbf{gen}_m(G)$ of an absolutely almost simple simply connected algebraic $k$-group $G$ as the set of $k$-isomorphism  classes of $k$-forms $G'$ of $G$ such that
$$
G' \times_k F \in \mathbf{gen}_F(G \times_k F)
$$
for all field extensions $F/k$. Then according to Theorem \ref{T:A10}, the motivic genus of $G = \mathrm{SL}_{1 , A}$ is always finite and reduces to one element if $A$ has exponent two. Furthermore, by Theorem \ref{T:A11}, for a $k$-group of type $\textsf{G}_2$, the motivic genus always reduces to one element (at least when $\mathrm{char}\: k \neq 2$). Next, it was shown by Izhboldin \cite{Izhb} that for nondegenerate quadratic forms $q$ and $q'$ of \emph{odd} dimension over a field $k$ of characteristic $\neq 2$, the condition

\vskip2mm

\noindent $(\bullet)$ $q$ and $q'$ have the same Witt index over any field extension $F/k$

\vskip2mm

\noindent is equivalent to the fact that $q$ and $q'$ are scalar multiples of each other (this conclusion being {\it false} for even-dimensional forms). It follows that $\vert \mathbf{gen}_m(G) \vert = 1$ for the spinor group $G = \mathrm{Spin}_n(q)$ with $n$ odd. We note that the condition $(\bullet)$ amounts to the fact that the motives of $q$ and $q'$ in the category of Chow motives are isomorphic (cf. \cite{Vish1}, \cite[Theorem 4.18]{Vish2}, and \cite{Karp}), which prompted the choice of terminology for this version of the notion of the genus. One can expect the motivic genus to be finite for all absolutely almost simple simply connected groups (at least over fields of ``good" characteristic). On the other hand, Conjecture 8.19 asserts that the genus gets reduced to the motivic genus (i.e., becomes as small as possible) after a suitable purely transcendental extension of the base field.

\section{Applications to Zariski-dense subgroups}\label{S-WCZariski} The analysis of Zariski-dense (thin) subgroups of semi-simple algebraic groups is a very broad and active area (see, for example, \cite{Breu-Oh}). Our goal in this section is to give some indications of how reduction techniques, and particularly Conjecture 5.7, can be applied in this context. More specifically, we will
focus on the geometry of locally symmetric spaces, and at the end will demonstrate how this approach leads to a finiteness result for length-commensurable Riemann surfaces without any assumptions of arithmeticity -- see Theorem \ref{T:AZD6}.


\vskip1mm

We begin by quickly recalling the standard geometric set-up. Let $G$ be a simple algebraic group over $\mathbb{R}$. We view the group of $\R$-points $\mathcal{G} = G(\mathbb{R})$ as a Lie group, pick a maximal compact subgroup $\mathcal{K}$ of $\mathcal{G}$, and consider the associated symmetric space $\mathfrak{X} = \mathcal{G} / \mathcal{K}$. Furthermore, given a discrete torsion-free subgroup $\Gamma \subset \mathcal{G}$, we let $\mathfrak{X}_{\Gamma} = \Gamma \backslash \mathfrak{X}$ denote the corresponding locally symmetric space. We say that $\mathfrak{X}_{\Gamma}$ is {\it arithmetically defined} if the subgroup $\Gamma$ is {\it arithmetic} (see \cite[\S1]{PrRap-WC} for the details). As in the case of Riemann surfaces that we saw in \S\ref{S-Quat}, closed geodesics in $\mathfrak{X}_{\Gamma}$ correspond to nontrivial semi-simple elements of $\Gamma$. However, the formula relating the length of the closed geodesic $c_{\gamma}$ attached to a semi-simple element $\gamma \in \Gamma$ to the
eigenvalues of $\gamma$ is substantially more complicated than equation (\ref{E-GeoLength}) for Riemann surfaces, particularly when the rank of $\mathfrak{X}$, i.e. the real rank of $G$, is $> 1$ (see \cite[Proposition 8.5]{PrRap-WC} for the precise statement). Consequently, the fact that two locally symmetric spaces are length-commensurable does not translate into a simple condition like (\ref{E:A6}) on the eigenvalues of semi-simple elements. Instead, the  characterization of length-commensurable locally symmetric spaces requires the following relationship that was introduced in \cite{PrRap-WC}.

\vskip2mm

\noindent {\bf Definition 9.1.} {\it Let $F$ be an infinite field.}

\vskip2mm

\noindent {\rm (1)} \parbox[t]{16.3cm}{\it Let $\gamma_1 \in \mathrm{GL}_{n_1}(F)$ and $\gamma_2 \in \mathrm{GL}_{n_2}(F)$ be \emph{semi-simple (diagonalizable)} matrices, and let
$$
\lambda_1, \ldots , \lambda_{n_1} \ \ \text{and} \ \ \mu_1, \ldots , \mu_{n_2}
$$
be their eigenvalues (in a fixed algebraic closure $\overline{F}$). We say that $\gamma_1$ and $\gamma_2$ are \emph{weakly commensurable} if there exist integers $a_1, \ldots , a_{n_1}$ , $b_1, \ldots , b_{n_2}$ such that
\begin{equation}\tag{WC}\label{E:WC}
\lambda_1^{a_1} \cdots \lambda_{n_1}^{a_{n_1}} = \mu_1^{b_1} \cdots \mu_{n_2}^{b_{n_2}}.
\end{equation}}

\vskip1mm

\noindent {\rm (2)} \parbox[t]{16.3cm}{\it Let $G_1 \subset \mathrm{GL}_{n_1}$ and $G_2 \subset \mathrm{GL}_{n_2}$ be reductive algebraic $F$-groups. We say that Zariski-dense subgroups $\Gamma_1 \subset G_1(F)$ and $\Gamma_2 \subset G_2(F)$ are \emph{weakly commensurable} if every semi-simple element $\gamma_1 \in \Gamma_1$ of infinite order is weakly commensurable to some semi-simple element $\gamma_2 \in \Gamma_2$ of infinite order, and vice versa.}

\vskip2.5mm

One can view this relationship simply as a higher-dimensional version of (\ref{E:A6}) in the case where $G_1 = G_2 = \mathrm{SL}_2$. But in fact, even irrespective of locally symmetric spaces, it provides a way of matching the eigenvalues of semi-simple elements of $\Gamma_1$ and $\Gamma_2$ that does not depend on the choice of the matrix realizations of the ambient algebraic groups $G_1$ and $G_2$. In the context of locally symmetric spaces though, one proves the following (cf. \cite[Corollary 8.14]{PrRap-WC}):

\vskip2mm

\noindent {\it Let $G_1$ and $G_2$ be simple real algebraic groups, and let $\mathfrak{X}_i$ be the symmetric space of $\mathcal{G}_i = G_i(\mathbb{R})$ for $i = 1, 2$. If $\Gamma_i \subset \mathcal{G}_i$ is a torsion-free lattice and the locally symmetric spaces $\mathfrak{X}_{\Gamma_i} := \Gamma_i \backslash \mathfrak{X}_i$, for $i = 1, 2$ are length-commensurable (in particular, compact isospectral or iso-length-spectral), then their fundamental groups $\Gamma_1$ and $\Gamma_2$ are weakly commensurable.}

\vskip2mm

\noindent On the other hand, the locally symmetric spaces $\mathfrak{X}_{\Gamma_1}$ and $\mathfrak{X}_{\Gamma_2}$ are commensurable if there exists an $\mathbb{R}$-isomorphism $\varphi \colon G_1 \to G_2$ such that $\varphi(\Gamma_1)$ and $\Gamma_2$ are commensurable in
the usual sense (in this case, we say that $\Gamma_1$ and $\Gamma_2$ are commensurable up to an isomorphism between $G_1$ and $G_2$). This suggests that to attack the geometric problem of when length-commensurable locally symmetric spaces are commensurable, one should try to prove that
in the cases of interest, the weak commensurability of Zariski-dense subgroups implies their commensurability (in a suitable sense). At first glance, however, the chances of obtaining a sufficiently general result along these lines appear to be rather slim. Indeed, the matrices
$$
A = \left(\begin{array}{ccc} 12 & 0 & 0 \\ 0 & 2 & 0 \\ 0 & 0 & 1/24
\end{array} \right) \ \  ,  \ \ B = \left( \begin{array}{ccc} 4 & 0 & 0 \\ 0 & 3 & 0 \\ 0 & 0 & 1/12    \end{array} \right) \ \ \in \ \
\mathrm{SL}_3(\mathbb{C})
$$
are weakly commensurable because
$$
\lambda_1  = 12  = 4 \cdot 3 =  \mu_1 \cdot \mu_2 \ \ (\text{or}  \ \ \lambda_1 = \mu^{-1}_3).
$$
On the other hand, no powers $A^m$ and $B^n$ for $m , n \neq 0$ are conjugate, so the subgroup $\langle A \rangle$ is not commensurable to any conjugate of $\langle B \rangle$. It turns out, though,
that the situation changes dramatically if, instead of ``small subgroups" (such as cyclic ones), one considers ``big subgroups" (e.g. Zariski-dense subgroups) of simple algebraic groups. In fact, the case of arithmetic subgroups was worked out almost completely \cite{PrRap-WC}\footnotemark.
Since the results that we will discuss deal with fields of characteristic zero, we will assume for the remainder of this section that \underline{{\it the base field has characteristic zero}.}

\footnotetext{We will not go into the details of this analysis here and would only like to point out that one of the important factors is the existence of so-called {\it generic} elements in every Zariski-dense subgroup --- see \cite{PrRap-Thin}
for a detailed discussion. The reader interested in the technical ingredients can also review the Isogeny Theorem (Theorem 4.2) in \cite{PrRap-WC},
which provides a far-reaching generalization of the following fact used in section \S\ref{S-Quat}: {\it if $\gamma_1 , \gamma_2 \in \mathrm{SL}_2(F)$ are semi-simple elements of infinite order that are weakly commensurable, then for any subfield $K$ that contains the traces of $\gamma_1$ and $\gamma_2$, the subalgebras $K[\gamma_1]$ and $K[\gamma_2]$ are $K$-isomorphic}.}

\addtocounter{thm}{1}

\begin{thm}\label{T:AZD1}{\rm (cf. \cite[Theorems 4 and 5]{PrRap-WC})}
Let $G_1$ and $G_2$ be absolutely almost simple algebraic groups over a field $F$ of characteristic zero.

\vskip2mm

\noindent {\rm (1)} \parbox[t]{16cm}{Assume that $G_1$ and $G_2$ are of the the same type, which is different from $\textsf{A}_n$, $\textsf{D}_{2n+1}$ $(n > 1)$, or $\textsf{E}_6$. If Zariski-dense arithmetic subgroups $\Gamma_1 \subset G_1(F)$ and $\Gamma_2 \subset G_2(F)$ are weakly commensurable, then they are commensurable\footnotemark.}

\vskip1mm

\noindent {\rm (2)} \parbox[t]{16cm}{In all cases, the Zariski-dense arithmetic subgroups $\Gamma_2 \subset G_2(F)$ that are weakly commensurable to a given
Zariski-dense arithmetic subgroup $\Gamma_1 \subset G_1(F)$ form finitely many commensurability classes.}
\end{thm}

\footnotetext{This means that that there exists an $F$-isomorphism $\varphi \colon \overline{G}_1 \to \overline{G}_2$ between the corresponding adjoint groups such that $\varphi(\overline{\Gamma}_1)$ is commensurable with $\overline{\Gamma}_2$, where $\overline{\Gamma}_i$ denotes the image of $\Gamma_i$ in $\overline{G}_i(F)$.}

Let us note that this theorem was established in \cite{PrRap-WC} not only for arithmetic, but also for $S$-arithmetic subgroups. Furthermore, just as in Theorem \ref{T:A1}, the types excluded in part (1) of the theorem are honest exceptions: for each of them, one can construct arbitrarily large, but finite, families of weakly commensurable and pairwise noncommensurable arithmetic subgroups (cf. \cite[\S9]{PrRap-WC}). As we will see in Theorem \ref{T:AZD2} below,
the only situation in which $G_1$ and $G_2$ of {\it different} types can contain finitely generated Zariski-dense weakly commensurable is when one of the groups is of type $\textsf{B}_n$ and the other of type $\textsf{C}_n$ $(n \geq 3)$. Weakly commensurable {\it arithmetic} subgroups in this case were completely classified in \cite{GarRap}.

As we already indicated above, this line of work was initially motivated by questions in differential geometry, and
Theorem \ref{T:AZD1} has served as a basis for various geometric applications to isospectral and length-commensurable locally symmetric spaces. For example, it implies the following statement:

\vskip2mm

\noindent {\it Let $M_1$ and $M_2$ be arithmetically defined hyperbolic manifolds of the same dimension $d \not\equiv 1(\mathrm{mod}\: 4)$. If they are length-commensurable (in particular, if they are compact and isospectral) then they are commensurable.}

\vskip2mm
\noindent (To be more precise, it is actually enough to assume that only one manifold is arithmetically defined and the other is of finite volume.) We refer the reader to \cite{PrRap-WC}, \cite{PrRap-GeomDed}, and \cite{R-ICM} for a number of other geometric applications.


The proof of Theorem \ref{T:AZD1} does not deal with arithmetic subgroups directly, but rather uses the following criterion for their conjugacy. Let $G$ be an absolutely almost simple algebraic group over a field $F$ of characteristic zero. Then the commensurability classes of Zariski-dense arithmetic subgroups $\Gamma$ of $G(F)$ can be parametrized by pairs $(K , \mathscr{G})$, where $K$ is a number field contained in $F$ and $\mathscr{G}$ is an $F/K$-form of the adjoint group $\overline{G}$. More precisely, every arithmetic subgroup is $(K , \mathscr{G})$-arithmetic for some pair $(K , \mathscr{G})$ as above, which means that there is an $F$-isomorphism $\theta \colon \mathscr{G} \times_K F \to \overline{G}$ such that the image $\overline{\Gamma}$ of $\Gamma$ in $\overline{G}(F)$ is commensurable with $\theta(\mathscr{G}(\mathcal{O}_K))$, where $\mathcal{O}_K$ is the ring of algebraic integers in $K$ (this description is very similar to the standard description of arithmetic Fuchsian groups). Then, given absolutely almost simple $F$-groups $G_1$ and $G_2$, {\it Zariski-dense arithmetic subgroups $\Gamma_1 \subset G_1(F)$ and $\Gamma_2 \subset G_2(F)$ corresponding to the pairs $(K_1 , \mathscr{G}_1)$ and $(K_2 , \mathscr{G}_2)$ are commensurable} (up to an $F$-isomorphism between $\overline{G}_1$ and $\overline{G}_2$) {\it if and only if
$$
K_1 = K_2 =: K \ \ \text{and} \ \ \mathscr{G}_1 \simeq \mathscr{G}_2 \ \ \text{over} \ \ K
$$}
(cf. \cite[Proposition 2.5]{PrRap-WC}). So, in the proof of Theorem \ref{T:AZD1}, one first shows that in both parts, there is an equality $K_1 = K_2 =: K$; then, using various local-global considerations, it is shown that $\mathscr{G}_1 \simeq \mathscr{G}_2$ in part (a) and that there are finitely many $K$-isomorphism classes of possible groups $\mathscr{G}_2$ for a given $\mathscr{G}_1$ in part (b).

\vskip1mm

The important point here is that such characteristics as $K$ and $\mathscr{G}$ can be defined and analyzed not only for arithmetic, but in fact for arbitrary Zariski-dense subgroups (although then $K$ does not have to be a number field). Although, in general, they do not necessarily determine the commensurability class of a subgroup, they still carry important information. We begin with two results for arbitrary finitely generated Zariski-dense subgroups that were established in \cite{PrRap-WC}. To fix notations, let $G_1$ and $G_2$ be absolutely almost simple algebraic groups defined over a field $F$ of characteristic zero, and for $i = 1, 2$, let $\Gamma_i \subset
G_i(F)$ be a finitely generated Zariski-dense subgroup.
\begin{thm}\label{T:AZD2}{\rm (\cite[Theorem 1]{PrRap-WC})}
If $\Gamma_1$ and $\Gamma_2$ are weakly commensurable, then either $G_1$ and $G_2$ have the same type or one of them is of type $\textsf{B}_n$ and the other of type $\textsf{C}_n$ for some $n \geq 3$.
\end{thm}
(We note that arithmetic Zariski-dense subgroups in the groups of type $\textsf{B}_n$ and $\textsf{C}_n$ can be weakly commensurable --- cf. \cite{GarRap}, \cite[Example 6.7]{PrRap-WC}.)

Next, given a semi-simple $F$-group $G$ and a Zariski-dense subgroup $\Gamma \subset G(F)$, we let $K_{\Gamma}$ denote the {\it trace field}, i.e. the subfield of $F$ generated by the traces $\mathrm{tr}(\mathrm{Ad}\: \gamma)$ of all elements $\gamma \in \Gamma$ in the adjoint representation on the corresponding Lie algebra $\mathfrak{g} = L(G)$. According to a result of E.B.~Vinberg \cite{Vinberg},
the field $K = K_{\Gamma}$ is the {\it minimal field of definition} of $\Gamma$. This means that $K$ is the minimal subfield of $F$ with the property that all transformations in $\mathrm{Ad}\: \Gamma$ can be simultaneously represented by matrices having all entries in $K$ in a suitable basis of $\mathfrak{g}$. If such a basis is chosen, then the Zariski-closure of $\mathrm{Ad}\: \Gamma$ in $\mathrm{GL}(\mathfrak{g})$ is a semi-simple algebraic $K$-group $\mathscr{G}$. It is an $F/K$-form of the adjoint group $\overline{G}$, and we will call it the {\it algebraic hull} of $\mathrm{Ad}\: \Gamma$. One proves that if $\Gamma$ is a Zariski-dense $(K , \mathscr{G})$-arithmetic subgroup of an absolutely almost simple $F$-group $G$, then the trace field of $\Gamma$ coincides with $K$ and the algebraic hull with $\mathscr{G}$.
\begin{thm}\label{T:AZD3}{\rm (cf. \cite[Theorem 2]{PrRap-WC})}
If $\Gamma_1$ and $\Gamma_2$ are weakly commensurable, then $K_{\Gamma_1} = K_{\Gamma_2}$.
\end{thm}

For the sake of completeness, we mention one further result.
Assume that $\Gamma_1$ and $\Gamma_2$ as above are weakly commensurable, and let $K$ be their common trace field and $\mathscr{G}_i$ be the algebraic hull of $\mathrm{Ad}\: \Gamma_i$ for $i = 1, 2$. We set $L_i$ to be the minimal Galois extension of $K$ over which $\mathscr{G}_i$ becomes an inner form of the split group (see Example 4.5 for the relevant definitions).

\begin{prop}\label{P:AZD1}{\rm (cf. \cite[Lemma 5.2]{PrRap-Kunming})}
In the above notations, we have $L_1 = L_2$.
\end{prop}

What makes these results interesting is that $\Gamma_1$ and/or $\Gamma_2$ may very well be free groups (in fact, by a theorem of Tits \cite{TitsFree}, the group $G_i(F)$ always contains a Zariski-dense subgroup which is a free group of rank two). In this case, structurally these groups carry no imprint of the ambient simple algebraic group, but nevertheless, according to Theorem \ref{T:AZD2}, information about the eigenvalues of their elements is sufficient to recover the type of this group. Furthermore, classical rigidity results due to Mostow and Margulis imply that an isomorphism between higher rank arithmetic group/lattices yields an isomorphism between their fields of definition\footnote{See section \S\ref{S-Rigidity} for some rigidity results over rings more general than rings of algebraic $S$-integers}. However, again, the structural approach does not extend to arbitrary finitely generated Zariski-dense subgroups, while information about eigenvalues
can be used to identify the fields of definition\footnotemark.
\footnotetext{Of course, the traces of elements in the adjoint representation that generate the field of definition can be easily expressed in terms of the eigenvalues, but in our set-up, all we can work with are relations like (\ref{E:WC}) in Definition 9.1 for $\gamma_1 \in \Gamma_1$ and $\gamma_2 \in \Gamma_2$, which do not  immediately yield any relation between $\mathrm{tr}(\mathrm{Ad}\: \gamma_1)$ and $\mathrm{tr}(\mathrm{Ad}\: \gamma_2)$.}
So, Theorems \ref{T:AZD2} and \ref{T:AZD3} point to a new form of rigidity, which is potentially applicable to arbitrary finitely generated Zariski-dense subgroups and is based on information about the eigenvalues of elements in the subgroups --- for these reasons, it was named {\it eigenvalue rigidity} in \cite{R-ICM}. The critical issue in this analysis that is {\it not} addressed by the above results is the precise relationship between the algebraic groups containing $\Gamma_1$ and $\Gamma_2$. Let us fix an absolutely almost simple algebraic $F$-group $G_1$
and a finitely generated Zariski-dense subgroup $\Gamma_1 \subset G_1(F)$, and denote by $K = K_{\Gamma_1}$ and $\mathscr{G}$ its trace field and algebraic hull, respectively.
Next, let $G_2$ be another absolutely almost simple $F$-group with a finitely generated Zariski-dense subgroup $\Gamma_2 \subset G_2(F)$ that is
weakly commensurable to $\Gamma_1$. What can we say about the algebraic hull $\mathscr{G}_2$ of $\Gamma_2$? By Theorem \ref{T:AZD3}, the trace field $K_{\Gamma_2}$ must coincide with $K$, so assuming that $G_2$ is adjoint (which we always may), we conclude that $\mathscr{G}_2$ is an $F/K$-form of $G_2$.
In general, as $\Gamma_2$ varies, while remaining weakly commensurable to $\Gamma_1$, the $K$-isomorphism class of $\mathscr{G}_2$ can also vary. However, one expects that this class always remains within a {\it finite} set of possibilities.


\vskip2mm

\noindent{\bf Conjecture 9.6.} (Finiteness Conjecture for Algebraic Hulls of Weakly Commensurable Subgroups) {\it Let $G_1$ be an absolutely almost simple algebraic $F$-group, and $\Gamma_1 \subset G_1(F)$ be a finitely generated Zariski-dense subgroup with trace field $K = K_{\Gamma_1}$. Given another absolutely simple adjoint $F$-group $G_2$, there exists a \emph{finite} collection $\mathscr{G}_2^{(1)}, \ldots , \mathscr{G}_2^{(r)}$ of $F/K$-forms of $G_2$ such that if $\Gamma_2 \subset G_2(F)$ is a finitely generated Zariski-dense subgroup that is weakly commensurable to $\Gamma_1$, then it is conjugate to a subgroup of one of the $\mathscr{G}_2^{(i)}(K)$ $(\subset G_2(F))$.}

\vskip2mm

Now, instead of fixing an absolutely simple $F$-group $G_2$, one can consider all possible absolutely simple adjoint $F$-groups $G_2$ such that there exists a finitely generated Zariski-dense subgroup $\Gamma_2 \subset G_2(F)$ that is weakly commensurable to $\Gamma_1$. First, according to Theorem \ref{T:AZD2}, apart from the ambiguity between types $\textsf{B}_n$ and $\textsf{C}_n$, $G_2$ must have the same type as $G_1$. We may therefore assume that the type of $G_2$ is fixed. Second, replacing the field $F$ with its algebraic closure $\overline{F}$, we can assume that $G_2$ (or, more precisely, $G_2 \times_F \overline{F}$) itself is fixed as an $\overline{F}$-group. In other words, taking into account the ambiguity between types $\mathsf{B}_n$ and $\mathsf{C}_n$, there are either one or two possibilities for $G_2$ as an $\overline{F}$-group. Applying Conjecture 9.6, we then conclude that, even without initially fixing $G_2$ as an $F$-group, there are only finitely many $K$-isomorphism classes of algebraic hulls $\mathscr{G}_2$.


Thus, if proven, Conjecture 9.6, in conjunction with the preceding results, would tell us that if we fix a finitely generated Zariski-dense subgroup $\Gamma_1 \subset G_1(F)$ with trace field $K = K_{\Gamma_1}$, then the finitely generated Zariski-dense subgroups of absolutely simple algebraic groups $F$-groups that are weakly commensurable to $\Gamma_1$ will all have the same trace field $K$ and there will be only finitely many possibilities for their algebraic hulls. Such a statement would completely resolve one of the keys issues in the study of eigenvalue rigidity (we refer the reader to \cite[\S10]{PrRap-Thin}
for several other open problems in this area).  In more informal terms, Conjecture 9.6 would basically imply that an absolutely simple algebraic group is almost determined (i.e., guaranteed to be in a finite set of possibilities) by the eigenvalues of the elements of a finitely generated Zariski-subgroup, regardless of how
small/thin this subgroup may be (e.g., it can be a free group on two generators). For a slightly different perspective, recall that in Conjecture 8.10, we considered the problem of whether an absolutely almost simple algebraic group $G$ over a finitely generated field $K$ of good characteristic is almost  determined by the set of $K$-isomorphism classes of its maximal $K$-tori. From this point of view, Conjecture 9.6 would imply that $G$ is still almost determined (at least in characteristic zero) by the $K$-isomorphism classes of those maximal $K$-tori that intersect nontrivially a given Zariski-dense subgroup --- cf. \cite[\S5]{PrRap-Kunming}. (Note that this is consistent with our discussion at the end of \S\ref{S-Quat}). Here is what Conjecture 9.6 yields in some concrete situations.

\vskip2mm

\noindent {\bf Example 9.7.} Let $A$ be a central simple algebra of degree $n$ over a finitely generated field $K$ of characteristic zero, set $G = \mathrm{SL}_{1 , A}$ to be the corresponding algebraic $K$-group associated with group of elements of reduced norm 1, and suppose
$\Gamma \subset G(K)$ is a finitely generated Zariski-dense subgroup with trace field $K = K_{\Gamma}$. Now, let $G'$ be another absolutely almost simple simply connected $K$-group such that there exists a finitely generated Zariski-dense subgroup $\Gamma' \subset G'(K)$ that is weakly commensurable to $\Gamma'$. Then we deduce from Theorem \ref{T:AZD2} and Proposition \ref{P:AZD1} that $G'$ is necessarily of the form $G' = \mathrm{SL}_{1 , A'}$, where $A'$ is a central simple $K$-algebra of the same degree $n$ (note that Proposition \ref{P:AZD1} implies that $G'$ is an inner form of $G$). Furthermore, Conjecture 9.6 in this situation would tell us that there are only finitely many possibilities for $A'$ up to $K$-isomorphism.

\vskip2mm

\noindent {\bf Example 9.8.} Let $q$ be a nondegenerate quadratic form in $n \geq 5$ variables over a finitely generated field $K$ of characteristic zero, let $G = \mathrm{SO}_n(q)$ be the corresponding special orthogonal group, and take $\Gamma \subset G(K)$ to be a finitely generated Zariski-dense subgroup with trace field $K = K_{\Gamma}$. Then Conjecture 9.6 would imply that there exist finitely many similarity classes of nondegenerate $n$-dimensional quadratic forms $q'$ over $K$ such that for $G' = \mathrm{SO}_n(q')$, the group $G'(K)$ contains a finitely generated Zariski-dense subgroup $\Gamma'$ that is weakly commensurable to $\Gamma$.

\vskip2mm

The important point in the context of our discussion in this article is that the assertion of Conjecture 9.6 can be derived from Conjecture 5.7. The argument hinges on the following result from \cite{CRR-GR} that links weak commensurability with good reduction.

\addtocounter{thm}{3}

\begin{thm}\label{T:AZD4}
Let $G$ be an absolutely almost simple simply connected algebraic group over a finitely generated field $K$ of characteristic zero, and let $V$ be a divisorial set of places of $K$. Given a finitely generated Zariski-dense subgroup $\Gamma \subset G(K)$ with trace field $K$, there exists a finite subset $S(\Gamma) \subset V$ such that every absolutely almost simple simply connected algebraic $K$-group $G'$ with the property that there exists a finitely generated Zariski-dense subgroup $\Gamma' \subset G'(K)$ that is weakly commensurable to $\Gamma$ has good reduction at all $v \in V \setminus S(\Gamma)$.
\end{thm}

The proof of this theorem is based on a significant elaboration of the ideas involved in the proof of Theorem \ref{T:A4}. The key point is to show that in order to characterize good reduction, one does not need to consider all maximal tori over the base field, but only those that intersect nontrivially a given finitely generated Zariski-dense subgroup and also are {\it generic} (cf. \cite[\S5]{PrRap-Kunming}).

\vskip1mm

In the case where the trace field of $\Gamma$ is a number field (although $\Gamma$ does not need to be arithmetic), Conjecture 9.6 was proved in
\cite[\S5]{PrRap-Kunming}. Furthermore, arguing as in Example 9.7 and combining Theorems \ref{T:ResX1} and \ref{T:AZD4}, one proves Conjecture 9.6 in the case where the
algebraic hull $\mathscr{G}_1$ of $\Gamma_1$ is an inner form of type $\textsf{A}_n$ $(n \geq 1)$. We leave it to the reader to consider various other cases of Conjecture 9.6 that can be obtained by combining Theorem \ref{T:AZD4} with the results from \S\ref{S-ResultsMain} on Conjecture 5.7.
In the case of lattices, the preceding observations lead to the following consequence.
Let $G$ be an absolutely almost simply algebraic $\mathbb{R}$-group, and let $\Gamma$ be a lattice in $\mathcal{G} = G(\mathbb{R})$. It follows from \cite[7.67, 7.68]{Rag} that if $G$ is not isogenous to $\mathrm{SL}_2$, then the trace field $K_{\Gamma}$ is necessarily a number field. Thus, we obtain the following.

\begin{thm}\label{T:AZD5}
Conjecture 9.6 is true when $F = \mathbb{R}$ and $\Gamma_1$ is a lattice (not necessarily arithmetic) in $G_1(\mathbb{R})$.
\end{thm}

We conclude this section with a finiteness statement for Riemann surfaces that does not require any assumptions of arithmeticity or even the finiteness of the volume. Motivated by the well-known result that a family of isospectral compact Riemann surfaces consists of finitely many isometry classes \cite{McKean}, one may wonder if a family of length-commensurable Riemann surfaces necessarily consists of finitely many commensurability classes. While this question remains open, the following result for the associated quaternion algebras is established in \cite{CRR-GR} (here we use the notations introduced in \S\ref{S-Quat}).
\begin{thm}\label{T:AZD6}
Let $M_i = \mathbb{H}/\Gamma_i$ $(i \in I)$ be a family of length-commensurable Riemann surfaces, where the subgroups $\Gamma_i \subset \mathrm{SL}_2(\mathbb{R})$ are finitely generated, discrete, and Zariski-dense, and have torsion-free images in $\mathrm{PSL}_2(\mathbb{R})$. Then all of the $\Gamma_i$ have the same trace field $K$, and the corresponding quaternion algebras $A_{\Gamma_i}$ $(i \in I)$ split into finitely many $K$-isomorphism classes.
\end{thm}

Indeed, it follows from the discussion in \S\ref{S-Quat} that the subgroups $\Gamma_i$ are pairwise weakly commensurable. Hence, by Theorem \ref{T:AZD3}, they all have the same trace field $K$, which is finitely generated. Let $V$ be a divisorial set of places of $K$. Then it follows from Theorem \ref{T:AZD4} that there exists a finite subset $S \subset V$ such that all algebras $A_{\Gamma_i}$ are unramified at all $v \in V \setminus S$. Consequently, the finiteness assertion follows from the finiteness of the unramified Brauer group, as discussed in the proof of Theorem \ref{T:ResX1}. (In the present case, one can also give a rather elementary argument by observing that all of the algebras $A_{\Gamma_i}$ share a quadratic extension of $K$ and arguing as in \cite[\S2]{CRR-Israel}). Curiously, this theorem is probably one of the first applications of reduction techniques and arithmetic geometry to differential geometry.

\section{Afterword}\label{S-Conclusion} Our primary goal in this article was to introduce Conjecture 5.7 (our Main Conjecture) and discuss its links with several other problems in the investigation of algebraic groups  over higher-dimensional fields. For the sake of completeness, we would like to conclude with a brief discussion of two other issues that play a critical role in the classical arithmetic theory of algebraic groups, but whose extension to the more general setting is not directly related to the Main Conjecture: strong approximation and rigidity.

\subsection{Strong approximation.}\label{S-Strong} Let $K$ be a field equipped with a set $V$ of discrete valuations satisfying condition (A) (see \S\ref{S-ConjFGfields} and also \S\ref{S-OtherConj} for notations and terminology pertaining to adelic groups).
Given an algebraic $K$-group $G$, the corresponding adelic group $G(\mathbb{A}(K , V))$ is endowed with a natural topology that has  sets of the form
$$
\prod_{v \in V \setminus T} G(\mathcal{O}_v) \times  \prod_{v \in T} U_v,
$$
where $T \subset V$ is a finite subset and $U_v \subset G(K_v)$ is an arbitrary open subset for $v \in T$, as a basis of open neighborhoods. We say that $G$ has {\it strong approximation} with respect to $V$ if the diagonal embedding $G(K) \hookrightarrow G(\mathbb{A}(K , V))$ has dense image.\footnote{Let us point out that here we deviate from the standard terminology. Namely, recall that in the classical setting where $K$ is a number field, $V^K$ is the set of all places of $K$, and $S \subset V^K$ is a finite subset with complement $V = V^K \setminus S$, we say that $G$ has strong approximation {\it with respect to $S$} if the diagonal embedding  $G(K) \hookrightarrow G(\mathbb{A}(K , V))$ has dense image (cf. \cite[Ch. 7]{Pl-R}).}
This property plays a fundamental role in the study of various questions about algebraic groups over global fields. We refer the reader to \cite[Ch. 7]{Pl-R} for a detailed account (based on Platonov's proof \cite{PlatonovSA}) of strong approximation for algebraic groups over number fields, and to the original papers of Margulis \cite{MargulisSA} and Prasad \cite{PrasadSA} for groups over arbitrary global fields.
The survey \cite{Rap-BreuOh} also contains a discussion of some variations, generalizations, and applications of strong approximation. In this section, we will briefly consider strong approximation over fields other than global.


\vskip2mm

\noindent {\bf Example 10.1.} Let $R$ be a Noetherian integral domain that is integrally closed in its field of fractions $K$, and let $V$ be the set of discrete valuations of $K$ associated with height one prime ideals of $R$. Take $G$ to be the additive group $\mathbb{G}_a$, so that the corresponding adelic group is simply the additive group of the adele ring $\mathbb{A}(K , V)$. If $G$ has strong approximation with respect to $V$, then it easily follows from the definition of the adelic topology that for any height one prime ideals $\mathfrak{p}_1 , \mathfrak{p}_2$ of $R$, $\mathfrak{p}_1 \neq \mathfrak{p}_2$, the diagonal map
$$
R \longrightarrow R/\mathfrak{p}_1 \times R/\mathfrak{p}_2
$$
must be surjective. However, if the Krull dimension of $R$ is $> 1$, then typically $R$ contains height one prime ideals $\mathfrak{p}_1 \neq \mathfrak{p}_2$ such that $\mathfrak{p}_1 + \mathfrak{p}_2 \neq R$, in which case the above map is not surjective, and hence $G$ fails to have strong approximation with respect to $V$. On the other hand, if the Krull dimension of $R$ is 1, i.e. $R$ is a Dedekind domain, then any two distinct prime ideals $\mathfrak{p}_1 , \mathfrak{p}_2$ of $R$ satisfy $\mathfrak{p}_1 + \mathfrak{p}_2 = R$. It follows from the Chinese Remainder Theorem that $R$ is dense in
$$
\mathbb{A}^{\infty}(K , V) = \prod_{v \in V} \mathcal{O}_v,
$$
easily implying that $G$ does have strong approximation with respect to $V$.

\vskip2mm

There are two takeaways from this example. First, we should confine the consideration of strong approximation to the situation where $K$ is the fraction field of a Dedekind domain $R$ and $V$ is the set of discrete valuations of $K$ associated with the nonzero prime ideals of $R$. The most interesting case is where $R = k[C]$ is the coordinate ring of a smooth geometrically integral affine curve $C$ over a field $k$ and $V$ is the corresponding set of geometric places of the function field $K = k(C)$. (Of course, if $k$ is a finite field, then the function field $K$ is global, in which case strong approximation is completely understood. However, the problem for other fields remains wide open --- see below.)

Second, the $K$-subgroups of $G$ isomorphic to $\mathbb{G}_a$ may be helpful for proving strong approximation in $G$. More precisely, let us assume that $G$ is an absolutely almost simple simply connected $K$-group. We recall that $G$ is called $K$-{\it isotropic} if it contains a
nontrivial $K$-split torus (in which case it also contains an
abundance of 1-parameter unipotent subgroups), and $K$-{\it
anisotropic} otherwise.
Assuming that $G$ is $K$-isotropic, one can consider the (normal) subgroup $G(K)^+ \subset G(K)$ generated by the $K$-rational elements of the unipotent radicals of $K$-defined parabolic subgroups. In \cite{PlatonovTA},
V.P.~Platonov constructed the first examples where $G(K)^+ \neq G(K)$, thereby disproving the Kneser-Tits conjecture (see \cite{GilleBourbaki} for a relatively recent survey of the Kneser-Tits problem). At the same time, the quotient $$W(G , K) = G(K)/G(K)^+,$$ known as the {\it Whitehead group} of $G$ over $K$, is conjectured to be abelian (in which case it will automatically be of finite exponent depending only on the type of $G$), which has been established in many cases. (In fact, $W(G , K)$ is expected to be {\it finite} abelian group whenever $K$ is finitely generated, but only partial results are available so far --- cf. \cite{CTFinitude}).
Motivated by these properties, we will say that  a normal subgroup $N$ of an abstract group $H$ is ``big'' if the quotient $H/N$ is an abelian group of finite exponent.
Now, if $K$ is the fraction field of a Dedekind domain $R$ and $V$ is the set of discrete valuations of $K$ associated with the nonzero prime ideals of $R$, then using strong approximation in $\mathbb{G}_a$, one easily proves that for any absolutely almost simple simply connected $K$-isotropic group $G$, the closure $\overline{G(K)}$ of $G(K)$ in $G(\mathbb{A}(K , V))$ always contains
$$
G(\mathbb{A}(K , V)) \, \bigcap \, \prod_{v \in V} G(K_v)^+.
$$
Thus, for those types where the Whitehead group $W(G , F)$ is known to be abelian for any field extension $F/K$, the closure $\overline{G(K)^+}$ is a ``big'' subgroup of $G(\mathbb{A}(K , V))$. We observe that Platonov \cite{PlatonovSteklov} has found examples of isotropic groups where $G(K)$ is not dense in $G(K_v)$, so $G$ may not have strong approximation in the general case. It would be interesting to determine if $\overline{G(K)}$ may be of infinite index in $G(\mathbb{A}(K , V))$ when $K = k(C)$ is the function field of a smooth geometrically integral affine curve $C$ over a finitely generated field $k$ and $V$ is the set of geometric places of $K$.

\vskip2mm

\noindent {\bf Example 10.2.} Let $K = \mathbb{C}(t)$, take $V$ to be the set of discrete valuations of $K$ corresponding to all linear polynomials $t - a$, $a \in \mathbb{C}$, and let $G$ be an absolutely almost simple simply connected $K$-group. It follows from Tsen's theorem that $K$ has cohomological dimension $\leq 1$, which implies that $G$ is quasi-split (cf. \cite[Ch. III, \S2]{Serre-GC}). Then the Whitehead group $W(G, F)$ is trivial for any field extension $F/K$ (cf. \cite{TitsWhitehead}). So, the above discussion shows that $G$ has strong approximation with respect to $V$.

\vskip2mm

Of course, techniques using 1-parameter unipotent subgroups are inapplicable if $G$ is $K$-anisotropic. So, we would like to propose the following.

\vskip2mm

\noindent {\bf Problem 10.3.} {\it Let $K$ be the field of fractions of a Dedekind domain $R$, and let $V$ be the set of discrete valuations of $K$ corresponding to the nonzero prime ideals of $R$. Investigate the problem of strong approximation for an absolutely almost simple simply connected $K$-anisotropic group $G$. When can one guarantee that the closure $\overline{G(K)}$ of $G(K)$ in $G(\mathbb{A}(K , V))$ is a ``big'' subgroup of the latter? More specifically, for the coordinate ring $R = k[C]$ of a smooth geometrically integral affine curve $C$ over a finitely generated field $k$ and $V$ the corresponding set of geometric places of $K = k(C)$, in what cases is $\overline{G(K)}$ of finite index in $G(\mathbb{A}(K , V))$? Can one always ensure that it is of finite index by deleting from $V$ a finite set of places?}

\vskip2mm

To the best of our knowledge, in the anisotropic case over fields other than global, strong approximation has been established only for the groups $\mathrm{SL}_{1 , D}$, where $D$ is a quaternion division algebra over $\mathbb{R}(t)$ (see \cite{Yamasaki}). It is likely that this result can be extended to function fields $\mathbb{R}(C)$ of arbitrary smooth geometrically integral affine real curves $C$. Furthermore, since every semi-simple algebraic group over $\mathbb{R}(C)$ becomes quasi-split over its quadratic extension $\mathbb{C}(C)$ (cf. Example 10.2), using strong approximation for the quaternionic group, in conjunction with the techniques of dealing with semi-simple groups that split over a quadratic extension of the base field (cf. \cite{Weisfeiler}), one may be able to establish strong approximation for many (and maybe even all) absolutely almost simple simply connected groups over $\mathbb{R}(C)$. Because of the paucity of research done in this area so far, it is difficult to predict what methods may be useful, but one should probably re-examine Kneser's approach to strong approximation (\cite{Kneser1}, \cite{Kneser2}),
which relies primarily on considerations from Galois cohomology.



\subsection{Rigidity}\label{S-Rigidity}
 The analysis of representations, and more generally actions, of arithmetic groups and lattices has been one of the central subjects in the theory of arithmetic groups and discrete subgroups of Lie groups in the past 60 years. So, to complete our account of the trends in the arithmetic theory of algebraic groups over higher-dimensional fields, we would like to discuss briefly one result on representations of higher-dimensional analogues of arithmetic groups. The reader may want to consult \cite{IR3} for more information. This subject goes back to the classical paper of Bass, Milnor, and Serre \cite{BMS}, where, as a consequence of the solution of the Congruence Subgroup Problem, it was shown that for $n \geq 3$, every finite-dimensional complex representation $$\rho \colon \mathrm{SL}_n(\mathbb{Z}) \to \mathrm{GL}_k(\mathbb{C})$$ is {\it almost algebraic}, i.e. there exists a morphism of algebraic groups $\sigma \colon \mathrm{SL}_n(\mathbb{C}) \to \mathrm{GL}_k(\C)$ such that for a suitable finite-index subgroup $\Delta \subset \mathrm{SL}_n(\mathbb{Z})$, the restrictions $\rho \vert \Delta$ and $\sigma \vert \Delta$ coincide (cf. \cite[Theorem 16.2]{BMS}). Serre \cite{SerreCongruence} proved a similar result for the group ${\rm SL}_2(\Z[1/p])$. Subsequently, very general results about
representations of higher-rank arithmetic groups were obtained by Margulis in his Superrigidity Theorem (cf. \cite[Ch. VII]{Mar}).
At the same time, Steinberg \cite{St2} showed that the above results for representations of ${\rm SL}_n (\Z)$ ($n \geq 3$) can be derived directly from the commutator relations for elementary matrices.

Taking inspiration from Steinberg's generators-relations approach, the second author introduced in \cite{IR} a novel method for analyzing abstract representations of elementary subgroups of higher-rank Chevalley groups over arbitrary commutative rings. To fix notations, let $\Phi$ be a reduced irreducible root system of rank $\geq 2$, and let $G$ be the simply-connected Chevalley group scheme over $\Z$ of type $\Phi$. Given a commutative ring $R$, we denote by $E(R)$ the subgroup of $G(R)$ generated by the $R$-points of the canonical 1-parameter root subgroups $e_{\alpha} \colon \mathbb{G}_a \to G$ for all $\alpha \in \Phi$ (this group is usually called the {\it elementary subgroup} of $G(R)$). In \cite{IR}, the second author studied in detail the finite-dimensional representations
$$
\rho \colon E(R) \to {\rm GL}_n(K),
$$
where $K$ is an algebraically closed field of characteristic 0. He showed that if $(\Phi, R)$ is a {\it nice pair}\footnotemark, then \footnotetext{We will say that $(\Phi, R)$ is a {\it nice pair} if $2$ is a unit in $R$ whenever $\Phi$ contains a subsystem of type $\mathsf{B}_2$, and 2 and $3$ are units in $R$ whenever $\Phi$ is of type $\mathsf{G}_2$.} such a representation often has a {\it standard description}, i.e. there exists a commutative finite-dimensional $K$-algebra $B$, a ring homomorphism $f \colon R \to B$ with Zariski-dense image, and a morphism $\sigma \colon G(B) \to {\rm GL}_n(K)$ of algebraic $K$-groups such that for a suitable finite-index subgroup $\Gamma \subset E(R)$ of, we have
$$
\rho \vert_{\Gamma} = (\sigma \circ F) \vert_{\Gamma},
$$
where $F \colon E(R) \to E(B)$ is the group homomorphism induced by $f$ (see \cite[Main Theorem]{IR} for the precise statement). A key step in the proof of this theorem is the construction of an algebraic ring that is naturally associated to the representation $\rho$ and which captures information about the images of all root subgroups of $E(R)$ (see \cite[Theorem 3.1]{IR}). This result, on the one hand, confirmed in the case of split groups a long-standing conjecture of Borel and Tits \cite{BT} on the structure of abstract homomorphisms of algebraic groups, and, on the other hand, subsumed most previous rigidity statements for Chevalley groups over commutative rings. Subsequently, an analogous statement was also obtained for representations of the groups $G(k)$, where $G = {\rm SL}_{n, D}$ with $n \geq 3$ and $D$ is a finite-dimensional central division algebra over a field $k$ of characteristic 0 (cf. \cite{IR1}). Moreover, in \cite{IR1} and \cite{IR-Char}, these results were applied to the analysis of character varieties of certain finitely generated groups. Additionally, by studying the structure of algebraic rings in positive characteristic, D.~Boyarchenko and the second author \cite{BoRa} established, in many situations, the existence of standard descriptions for representations of elementary subgroups of Chevalley groups over fields of characteristic $p$.

Now, while the rigidity properties of arithmetic groups and lattices are well understood, over the
last 20 years, there has been a great deal of interest in the representations and related properties (such as Kazhdan's property (T)) of groups over higher-dimensional rings, particularly the groups ${\rm SL}_n(\Z[x_1, \dots, x_k])$ with $n \geq 3$ (these are sometimes called {\it universal lattices}). In the context of the analysis abstract homomorphisms, we should mention that, using a variation of the method of Bass, Milnor, and Serre, Shenfeld \cite{Sh} showed that any {\it completely reducible} finite-dimensional complex representation of ${\rm SL}_n(\Z[x_1, \dots, x_k])$ ($n \geq 3)$ has a standard description, thereby answering a question of Kazhdan. However, until recently, there were no rigidity statements available for {\it arbitrary} finite-dimensional representations of universal lattices. In \cite{IR-Curves}, using the general framework developed in \cite{IR}, in conjunction
with a careful analysis of certain central extensions, the second author obtained the following rigidity result.


\addtocounter{thm}{3}

\begin{thm}\label{T-RigidityCurves}{\rm (cf. \cite[Corollary 1.3]{IR-Curves})}
Let $\Phi$ be a reduced irreducible root system of rank $\geq 2$, $G$ be the corresponding simply-connected Chevalley group scheme over $\Z$, and be $K$ an algebraically closed field of characteristic 0. If $\mathcal{O}$ is the ring of $S$-integers in a number field such that $(\Phi, \mathcal{O})$ is a nice pair, then any representation
$$
\rho \colon E(\mathcal{O}[x]) \to {\rm GL}_m(K)
$$
has a standard description.
\end{thm}

In fact, the results of \cite{IR-Curves} deal, more generally, with representations of the groups $E(R)$, where $R$ is a ring with ``few" derivations, which, in particular, explains the classical rigidity results for
Chevalley groups of rank $\geq 2$ over number rings. Now, by a well-known result of Suslin \cite{Suslin}, if $\Phi$ is of type $\mathsf{A}$, then $E(\mathcal{O}[x]) = {\rm SL}_n(\mathcal{O}[x])$, so Theorem \ref{T-RigidityCurves} shows, in particular, that {\it any} finite-dimensional representation of the universal lattice ${\rm SL}_n(\Z[x])$ has a standard description. We should note that Suslin's result has recently been extended by Stavrova \cite{Stavrova} to {\it all} simply-connected Chevalley groups of rank $\geq 2$, thereby yielding the existence of standard descriptions for arbitrary representations of $G(\mathcal{O}[x]).$

\vskip3mm

\noindent {\small {\bf Acknowledgements.}  Special thanks are due to Brian Conrad, who carefully read the article and made a number of
suggestions that helped to improve the exposition. We are also grateful to Uriya First, Ariyan Javanpeykar, Boris Kunyavski\={\i}, Daniel Loughran, Alexander Merkurjev, Dipendra Prasad, C.~Rajan, Zinovy Reichstein, Charlotte Ure, and Uzi Vishne for helpful comments.}

\bibliographystyle{amsplain}

\begin{thebibliography}{100}

\bibitem{Abr} V.A.~Abrashkin, {\it $p$-divisible groups over $\Z$}, Izv. Akad. Nauk SSSR Ser. Mat. {\bf 41} (1977), no. 5, 987-1007, 1199.

\bibitem{ANT} {\it Algebraic Number Theory}, ed. by J.W.S.~Cassels and A.~Fr\"ohlich, 2nd edition, London Math. Soc., 2010.

\bibitem{Amitsur} S.A.~Amitsur, {\it Generic splitting fields of
central simple algebras}, Ann. of math. (2) {\bf 62}(1955), 8-43.


\bibitem{At} M.F.~Atiyah, I.G.~MacDonald, {\it Introduction to Commutative Algebra}, Westview Press, 1969

\bibitem{Bass} H.~Bass, {\it K-theory and stable algebra,} Publ. math. IHES {\bf 22}(1964), 5-60.


\bibitem{Bass-Conj} H.~Bass, {\it Some problems in ``classical" algebraic $K$-theory}, Algebraic $K$-theory, II: ``Classical'' algebraic $K$-theory and connections with arithmetic (Proc. Conf., Battelle Memorial Inst., Seattle, Wash., 1972), 3-73, Lecture Notes in Math. {\bf 342}, Springer, Berlin, 1973.

\bibitem{BMS} H.~Bass, J.~Milnor, J.P.~Serre, {\it Solution of the congruence subgroup problem for $SL_ n(n \geq 3)$ and $Sp_{2n} (n \geq 2)$}, Publ. Math., Inst. Hautes Etud. Sci. {\bf 33} (1967), 59-137.

\bibitem{GilleG2} C.~Beli, P.~Gille, T.-Y.~Lee, {\it Examples of algebraic groups of type $\mathsf{G}_2$ having the same maximal tori},
Proc. Steklov Inst. Math. {\bf 292} (2016), no. 1, 10-19.

\bibitem{BorelChevalley} A.~Borel, {\it Properties and linear representations of Chevalley groups}, in  {\it Seminar on Algebraic Groups and Related Finite Groups}, pp. 1-55, Lecture Notes in Mathematics {\bf 131}, Springer, 1970.

\bibitem{BorelAG} A.~Borel, {\it Linear Algebraic Groups. Second Enlarged Edition}, GTM 126, Springer, 1997.

\bibitem{Borel} A.~Borel, {\it Some finiteness properties of adele groups over number fields}, Inst. Hautes \'Etudes Sci. Publ. Math. No. {\bf 16} (1963), 5-30.

\bibitem{BorelSerre} A.~Borel, J.-P.~Serre, {\it Th\'eor\`emes de finitude en cohomologie galoisienne}
Comment. Math. Helv. {\bf 39} (1964), 111-164.

\bibitem{BT} A.~Borel, J.~Tits, {\it Homomorphismes ``abstraits" de groupes alg\'ebriques simples}, Ann. of Math. {\bf 97} (1973), no. 3, 499-571

\bibitem{BoRa} M.~Boyarchenko, I.A.~Rapinchuk, {\it On abstract representations of the groups of rational points of algebraic groups in positive characteristic}, Arch. Math. (Basel) {\bf 107} (2016), no. 6, 569–580.

\bibitem{Bour-CA} N.~Bourbaki, {\it Commutative Algebra. Chapters 1-7}, Springer, 1989.

\bibitem{ChGP} V.I.~Chernousov, P.~Gille, A.~Pianzola, {\it Torsors over the punctured affine line}, Amer. J. Math. {\bf 134} (2012), no. 6, 1541-1583.

\bibitem{CNPY16} V.I.~Chernousov, E.~Neher, A.~Pianzola, U.~Yahorau, {\it On conjugacy of Cartan subalgebras in extended affine Lie algebras},
Adv. Math. {\bf 290} (2016), 260-292.

\bibitem{CRR01} V.I.~Chernousov, A.S.~Rapinchuk, and I.A.~Rapinchuk, {\it The genus of a division algebra
and the unramified Brauer group},  Bull. Math. Sci. {\bf 3} (2013), 211-240.


\bibitem{CRR2}  V.I.~Chernousov, A.S.~Rapinchuk, I.A.~Rapinchuk, {\it Division algebras with the same maximal subfields}, Russian Math. Surveys {\bf 70}(2015), no. 1, 91-122.

\bibitem{CRR3}  V.I.~Chernousov, A.S.~Rapinchuk, I.A.~Rapinchuk, {\it On the size of the genus of a division algebra,} Proc. Steklov Inst. of Math. {\bf 292}(2016), no. 1, 63-93.

\bibitem{CRR4} V.I.~Chernousov, A.S.~Rapinchuk, I.A.~Rapinchuk, {\it On some finiteness properties of algebraic groups over finitely generated fields}, C.R. Acad. Sci. Paris, Ser. I, {\bf 354}(2016), 869-873.

\bibitem{CRR-Spinor} V.I.~Chernousov, A.S.~Rapinchuk, I.A.~Rapinchuk, {\it Spinor groups with good reduction}, Compositio Math. {\bf 155} (2019), 484-527.

\bibitem{CRR-Israel} V.I.~Chernousov, A.S.~Rapinchuk, I.A.~Rapinchuk, {\it The finiteness of the genus of a finite-dimensional division algebra, and generalizations}, to appear in Israel J. Math.

\bibitem{CRR-GR} V.I.~Chernousov, A.S.~Rapinchuk, I.A.~Rapinchuk, {\it Simple algebraic groups with the same maximal tori, weakly commensurable Zariski-dense subgroups, and good reduction}, in preparation

\bibitem{CT-SB} J.-L.~Colliot-Th\'el\`ene, {\it Birational invariants, purity and the Gersten conjecture}, in K-theory and algebraic geometry: connections with quadratic forms and division algebras (Santa Barbara, CA, 1992), 1-64, Proc. Sympos. Pure Math. {\bf 58}, Part 1, Amer. Math. Soc., Providence, RI, 1995.

\bibitem{CTFinitude} J.-L.~Colliot-Th\'el\`ene, {\it Quelques r\'esultats de finitude pour le groupe $SK_1$ d'une alg\`ebre de biquaternions}, $K$-Theory {\bf 10} (1996), no. 1, 31-48.

\bibitem{CTPS} J.-L.~Colliot-Th\'el\`ene, R.~Parimala, V.~Suresh, {\it Patching and local-global principles for homogeneous spaces over function fields of p-adic curves}, Comment. Math. Helv. 87 (2012), no. 4, 1011-1033.


\bibitem{CTS} J.-L.~Colliot-Th\'el\`ene, J.-J.~Sansuc, {\it Fibr\'es quadratiques et composantes connexes r\'eelles}, Math. Ann. {\bf 244} (1979), no. 2, 105-134.



\bibitem{ConradFiniteness} B.~Conrad, {\it Finiteness theorems for algebraic groups over function fields}, Compos. Math. {\bf 148} (2012), no. 2, 555-639.

\bibitem{Conrad} B.~Conrad, {\it Non-split reductive groups over $\Z$}, Autour des sch\'emas en groupes. Vol. II, 193-253, Panor. Synth\`eses, {\bf 46}, Soc. Math. France, Paris, 2015.

\bibitem{Con} B.~Conrad, {\it Reductive group schemes}, Autour des sch\'emas en groupes, \'Ecole d'\'et\'e ``Sch\'emas en groupes," vol. I (Luminy 2011),
Soc. Math. France, Paris 2014.

\bibitem{CGP} B.~Conrad, O.~Gabber, G.~Prasad, {\it Pseudo-reductive groups. Second edition}, New Mathematical Monographs {\bf 26} Cambridge University Press, Cambridge, 2015.

\bibitem{Cut} S.D.~Cutkosky, {\it Introduction to Algebraic Geometry}, GSM {\bf 188}, AMS, 2018.


\bibitem{Darmon} H.~Darmon, {\it Rational points on curves}, in {\it Arithmetic geometry}, Clay Math. Proc. {\bf 8} (2009), 7-53.

\bibitem{DemGab} M.~Demazure, Michel, P.~Gabriel, {\it Groupes alg\'ebriques. Tome I: G\'eom\'etrie alg\'ebrique, g\'en\'eralit\'es, groupes commutatifs}, North-Holland Publishing Co., Amsterdam, 1970.

\bibitem{DemGr} M.~Demazure, A.~Grothendieck, {\it Sch\'emas en groupes},
S\'eminaire
de g\'eom\'etrie alg\'ebrique du Bois Marie 1962/64 (SGA 3). Lect. Notes Math. {\bf 151}, {\bf 152}, {\bf 153}, Springer, 1970.

\bibitem{EKM} R.~Elman, N.~Karpenko, A.~Merkurjev, {\it The algebraic and geometric theory of quadratic forms},
Amer. Math. Soc. Colloq. Publ. {\bf 56}, Amer. Math. Soc., Providence, RI, 2008.

\bibitem{Faltings} G.~Faltings, {\it Endlichkeitss\"{a}tze f\"{u}r abelsche Variet\"{a}ten \"{u}ber Zahlk\"{o}rpern}, Invent. Math. {\bf 73} (1983), no. 3, 349-366.

\bibitem{FarbDennis} B.~Farb, R.K.~Dennis, {\it Noncommutative algebra}, GTM {\bf 144}, Springer, 1993.


\bibitem{Fon} J.-M.~Fontaine, {\it Il n'y a pas de vari\'et\'e ab\'elienne sur $\Z$}, Invent. Math. {\bf 81} (1985), no. 3, 515-538.

\bibitem{Forster} O.~Forster, {\it Lectures on Riemann surfaces}, GTM {\bf 81}, Springer, 1981.

\bibitem{Gar} S.~Garibaldi, {\it Outer automorphisms of algebraic groups and determining groups by their maximal tori},
Michigan Math. J. {\bf 61} (2012), no. 2, 227-237.

\bibitem{GS} S.~Garibaldi, D.~Saltman, {\it Quaternion Algebras with the Same Subfields}, Quadratic forms, linear algebraic groups, and cohomology,  225-238, Dev. math. 18, Springer, New York, 2010.

\bibitem{GarRap} S.~Garibaldi, A.S.~Rapinchuk, {\it Weakly commensurable $S$-arithmetic subgroups in almost simple algebraic groups of types $\mathsf{B}$ and $\mathsf{C}$}, Algebra Number Theory {\bf 7} (2013), no. 5, 1147-1178.


\bibitem{GilleAffine} P.~Gille, {\it Torseurs sur la droite affine}, Transform. Groups {\bf 7} (2002), no. 3, 231-245.

\bibitem{GillePian} P.~Gille, A.~Pianzola, {\it Isotriviality and \'etale cohomology of Laurent polynomial rings}, J. Pure Appl. Algebra {\bf 212} (2008), no. 4, 780-800.

\bibitem{GilleBourbaki} P.~Gille, {\it Le probl\`eme de Kneser-Tits}, S\'eminaire Bourbaki. Vol. 2007/2008,
Ast\'erisque No. 326 (2009), Exp. No. 983, vii, 39-81 (2010).

\bibitem{Gille} P.~Gille, T.~Szamuely, {\it Central Simple Algebras and Galois Cohomology}, Cambridge Univ. Press, Cambridge, 2006.


\bibitem{Gross} B.H.~Gross, {\it Groups over $\Z$}, Invent. math. {\bf 124}(1996), no. 1-3, 263-279.

\bibitem{Guo} N.~Guo, {\it The Grothendieck-Serre Conjecture over Semilocal Dedekind Rings}, arXiv:1902.02315.

\bibitem{HH} S.~Harada, T.~Hiranouchi, {\it Smallness of fundamental groups for arithmetic schemes}, J. Number Theory {\bf 129} (2009), no. 11, 2702-2712.

\bibitem{HHK} D.~Harbater, J.~Hartmann, D.~Krashen, {\it Applications of patching to quadratic forms and central simple algebras}, Invent. Math. {\bf 178} (2009), no. 2, 231-263.

\bibitem{HHK1} D.~Harbater, J.~Hartmann, D.~Krashen, {\it Local-global principles for torsors over arithmetic curves},
Amer. J. Math. {\bf 137} (2015), no. 6, 1559-1612.

\bibitem{Harder} G.~Harder, {\it Halbeinfache Gruppenschemata \"uber Dedekindringen}, Invent. math. {\bf 4}(1967), 165-191.

\bibitem{Harder1} G.~Harder, {\it \"{U}ber die Galoiskohomologie halbeinfacher algebraischer Gruppen. III}, J. Reine Angew. Math. {\bf 274/5} (1975), 125-138.

\bibitem{Hart} R.~Hartshorne, {\it Algebraic Geometry}, GTM {\bf 52}, Springer, 1977.

\bibitem{Izhb} O.T.~Izhboldin, {\it Motivic equivalence of quadratic forms}, Doc. math. {\bf 3}(1998), 341-351.


\bibitem{Jann} U.~Jannsen, {\it Principe de Hasse cohomologique}, S\'eminaire de Th\'eorie des Nombres, Paris 1989-90, 121-140.



\bibitem{Jann2} U.~Jannsen, {\it Hasse principles for higher-dimensional fields},  Ann. of Math. (2) {\bf 183} (2016), no. 1, 1-71.

\bibitem{JL} A.~Javanpeykar, D.~Loughran, {\it Good reduction of algebraic groups and flag varieties}, Arch. Math. {\bf 104} (2015), no. 2, 133-143.

\bibitem{Kac} M.~Kac, {\it Can one hear the shape of a drum?}, Amer. Math. Monthly {\bf 73} (1966), no. 4, part II, 1-23.

\bibitem{Kahn} B.~Kahn, {\it Sur le groupe des clsses d'un sch\'ema arithm\'etiques}, Bull. Soc. Math. France {\bf 134}(2006), no. 3, 395-415.

\bibitem{Karp} N.A.~Karpenko, {\it Criteria of motivic equivalence for quadratic forms and central simple algebras}, Math. Ann.
{\bf 317}(2000), no. 3, 585-611.

\bibitem{Kato} K.~Kato, {\it A Hasse principle for two dimensional global fields, with an appendix by J.-L.~Colliot-Th\'el\`ene}, J. reine angew. Math.
{\bf 366}(1986), 142-183.

\bibitem{Kneser1} M.~Kneser, {\it Starke Approximation in algebraischen Gruppen. I}, J. Reine Angew. Math. {\bf 218} (1965), 190-203.

\bibitem{Kneser2} M.~Kneser, {\it Strong approximation}, in {\it Algebraic Groups and Discontinuous Subgroups} (Proc. Sympos. Pure Math., Boulder, Colo., 1965), pp. 187-196, Amer. Math. Soc., Providence, R.I.


\bibitem{KMRT} M.-A.~Knus, A.~Merkurjev, M.~Rost, J.-P.~Tignol, {\it The book of involutions}, American Mathematical Society Colloquium Publications {\bf 44}, 1998.

\bibitem{KO} M.-A.~Knus, M.~Ojanguren, {\it Th\'eorie de la descente et alg\`ebres d'Azumaya}, Lecture Notes in Mathematics {\bf 389}, Springer, 1974.

\bibitem{KrashMc} D.~Krashen, K.~McKinnie, {\it Distinguishing division algebras by finite splitting fields}, Manuscripta Math. {\bf 134} (2011), no. 1-2, 171–182.

\bibitem{Lam} T.Y.~Lam, {\it Introduction to Quadratic Forms over Fields}, GSM 67, AMS 2005.


\bibitem{LaDG} S.~Lang, {\it Fundamentals of Diophantine Geometry}, Springer, 1983.

\bibitem{McLReid} C.~Maclachlan, A.~Reid, {\it The arithmetic of hyperbolic 3-manifolds}, GTM {\bf 219}, Springer 2003.

\bibitem{MargulisSA} G.A.~Margulis, {\it Cobounded subgroups in algebraic groups over local fields},
Funkcional. Anal. i Prilozen. {\bf 11} (1977), no. 2, 45-57, 95.

\bibitem{Mar} G.A.~Margulis, {\it Discrete Subgroups of Semisimple Lie Groups}, Springer, New York/Berlin 1991.

\bibitem{McKean} H.P.~McKean, {\it Selberg's trace formula as applied to a compact Riemann surface}, Comm. Pure Appl. Math. {\bf 25} (1972), 225-246.

\bibitem{Meyer} J.S.~Meyer, {\it A division algebra with infinite genus}, Bull. London Math. Soc. {\bf 46}(2014), 463-468.

\bibitem{MilneCFT} J.S.~Milne, {Class Field Theory}, available on the author's website at https://www.jmilne.org/math/CourseNotes/cft.html

\bibitem{Milne} J.S.~Milne, {\it Etale cohomology}, Princeton Univ. Press, 1980.

\bibitem{MilnorKTheory} J.~Milnor, {\it Introduction to algebraic $K$-theory}, Annals of Mathematics Studies, No. 72. Princeton University Press, 1971.

\bibitem{Mordell} L.J.~Mordell, {\it On the rational solutions of the indeterminate equations of the third and fourth degrees}, Proc. Camb. Philos. Soc. {\bf 21} (1922), 179-192.

\bibitem{Moret-Bailly} L.~Moret-Bailly, {\it Pinceaux de vari\'et\'es ab\'eliennes}, Ast\'erisque {\bf 129}, 1985.

\bibitem{MumfordAV} D.~Mumford, {\it Abelian varieties}, Tata Institute of Fundamental Research Studies in Mathematics {\bf 5}, 1970.

\bibitem{Neukirch} J.~Neukirch, {\it Algebraic number theory}, Grundlehren der Mathematischen Wissenschaften {\bf 322}, Springer, 1999.

\bibitem{Nisn1} E.A.~Nisnevich, {\it Espaces homog\`enes principaux rationnellement triviaux et arithm\'etique des sch\'emas en groupes r\'eductifs sur les anneaux de Dedekind}, C. R. Acad. Sci. Paris Sér. I Math. {\bf 299} (1984), no. 1, 5-8.

\bibitem{OMeara} O.T.~O'Meara, {\it Introduction to quadratic forms}, Die Grundlehren der mathematischen Wissenschaften {\bf 117} Springer, 1963. 1963
\bibitem{OVV} D.~Orlov, A.~Vishik, V.~Voevodsky, {\it An exact sequence for $K_*^M/2$ with applications to quadratic forms}, Ann. of Math. (2), {\bf 165}(2007), no. 1, 1-13.

\bibitem{ParimalaHasse} R.~Parimala, {\it A Hasse principle for quadratic forms over function fields}, Bull. Amer. Math. Soc. (N.S.) {\bf 51} (2014), no. 3, 447-461.

\bibitem{Par-Izv} A.N.~Parshin, {\it Minimal models of curves of genus 2, and homomorphisms of abelian varieties defined over a field of finite characteristic}, Izv. Akad. Nauk SSSR Ser. Mat. {\bf 36} (1972), 67-109.

\bibitem{Zar-Par} A.N.~Parshin, Yu.G.~Zarhin, {\it Finiteness problems in Diophantine geometry}, English version of the Appendix to the Russian translation of S.~Lang's {\it Fundamentals of Diophantine geometry}, arXiv:0912.4325

\bibitem{Pierce} R.S.~Pierce, {\it Associative algebras}, GTM {\bf 88}, Springer, 1982.

\bibitem{PlatonovSA} V.P.~Platonov, {\it The problem of strong approximation and the Kneser-Tits hypothesis for algebraic groups}, Izv. Akad. Nauk SSSR Ser. Mat. {\bf 33} (1969), 1211-1219.

\bibitem{PlatonovTA} V.P.~Platonov, {\it On the Tannaka-Artin problem}, Dokl. Akad. Nauk SSSR {\bf 221} (1975), no. 5, 1038-1041.

\bibitem{PlatonovSteklov} V.P.~Platonov, {\it Reduced $K$-theory and approximation in algebraic groups}, Trudy Mat. Inst. Steklov. {\bf 142} (1976), 198-207, 270.

\bibitem{Pl-R} V.P.~Platonov, A.S.~Rapinchuk, {\it Algebraic Groups and Number Theory}, Academic Press, 1994.

\bibitem{PrasadSA} G.~Prasad, {\it Strong approximation for semi-simple groups over function fields},
Ann. of Math. (2) {\bf 105} (1977), no. 3, 553-572.

\bibitem{PrRap-WC} G.~Prasad, A.S.~Rapinchuk, {\it Weakly commensurable arithmetic
groups and isospectral locally symmetric spaces},
Publ. math. IHES {\bf 109}(2009), 113-184.

\bibitem{PrRap-CMH} G.~Prasad, A.S.~Rapinchuk, {\it Local-global principles for embedding of fields with involution into simple algebras with involution}, Comment. Math. Helv. {\bf 85} (2010), no. 3, 583-645.

\bibitem{PrRap-GeomDed} G.~Prasad, A.S.~Rapinchuk, {\it On the fields generated by the lengths of closed geodesics in locally symmetric spaces},
Geom. Dedicata {\bf 172} (2014), 79-120.

\bibitem{PrRap-Thin} G.~Prasad, A.S.~Rapinchuk, {\it Generic elements in Zariski-dense subgroups and isospectral locally symmetric spaces}, in {\it Thin groups and superstrong approximation}, 211-252,
Math. Sci. Res. Inst. Publ. {\bf 61}, Cambridge Univ. Press, Cambridge, 2014.

\bibitem{PrRap-Kunming} G.~Prasad, A.S.~Rapinchuk, {\it Weakly commensurable groups, with applications to differential geometry} in Handbook of group actions. Vol. I, 495-524,
Adv. Lect. Math. {\bf 31}, Int. Press, Somerville, MA, 2015.

\bibitem{Rag} M.S.~Raghunathan, {\it Discrete subgroups of Lie groups}
Ergebnisse der Mathematik und ihrer Grenzgebiete {\bf 68}, Springer, 1972.


\bibitem{RagRam} M.S.~Raghunathan, A.~Ramanathan, {\it Principal bundles on the affine line}, Proc. Indian Acad. Sci. (Math. Sci.) {\bf 93}(1984), nos. 2-3, 137-145.

\bibitem{Rap-BreuOh} A.S.~Rapinchuk, {\it Strong approximation for algebraic groups}, in {\it Thin groups and superstrong approximation}, 269-298,
Math. Sci. Res. Inst. Publ. {\bf 61}, Cambridge Univ. Press, Cambridge, 2014.

\bibitem{R-ICM} A.S.~Rapinchuk, {\it Towards the eigenvalue rigidity of Zariski-dense subgroups}, Proc. ICM-2014 (Seoul), vol. II, 247-269.

\bibitem{RR} A.S.~Rapinchuk,  I.A.~Rapinchuk, {\it On division algebras having the
same maximal subfields}, Manuscr. math. {\bf
132} (2010), 273-293.

\bibitem{RR-Tori} A.S.~Rapinchuk, I.A.~Rapinchuk, {\it Some finiteness results for algebraic groups and unramified cohomology over higher-dimensional fields}, arXiv:2002.01520.

\bibitem{IR} I.A.~Rapinchuk, {\it On linear representations of Chevalley groups over commutative rings},  Proc. Lond. Math. Soc. (3) {\bf 102}, no. 5, (2011), 951-983.

\bibitem{IR1} I.A.~Rapinchuk, {\it On abstract representations of the groups of rational points of algebraic groups and their deformations}, Algebra \& Number Theory {\bf 7} (7) (2013), 1685-1723.

\bibitem{IR-Char} I.A.~Rapinchuk, {\it On the character varieties of finitely generated groups}, Math. Res. Lett. {\bf 22} (2) (2015), 579-604.

\bibitem{IR3} I.A.~Rapinchuk, {\it Abstract homomorphisms of algebraic groups and applications}, in {\it Handbook of group actions}, Vol. II, 397-447, Adv. Lect. Math. (ALM) {\bf 32}, Int. Press, Somerville, MA, 2015


\bibitem{IR-Curves} I.A.~Rapinchuk, {\it On abstract homomorphisms of Chevalley groups over the coordinate rings of affine curves}, Transformation Groups {\bf 24} (2019), no. 4, 1241-1259.

\bibitem{IRap} I.A.~Rapinchuk, {\it A generalization of Serre's condition $(\mathrm{F})$ with applications to the finiteness
of unramified cohomology}, Math. Z. {\bf 291} (2019), no. 1-2, 199-213.

\bibitem{Reid} A.~Reid, {\it Isospectrality and commensurability of arithmetic hyperbolic 2- and 3-manifolds}, Duke Math. J. {\bf 65} (1992), no. 2, 215-228.

\bibitem{Salt} D.~Saltman, {\it Lectures on division algebras}, CBMS Regional Conference Series, vol. 94, AMS, 1999.

\bibitem{Sam} P.~Samuel, {Anneaux gradu\'es factoriels et modules r\'eflexifs}, Bull. Soc. Math. France {\bf 92} (1964), 237-249.


\bibitem{Serre-Ladic} J.-P.~Serre, {\it Abelian $\ell$-adic representations and elliptic curves}, McGill University lecture notes written with the collaboration of Willem Kuyk and John Labute, W. A. Benjamin, Inc., 1968.

\bibitem{Serre-Course} J.-P.~Serre, {\it A course in arithmetic}, GTM {\bf 7}, Springer, 1973.

\bibitem{Serre-GC} J.-P.~Serre, {\it Galois cohomology}, Springer, 1997.

\bibitem{Serre-LocAlg} J.-P.~Serre, {\it Local Algebra}, Springer, 2000.

\bibitem{SerreCongruence} J.-P.~Serre, {\it Le probl\`eme des groupes de congruence pour $SL_2$}, Ann. of Math. (2) {\bf 92} (1970), 489-527.

\bibitem{Shafarevich} I.R.~Shafarevich, {\it Algebraic number fields}, Proc. ICM-1962 (Stockholm), 163-176.

\bibitem{Shatz} S.~Shatz, {\it, Profinite groups, arithmetic, and geometry}, Annals of Mathematics Studies, No. 67, Princeton University Press, 1972.

\bibitem{Sh} D.~Shenfeld, {\it On semisimple representations of universal lattices}, Groups Geom. Dyn.  {\bf 4}, no. 1 (2010), 179–193.


\bibitem{Silverman} J.~Silverman, {\it The arithmetic of elliptic curves. Second edition}, GTM {\bf 106}, Springer, 2009.

\bibitem{SilTate} J.~Silverman, J.~Tate, {\it Rational points on elliptic curves. Second edition}, Undergraduate Texts in Mathematics, Springer, 2015.

\bibitem{SpringerAG} T.A.~Springer, {\it Linear algebraic groups. Second edition}, Progress in Mathematics {\bf 9}, Birkh\"{a}user Boston, Inc., Boston, MA, 1998.

\bibitem{Srimathy} S.~Srinivasan, {\it Good Reduction of Unitary Groups of Quaternionic Skew-Hermitian Forms}, arXiv:1906.01414

\bibitem{Stavrova} A.~Stavrova, {\it Chevalley groups of polynomial rings over Dedekind domains}, J. Group Theory {\bf 23} (2020), no. 1, 121-132.

\bibitem{St2} R.~Steinberg, {\it Some consequence of elementary relations in $SL_n$}, Finite Groups --- Coming of Age, Proceedings of the Canadian Mathematical Society Conference held on June 15-28, 1982, AMS, Contemporary Mathematics Vol. 45, 1985

\bibitem{Suslin} A.A.~Suslin, {\it On the structure of the special linear group over rings of polynomials}, Izv. Akad. Nauk SSSR Ser. Mat. {\bf 41} (1977), 235-252; English transl. Math.-USSR Izv. {\bf 11} (1977), 221-238.

\bibitem{SzamuelyFG} T.~Szamuely, {\it Galois groups and fundamental groups}, Cambridge Studies in Advanced Mathematics {\bf 117}, Cambridge University Press, Cambridge, 2009.

\bibitem{Szpiro1} L.~Szpiro, {\it Sur le th\'eor\`eme de rigidit\'e de Parshin et Arakelov},  in {\it Journ\'ees de G\'eom\'etrie Alg\'ebrique de Rennes} (Rennes, 1978), Vol. II, pp. 169-202, Ast\'erisque {\bf 64}, Soc. Math. France, Paris, 1979.

\bibitem{Szpiro2}  L.~Szpiro, {\it Propri\'et\'es num\'eriques du faisceau dualisant relatif}, in {\it Seminar on Pencils of Curves of Genus at Least Two},  Ast\'erisque {\bf 86} (1981), 44-78.

\bibitem{Breu-Oh} {\it Thin groups and superstrong approximation}, Selected expanded papers from the workshop held in Berkeley, CA, February 6-10, 2012, edited by E.~Breuillard and H.~Oh, Mathematical Sciences Research Institute Publications {\bf 61}, Cambridge University Press, Cambridge, 2014.

\bibitem{Tikh} S.V.~Tikhonov, {\it Division algebras of prime degree with infinite genus}, Proc. Steklov Inst. {\bf 292} (2016), 256-259.

\bibitem{TitsClassification} J.~Tits, {\it Classification of algebraic semisimple groups}, in  {\it Algebraic Groups and Discontinuous Subgroups}, Proc. Sympos. Pure Math., Boulder, Colo., 1965, pp. 33-62, Amer. Math. Soc., Providence, R.I., 1966.

\bibitem{TitsFree} J.~Tits, {\it Free subgroups in linear groups}, J. Algebra {\bf 20} (1972), 250-270.

\bibitem{TitsWhitehead} J.~Tits, {\it Groupes de Whitehead de groupes alg\'ebriques simples sur un corps (d'apr\`es V. P. Platonov et al.)}, S\'eminaire Bourbaki (1976/77), Exp. No. 505, pp. 218-236, Lecture Notes in Math. {\bf 677}, Springer, 1978.

\bibitem{Vinberg} E.B.~Vinberg, {\it Rings of definition of dense subgroups of semisimple linear groups},
Izv. Akad. Nauk SSSR Ser. Mat. {\bf 35} (1971), 45-55.

\bibitem{Vish1} A.~Vishik, {\it Integral motives of quadrics}, preprint MPI-1998-13, Max Planck Institute
f\"ur Mathematik, Bonn 1998, 82 pp., http://www.mpim-bonn.mpg.de/node/263.

\bibitem{Vish2} A.~Vishik, {\it Motives of quadrics with applications to the theory of quadratic forms}, Geometric methods in the algebraic theory of
quadratic forms, Lecture Notes in Math., vol. 1835, Springer, Berlin 2004, pp. 25-101.


\bibitem{Voloch} J.F.~Voloch, {\it Diophantine geometry in characteristic $p$: a survey}, Arithmetic geometry (Cortona, 1994), 260-278,
Sympos. Math., XXXVII, Cambridge Univ. Press, Cambridge, 1997.


\bibitem{Voskr} V.E. Voskresenskii, {\it Algebraic groups and their birational invariants}, Translations of Mathematical Monographs {\bf 179}, American Mathematical Society, Providence, RI, 1998.

\bibitem{Waterhouse} W.~Waterhouse, {\it Introduction to affine group schemes}, GTM {\bf 66}, Springer, 1979.

\bibitem{Weisfeiler} B.~Weisfeiler, {\it Semisimple algebraic groups that are decomposable over a quadratic extension}, Izv. Akad. Nauk SSSR Ser. Mat. {\bf 35} (1971), 56-71.

\bibitem{Yamasaki} A.~Yamasaki, {\it Strong approximation theorem for division algebras over $\R(X)$}, J. Math. Soc. Japan {\bf 49} (1997), no. 3, 455-467.

\bibitem{Zarhin} Yu.G.~Zarhin, {\it Torsion of abelian varieties in finite characteristic}, Mat. Zametki {\bf 22} (1977), no. 1, 3-11.

\end{thebibliography}

\end{document}